\documentclass[a4paper,11pt, reqno]{amsart}
\usepackage[utf8]{inputenc}
\usepackage{frenchineq}
\usepackage{amsmath}  
\usepackage{amssymb}     
\usepackage{amsthm}  
\usepackage{bbm}   
\usepackage{setspace}
  
\usepackage{bm}
\usepackage{hyperref}
\usepackage{accents}
\usepackage{mathrsfs}
\usepackage{mathtools} 
\usepackage{fixmath}
\usepackage{fullpage}
\usepackage{paralist}
\usepackage{cleveref}
\usepackage[ruled, vlined, norelsize, linesnumbered]{algorithm2e}
\SetCommentSty{normalfont}

\usepackage{pgf,tikz}
\usetikzlibrary{arrows}
\usetikzlibrary{patterns}
\usetikzlibrary{positioning}
\usetikzlibrary{shapes.geometric}
\usetikzlibrary{fit}
\usetikzlibrary{calc}

\usepackage{environ}
\NewEnviron{Ralign}{\tagsright@true\begin{align}\BODY\end{align}}

\newcommand{\fitbox}[3]{
\path (#1 tl) ++ (-0.1,0.3) coordinate (#1 tl2);
\path (#1 br) ++ (0.1,-0.15) coordinate (#1 br2);
\node[draw=none,fit={(#1 tl2) (#1 br2)},inner sep=0cm] (#1) {};
\node[fill=none,anchor=south east,font=\tiny,text opacity=1,inner sep=0.08cm] (#1 label) at ({#1}.north east) {#3}; 
\filldraw[rounded corners,draw=#2,fill=#2,fill opacity=0.1] ({#1}.south west) -- ({#1}.south east) -- ({#1 label}.north east) -- ({#1 label}.north west) -- ({#1 label}.south west) -- ({#1}.north west) -- cycle;
\draw[draw=#2,dashed]  ({#1 label}.south east) -- ({#1 label}.south west);
\coordinate (#1 left label) at ({#1}.180); 
\coordinate (#1 right label) at ({#1}.0); 
}

\newtheorem{theorem}{Theorem}
\newtheorem{proposition}[theorem]{Proposition}
\newtheorem{prop_def}[theorem]{Proposition-Definition}
\newtheorem{corollary}[theorem]{Corollary}
\newtheorem{lemma}[theorem]{Lemma}

\theoremstyle{definition}

\newtheorem{assumption}{Assumption}

\theoremstyle{remark}
\newtheorem{remark}[theorem]{Remark}
\newtheorem{example}[theorem]{Example}

\DeclareSymbolFont{stmry}{U}{stmry}{m}{n}
\DeclareMathSymbol\mapsfromchar\mathrel{stmry}{"5B}

\newcommand{\ie}{\textit{i.e.},} 
\newcommand{\eg}{\textit{e.g.},} 
\newcommand{\etc}{\textit{etc}}
\newcommand{\arxiv}[1]{arXiv:#1}

\newcommand{\mathref}[1]{\text{\ref{#1}}}

\DeclareMathOperator*{\argmax}{\arg\max}

\DeclareMathOperator*{\sign}{sign}
\DeclareMathOperator*{\tsign}{tsign}
\DeclareMathOperator*{\tdet}{tdet}
\DeclareMathOperator*{\tper}{tper}

\newcommand{\bal}{\mathrel{\nabla}}

\DeclareMathOperator*{\tconv}{tconv}
\DeclareMathOperator*{\tcone}{tpos}
\DeclareMathOperator*{\tpos}{\tcone}

\DeclareMathOperator*{\conv}{conv}
\DeclareMathOperator*{\cone}{pos}

\DeclareMathOperator*{\Sym}{Sym}

\DeclareMathAlphabet\mathbfcal{OMS}{cmsy}{b}{n}

\DeclareMathAlphabet{\mathbbold}{U}{bbold}{m}{n}

\newcommand{\R}[0]{\mathbb{R}}

\newcommand{\trop}[0]{\mathbb{T}}    
\newcommand{\strop}[0]{\trop_{\!\pm}}    
\newcommand{\symtrop}[0]{\mathbb{S}}    
\newcommand{\tropProj}[0]{\trop \mathbb{P}}

\newcommand{\transpose}[1]{#1^{\top}}

\newcommand{\tplus}{\oplus}          
\newcommand{\tsum}{\bigoplus}        
\newcommand{\ttimes}{\odot}       
\newcommand{\tdot}{\odot} 	
\newcommand{\tscale}{\ttimes}               
\newcommand{\tminus}{\ominus}        

\newcommand{\splus}{\oplus}
\newcommand{\ssum}{\bigoplus}
\newcommand{\stimes}{\odot}
\newcommand{\sdot}{\odot}

\newcommand{\zero}{\mathbbold{0}}    
\newcommand{\unit}{\mathbbold{1}}    
\newcommand{\tropP}{\mathcal{P}}  
\newcommand{\tropC}{\mathcal{C}}  
\newcommand{\tropD}{\mathcal{D}}  
\newcommand{\tropE}{\mathcal{E}}  
  
\newcommand{\tropH}{\mathcal{H}}

\newcommand{\puiseux}[0]{\mathbb{K}} 

\DeclareMathOperator*{\val}{val}
\DeclareMathOperator*{\lc}{lc}
\DeclareMathOperator*{\sval}{sval}

\newcommand{\puiseuxP}[0]{\bm{\mathcal{P}}}   
\newcommand{\puiseuxH}[0]{\bm{\mathcal{H}}}             
\newcommand{\puiseuxC}[0]{\bm{\mathcal{C}}}   
   
\newcommand{\puiseuxE}[0]{\bm{\mathcal{E}}}   
\newcommand{\x}[0]{\bm{x}}
\newcommand{\y}[0]{\bm{y}}
\newcommand{\s}[0]{\bm{s}}

\newcommand{\cc}[0]{\bm{c}}
\newcommand{\A}[0]{\bm{A}}
\renewcommand{\P}[0]{\bm{P}}
\newcommand{\RR}[0]{\bm{R}}

\renewcommand{\b}[0]{\bm{b}}
\renewcommand{\a}[0]{\bm{a} }

\newcommand{\plambda}[0]{\bm{\lambda} }
\newcommand{\pmu}[0]{\bm{\mu} }
\renewcommand{\d}[0]{\bm{d}}
\newcommand{\M}[0]{\bm{M}}
\renewcommand{\H}[0]{\bm{H}}

\newcommand{\ep}{w}  
\newcommand{\epm}{W} 
\newcommand{\epr}{W} 

\newcommand{\pr}{A} 

\newcommand{\Mrow}{M} 

\newcommand{\pep}{\bm{w}} 
\newcommand{\pepm}{\bm{W}}
\newcommand{\pepr}{\bm{W}}

\newcommand{\ppr}{\A}

\newcommand{\graph}[0]{\mathcal{G}}
\newcommand{\digraph}[0]{\vec{\mathcal{G}}}

\newcommand{\breakHyp}[0]{\mathsf{Br}}
\newcommand{\enteringHyp}[0]{\mathsf{Ent}}

\newcommand{\Jdiff}[0]{\Delta}

\newcommand{\basis}[0]{I}

\newcommand{\ileaving}[0]{i_{\mathsf{out}}}
\newcommand{\ient}[0]{i_{\mathsf{ent}}}

\newcommand{\ibreak}[0]{k}

\newcommand{\aleaving}[0]{a_{\mathsf{old}}}

\newcommand{\leaving}[0]{\mathsf{old}}
\newcommand{\ent}[0]{\mathsf{new}}

\DeclarePairedDelimiter{\oiv}{]}{[}

\title{Tropicalizing the Simplex Algorithm}

\author{Xavier {A}llamigeon}
\address[X.~Allamigeon, P.~Benchimol and S.~Gaubert]{INRIA and CMAP, \'Ecole Polytechnique, 91128 Palaiseau Cedex France}
\email[X.~Allamigeon]{xavier.allamigeon@inria.fr}
\author{Pascal Benchimol}
\email[P.~Benchimol]{pascal.benchimol@polytechnique.edu}
\author{St{\'e}phane {G}aubert}
\email[S.~Gaubert]{stephane.gaubert@inria.fr}
\author{Michael Joswig}
\address[M.~Joswig]{Institut f{\"u}r Mathematik,
TU Berlin\\
Str.\ des 17. Juni 136\\
10623 Berlin, Germany}
\email{joswig@math.tu-berlin.de}

\thanks{P.~Benchimol is supported by a PhD fellowship of DGA and \'Ecole Polytechnique. X.~Allamigeon and S.~Gaubert are
  partially supported by the PGMO program of EDF and Fondation Math\'ematique Jacques Hadamard.  M.~Joswig is supported
  by Einstein Foundation Berlin and the German Research Foundation (DFG)}
\keywords{tropical geometry, linear programming, simplex method}
\subjclass[2010]{14T05, 90C05}
\begin{document}
\begin{abstract}
  We develop a tropical analog of the simplex algorithm for linear programming.  In particular, we obtain a combinatorial algorithm to perform one tropical
  pivoting step,
 including the computation of reduced costs, 
in $O(n(m+n))$ time, where $m$ is the number of constraints and $n$ is the dimension.
\end{abstract}

\maketitle 

\section{Introduction}

\noindent
The tropical semiring $(\trop, \tplus, \ttimes)$ is the set $\trop=\R \cup \{ - \infty \}$ endowed with the two operations $a \tplus b = \max(a,b)$
and $a \ttimes b = a + b$.  We are interested in the tropical equivalent of linear programming. In other words, our goal is to give an algorithm for
minimizing a tropical linear form $\max(c_1 + x_1, \dots, c_n + x_n)$ over a tropical polyhedron. The latter is the set of solutions $x\in \trop^n$ of
finitely many inequalities of the form
\[
\max(a_1+x_1, \dots,a_n+x_n, a_{n+1}) \geq \max(b_1+x_1, \dots,b_n+x_n, b_{n+1}) \enspace.
\]
All the coefficients $a_j,b_j,c_j$ are elements of $\trop$. An example is depicted in Figure~\ref{fig:intro_example} below.

Several avenues lead to this research.  First, the classical simplex method belongs to the most relevant algorithms, both for its applicability as
well as its theoretical implications.  So it is natural to explore variants and derivations, including tropical ones.  In the form that we are
studying this leads to a class of minmax problems which are also interesting from a purely complexity-theoretic point of view.  In \cite{AGG} it is
shown that a tropical analog of the feasibility problem in linear optimization is polynomial-time equivalent to deciding which player has a winning
strategy in a mean-payoff game.  The latter decision problem is among the few problems in NP as well as co-NP, see Zwick and Paterson~\cite{zwick},
for which no polynomial-time algorithm is known.  This game-theoretic perspective leads to a second approach to tropical linear programming.  A third
train of thought is more geometrical.  Viro suggested to investigate the tropical aspects of real algebraic geometry already in~\cite{Viro2000}.
Nonetheless, the main focus of tropical geometry so far concerns the tropicalization of algebraic varieties which are defined over the complex numbers
(or Puiseux series with complex coefficients). More recently, however, the tropicalization of real semi-algebraic sets has been studied by
Alessandrini \cite{Alessandrini13}.  In this vein our work seeks to contribute to understanding the tropicalizations of the most simple semi-algebraic
sets: convex polyhedra.  A related motivation arises from linear programming over ordered fields, the complexity of which is a well known open
question~\cite[Section~2]{megiddo}. Ordered fields arise naturally when dealing with perturbations of classical linear
programs~\cite{Jeroslow1973,FilarAltmanAvrachenkov2002}.

Tropical polyhedra or tropically convex sets have appeared in different guises in the works of several authors,
including~\cite{zimmermann77,CG,litvinov00,cgq02,BriecHorvath04}; the present work is specially motivated by the approach of Develin and
Sturmfels~\cite{develin2004}, in which tropical polyhedra are studied by combinatorial means, and by the further work of Develin and
Yu~\cite{DevelinYu07}, who showed that tropical polyhedra are precisely the images by the valuation of (convex) polyhedra over the field of Puiseux
series.

This in mind, the most natural approach for tropical linear programming probably is to do linear programming over real Puiseux series and to
tropicalize, \ie\ to devise a method which traces the valuation of the path followed by the simplex algorithm over real Puiseux series.  This is
exactly what we do here.  What makes our algorithm interesting is that the method itself does not manipulate Puiseux series (explicit lifts to real
Puiseux series are not needed).  Instead it directly processes the tropical input and ``stays tropical'' throughout the computations. In this way, the
arithmetical operations remain elementary.

\begin{figure}[t]\centering
  
  \begin{tikzpicture} [ line cap=round,line join=round,>=triangle 45,x=1.0cm,y=1.0cm ]
    \begin{scope}
      \clip(0,0) rectangle (10,8);

      \draw[color = orange] (5,3)-- (-1,3); \draw[color = orange]
      (5,3)-- (5,-2); \filldraw[draw=none,pattern=north east
      lines,pattern color=orange,fill opacity=0.5] (-1,3) -- (5,3) --
      (5,-1) -- (4.5, -1) -- (4.5, 2.5) -- (-1,2.5) -- cycle;

      \draw[color={rgb:red,10;green,50;blue,10}] (-1,2)-- (7,2);
      \draw[color={rgb:red,10;green,50;blue,10}] (7,2)-- (7,-1);
      \filldraw[draw=none,pattern=north west lines,pattern
      color={rgb:red,10;green,50;blue,10},fill opacity=0.5] (-1,2) --
      (7,2) -- (7,-1) -- (6.5, -1) -- (6.5, 1.5) -- (-1,1.5) -- cycle;

      \draw[color={rgb:red,50;green,0;blue,50}] (7,5)-- (7,-2);
      \draw[color={rgb:red,50;green,0;blue,50}] (7,5)-- (12,10);
      \filldraw[draw=none,pattern=north west lines,pattern
      color={rgb:red,50;green,0;blue,50},fill opacity=0.5] (7,-2) --
      (7,5) -- (12,10) -- (12, 9.5) -- (7.5, 5) -- (7.5,-2) -- cycle;

      \draw[color={rgb:red,50;green,50;blue,0}] (2,6)-- (2,-2);
      \draw[color={rgb:red,50;green,50;blue,0}] (2,6)-- (6,10);
      \filldraw[draw=none,pattern=north west lines,pattern
      color={rgb:red,50;green,50;blue,0},fill opacity=0.5] (2,-2) --
      (2,6) -- (6,10) -- (5.5, 10) -- (1.5, 6) -- (1.5,-2) -- cycle;

      \fill[fill=lightgray, fill opacity = 0.7] (2,3) -- (5,3) --
      (5,2) -- (7,2) -- (7,5) -- (12,10) -- (6,10) -- (2,6) -- cycle;

      \draw[ultra thick] (7,2) -- (7, -1);

      \draw[dashdotted, thick] (0,0) -- (10,10); \draw[color = blue,
      thick] (-1,6) -- (6,6) -- (6,-1); \draw[color = blue, thick]
      (-1,4) -- (4,4) -- (4,-1); \draw[color = blue, thick] (-1,3) --
      (2,3); \draw[color = blue, thick] (3,3) -- (3,-1); \draw[color =
      red, ultra thick] (2,3) -- (3,3);

    \end{scope}

 \colorlet{mygray}{black!80!}
       \draw[->, mygray] (0,0) -- (10,0);
       \draw[->, mygray] (0,0) -- (0,8);

       \node[ anchor = north, mygray] at (10,0) {${\scriptstyle x_1}$};
       \node[ anchor = east, mygray] at (0,8) {${\scriptstyle x_2}$};
     \foreach \x in {0,1,...,9} { \node [anchor=north, mygray] at (\x,0) {${\scriptstyle\x}$}; }
     \foreach \y in {0,1,...,7} { \node [anchor=east, mygray] at  (0,\y) {${\scriptstyle\y}$}; }
     \foreach \x in {0,1,...,9} { \draw[mygray] (\x,0.1) -- (\x,-0.1); }
     \foreach \y in {0,1,...,7} { \draw[mygray] (0.1,\y) -- (-0.1,\y); }

  \end{tikzpicture}
  
  \caption{A tropical linear program. The feasible set is the union of the gray shaded area with the thick black halfline. Three level sets for the objective function $\max(x_1, x_2)$ are depicted in blue. The thick red segment is the set of optima.}
\label{fig:intro_example}
\end{figure}
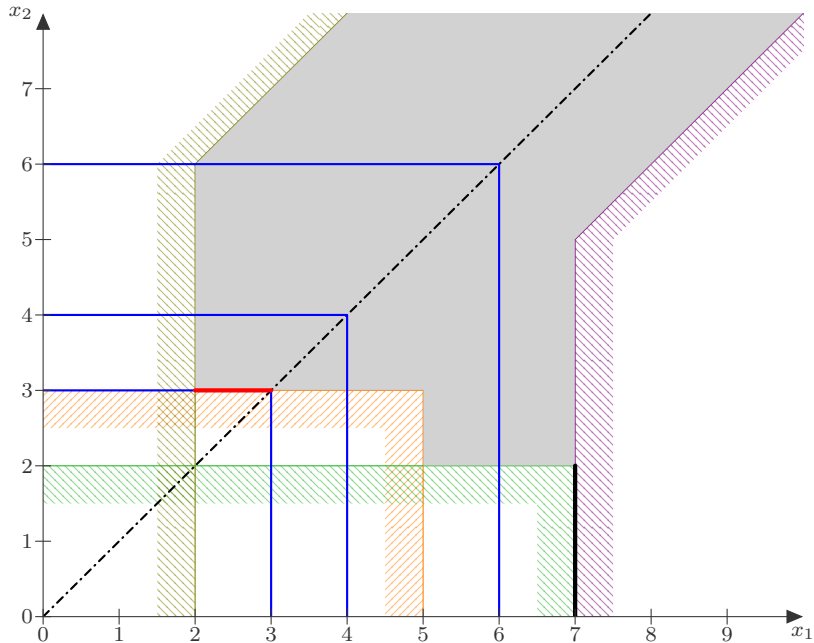

In order to make our ideas more apparent, and to avoid technical details which are too cumbersome to attack in a direct fashion, in the present paper
we assume that our tropical linear program is primally and dually non-degenerate.  Further, we assume that each point in the feasible region has
finite coordinates only.  Any tropical linear program satisfying these properties will be called \emph{standard}; see Assumptions
\ref{ass:general_position}, \ref{ass:finite_coordinates} and~\ref{ass:dual_general_position} below.  We defer all ramifications which come from
looking at degenerate input or infinite coefficients to a subsequent second paper.  Our main result is the following theorem.

\begin{theorem}\label{th-main}
  Consider a standard tropical linear program with $n$ variables and $m$ inequalities.  Then, the tropical simplex algorithm
  (Algorithm~\ref{alg:main}) terminates and returns an optimal solution for any tropical pivoting rule.  Every iteration (pivoting and computing
  reduced costs) can be done in time $O(n(m+n))$.  Moreover, the algorithm traces the image by the valuation map of the path followed by the classical
  simplex algorithm applied to any lift of this program to the field of real Puiseux series, with a compatible pivoting rule.
\end{theorem}
In particular, under the assumptions of Theorem~\ref{th-main}, linear programs over Puiseux series are implicitly solved by the tropical simplex
algorithm. By definition, a tropical pivoting rule selects a variable of tropically negative reduced cost. A classical pivoting rule is said to be
compatible with the former tropical pivoting rule if they select the same variables.  Tropical pivot rules are the topic of Section~\ref{sec:pivot}.

Our tropical simplex algorithm relies on several tools of independent interest.  For instance,
Corollary~\ref{coro:val_and_inter_commute_for_generic_matrices} shows that, again under the general position assumption, the cells of an
arrangement of hyperplanes over the field of real Puiseux series are in one to one correspondence with the cells of the arrangement of the
associated tropical hyperplanes.  This leads to the notion of \emph{tropical basic points} and \emph{tropical edges} of a system of tropical
affine inequalities.  Unlike the classical case, a tropical basic point may not be tropically extreme, see
Proposition~\ref{prop:extreme_basic} and Remark~\ref{rk:therearefewerextremepoints} below.  This stems from the lack of a good notion for a
general ``face'' of a tropical polyhedron; see the discussions in \cite{Joswig05,DevelinYu07} This is related to competing notions of
rank~\cite{DevelinSantosSturmfels05,AGG08b}.

\begin{algorithm}[t]\caption{Phase II tropical simplex algorithm\label{alg:main}}
\SetArgSty{text}
\DontPrintSemicolon
\footnotesize
\KwIn {A matrix $A\in\strop^{m\times n}$, a column vector $b\in\strop^m$, an unsigned row vector $c\in\trop^n$.
  A  tropical basic point $x^\basis$ of  $\tropP(A,b)$, and the corresponding set $I \subset [m]$.}
\KwOut {A tropical basic point of $\tropP(A,b)$ that is minimal with respect to $c$.}
compute the tropical reduced costs $y$ associated with $\basis$ \;
\While{$y$ has a tropically negative entry}{
  choose $ \ileaving\in \basis$ such that $y_{\ileaving}$ is tropically negative \;
  $K \leftarrow \basis\setminus\{\ileaving\}$ \;
  pivot along the tropical edge $\tropE_K$ to the tropical basic point $x^{\basis'}$ for a set of the form $\basis' = K \cup \{ \ient\}$ \;
  $\basis\leftarrow \basis'$ \;
  compute the tropical reduced costs $y$ associated with $\basis$ \;
}
\Return $x^{\basis}$ \;
\end{algorithm}

A fundamental discrepancy to the classical simplex algorithm is that a tropical edge in $\trop^n$ may have a more complex geometrical structure as,
indeed, it consists of up to $n$ ordinary segments. These segments can be determined from \emph{tangent digraphs}, which encode a local description of
a tropical polyhedron; tangent digraphs were initially introduced in the form of directed hypergraphs in~\cite{AllamigeonGaubertGoubaultDCG2013}. The
cornerstone of the tropical simplex algorithm is a new combinatorial characterization of the tangent digraph
(Proposition~\ref{prop:tangent_graph_interior_edge}) at a point inside a tropical edge. In particular, this entails an incremental computation of
tangent digraphs from one ordinary segment to another (Proposition~\ref{prop:tangent_graph_update}), leading to an $O(n(m + n))$ time method for one
full pivoting step, see Theorem~\ref{th:pivot}. Finally, we define the tropical reduced cost vector, which allows one to certify the optimality of a
given basic point. We show that the vector of reduced costs can be computed by solving a system of signed tropical linear equations, and that the
running time of this step is also bounded by $O(n(m+n))$, see Theorem~\ref{thm:reduced_costs}.

Let us finally point out some related work. The study of the analogs of linear programs over ordered semirings was undertaken in the
book~\cite{Zimmermann.U}, in particular, a duality theorem for a special class of linear programs can be found there.  The idea of looking for analogs
of convex programming results over ``extremal'' (a variant of tropical) structures is also apparent in~\cite{Zimmermann.K}.  Several recent works have
proposed algorithms to solve various tropical programming problems. In~\cite{BA-08}, a dichotomy algorithm is developed, allowing one to solve
tropical linear programming problems by a reduction to linear feasibility problems.  In~\cite{GKS11}, more general linear-fractional programming
problems are studied, in which one maximizes the difference of two tropical linear forms.  A policy iteration algorithm based on a parametric mean
payoff game is given there.  These policies seem to have interesting connections with basic points.  However, our present approach leads to a
fundamentally different method: we move along edges in the graph of the tropical polytope, whereas policy iteration type algorithms often take ``great
leaps'' in the same graph; also, one iteration of the present algorithm takes only $O(n(m+n))$ time whereas every iteration in~\cite{GKS11} requires
to solve a mean payoff game.  Yet another different class of algorithms for solving tropical linear feasibility problems relies on cyclic
projection~\cite{CGB-03,gauser,AGNS10}.  Recall also that the tropical linear feasibility problem is equivalent to mean payoff games, for which a
number of algorithms are available, like pumping~\cite{GKK-88}, value iteration~\cite{zwick,AGG}, or policy
iteration~\cite{cras,DG-06,bjorklund,chaloupka}. A reduction of mean payoff games to classical linear programs with exponentially large coefficients
is established in~\cite{schewe2009}.  The asymptotic simplex method developed in~\cite{FilarAltmanAvrachenkov2002} solves arbitrary linear programs on
Laurent series (which is sufficient for tropical linear programs with rational coefficients). Each iteration of their method requires $O(s(m+n)^2)$
operations, where $s \leq (m+n)$ is the maximum taken over the valuations of all Puiseux series arising during the computation.  Our tropical simplex
algorithm shows a better complexity per iteration since the factor of $s$ is dispensed with, but the approach of \cite{FilarAltmanAvrachenkov2002}
does not require any genericity assumptions.

This paper is organized as follows.  Section~\ref{sec-prelim} describes our notations and collects the relevant known facts about convex polyhedra
over real Puiseux series.  For the reader's convenience we introduce a running example which we refer to throughout this paper.  In
Section~\ref{sec:tropical_basic_points} we characterize the key players in our algorithms: tropical basic points and tropical edges.  The core of our
paper is Section~\ref{sec:pivot}, where we describe the tropical pivot.  Section~\ref{sec:reduced_costs} discusses tropical reduced costs.  Finally,
Theorem~\ref{th-main} is proved in Section~\ref{sec:generalization}.

Our algorithm for solving tropical linear programs is outlined in Algorithm~\ref{alg:main}.  It directly corresponds to Phase~II of the classical
simplex method over real Puiseux series.  Phase~II starts from a given tropical basic point and proceeds along improving edges towards an optimal
tropical basic point.  The Phase~I problem, to find a first tropical basic point will be addressed in a sequel to this work.  While classically Phase
I can be reduced to Phase II, in general, this requires to solve a degenerate linear program.  As explained above, this is out of the scope of the
present paper. 

\section{Preliminaries}\label{sec-prelim}

\subsection{Tropical arithmetic}
The domain for our computations is the set $\trop=\R \cup \{ - \infty \}$.  The neutral elements for the tropical ``addition'' and ``multiplication''
are $ \zero := - \infty$ and $\unit := 0$, respectively.  The usual definition of matrix operations carries over to tropical matrices. Given two
matrices $A = (a_{ij})$ and $B = (b_{ij})$, we denote by $A \tplus B$ and $A \tdot B$ the matrices with entries $a_{ij} \tplus b_{ij}$ and $\tsum_k
a_{ik} \ttimes b_{kj}$, 
respectively. We also denote by $\transpose{A}$ the transpose of the matrix $A$, by $\pr_i$ the $i$th row of $A$, and by $\pr_I$ the submatrix of $A$
formed from the rows $i \in I$. For the sake of simplicity, we identify vectors of size $n$ with $n{\times}1$-matrices.  Given $a = (a_{1j}) \in \trop^{1
  \times n}$ and $x \in \trop^n$, we denote by $\argmax(a \tdot x)$ the set of indices $i \in [n]=\{1,\dots,n\}$ attaining the maximum in
\[
a \tdot x = \max_{j \in [n]}(a_{1j} + x_j) \ .
\]
The usual total order $\leq$ on $\R$ extends to $\trop$. This induces a partial ordering of tropical vectors by
entry-wise comparisons.  The topology induced by the order makes $(\trop, \tplus, \ttimes)$ a topological semiring.

In the following, we will think of the $n$-fold product space $\trop^n$ as a semimodule over $\trop$, where scalars act tropically on vectors by
$( \lambda, x) \mapsto \lambda \tscale x := (\lambda + x_1, \dots, \lambda + x_n)$ and the tropical vector addition is $(x,y) \mapsto x \tplus y := (\max(x_1, y_1), \dots, \max(x_n, y_n) )$.

\subsubsection{Signed tropical numbers}\label{subsect-signed}
It will be convenient to use the set of \emph{signed tropical numbers}~\cite{akian1990linear}, denoted here by $\strop$. The latter set consists of two copies of $\trop$, called the set of \emph{positive tropical numbers} and the set of \emph{negative
  tropical numbers}, respectively. These two copies are glued by identifying the element $\zero$. Positive and negative tropical numbers are written as $a$ and $\tminus a$,
respectively, for some $a \in \trop$. By definition, the numbers $a$ and $\tminus a$ are different, unless $a = \zero$.
Their  \emph{sign} is $\sign(a)=1$ and $\sign(\tminus a) =-1$ when $a$ is not $\zero$ and $\sign(\zero)=0$. 
The \emph{modulus} of $x \in \{a, \tminus a \}$ is defined as $|x| :=a$. The multiplication  $x \stimes y$ of two elements $x,y \in \strop$ yields the element whose modulus is $|x|+|y|$ and whose sign is the product $\sign(x) \sign(y)$.
 The \emph{positive part} and the \emph{negative part} of an element $x \in \strop$ are the tropical numbers $x^+$ and $x^-$ defined by:
\begin{align*}
x^+ = \left \{
\begin{array}{cr}
|x | & \text{ if } x \text{ is positive} \\
\zero & \text{ otherwise}
\end{array}
\right . 
\hspace{1cm}
x^- = \left \{
\begin{array}{cr}
\zero & \text{ if } x \text{ is positive} \\
|x | & \text{ otherwise}
\end{array}
\right .
\end{align*}
Modulus, positive part and negative part extend to matrices entry-wise.  
It was shown in~\cite{akian1990linear} that signed tropical numbers
can be embedded in a semiring, called the {\em symmetrized tropical
semiring}. Indeed, the sum of two signed tropical numbers 
with opposite signs but identical modulus cannot be defined
as a signed tropical number, one needs to enlarge $\strop$
with a third type of elements, called {\em balanced elements},
to represent such sums. 
We will defer the discussion of the symmetrized
tropical semiring until Section~\ref{sec:symmetrized},
since the additional technicalities can be spared in the first three quarters of this paper. In particular, the addition of signed tropical numbers
will not be used before Section~\ref{sec:symmetrized}.

\subsubsection{General position}

The \emph{permanent} of the square matrix $M = (m_{ij}) \in \trop^{ n \times n}$ is given by
\begin{equation}
\tper(M) := \tsum_{\sigma \in \Sym(n)} m_{1 \sigma(1)} \ttimes \cdots \ttimes m_{n \sigma(n)} = \max_{\sigma \in \Sym(n)} m_{1 \sigma(1)} + \dots + m_{n \sigma(n)} 
\ ,
\label{eq:tper}
\end{equation}
where $\Sym(n)$ is the set of all permutations of $[n]$. Computing the tropical permanent amounts to finding a permutation which attains the maximum in
(\ref{eq:tper}). Such a permutation is a solution of the assignment problem with costs $(m_{ij})$. It can found in time $O(n^3)$ using the Hungarian method; see
\cite[\S17.3]{Schrijver03:CO_A}. A square matrix is said to be \emph{tropically singular} if $\tper(M)= \zero$ or if the maximum is attained at least twice in
(\ref{eq:tper}).

A slightly more restrictive notion of singularity arises when signs are taken into account. A signed matrix $M \in \strop^{n \times n}$ is \emph{tropically sign
  singular} if $\tper(|M|) = \zero$ or if the maximum in $\tper(|M|)$ is attained on two distinct permutations $\sigma$ and $\pi$ such that the terms
$\tsign(\sigma) \stimes m_{1\sigma(1)} \stimes \cdots \stimes m_{n \sigma(n)}$ and $\tsign(\pi) \stimes m_{1\pi(1)} \stimes \cdots \stimes m_{n \pi(n)}$ have
opposite tropical signs, where $\tsign(\sigma) = \unit$ if $\sigma$ is an even permutation and $\tsign(\sigma) = \tminus \unit$ otherwise.  The notion
of tropical sign singularity of a matrix appeared
in different forms in~\cite{gondran84}, \cite{akian1990linear}, and ~\cite[\S4]{Joswig05}.

We call a rectangular matrix $W \in \strop^{m \times n}$ \emph{tropically generic} if for every square submatrix $U$ of $W$ either $\tper(|U|) =
\zero$ or $|U|$ is not tropically singular.  Similarly, the matrix $W$ is \emph{tropically sign generic} if $\tper(|U|) = \zero$ or $U$ is not
tropically sign singular, again for all square submatrices $U$.

\begin{example}
Consider the following matrix with signed tropical entries.
  \begin{equation*} \label{eq:matrix_first_example}
   W = 
    \begin{pmatrix}
      -5 & -3 & \tminus 0 \\
      \tminus(-7) & -5 & 0 \\
      -7 & -2 & \tminus 0 \\
      -2 & \tminus (-6) & \tminus 0 \\
    \end{pmatrix} 
  \end{equation*}
The matrix $W$ is not tropically generic. Indeed, consider its submatrix $W'$ formed from the first  two rows and the first two columns. We have $\tper(|W'|) = ((-5) \ttimes (-5)) \tplus (|\tminus(-7)| \ttimes (-3)) = (-10) \tplus (-10)$, thus $|W'|$ is tropically singular. However, $W'$ is not tropically sign singular, as the terms $\unit \ttimes (-5) \ttimes (-5) = -10$ and $\tminus \unit \ttimes \tminus (-7) \ttimes (-3) =-10$ associated with the maximizing permutations in $\tper(|W'|)$ have the same tropical sign.

Now consider the submatrix $W'' = \Bigl(\begin{smallmatrix} \tminus( -7) &  0 \\ -7 & \tminus 0\end{smallmatrix}\Bigr)$ formed from the second and third  rows and the first and last columns of $W$. We have $\tper (|W''|) = \bigl( |\tminus(-7)| \ttimes | \tminus 0| \bigr) \tplus \bigl((-7) \ttimes  0  \bigr) = (-7) \tplus (-7)$  and the two terms $\unit \ttimes \tminus (-7) \ttimes \tminus 0 =  -7$ and $\tminus \unit \ttimes (-7) \ttimes 0 =  \tminus(-7)$ have opposite tropical signs. Thus $W''$ is not tropically sign singular, and therefore $W$ is not tropically sign generic.
\end{example}

\subsection{Tropically convex sets and tropical polyhedra}

A set $S \subset \trop^n$ is said to be a \emph{tropically convex} if $\lambda \tscale x \tplus \mu \tscale y \in S$ for all $x, y \in S$ and $\lambda, \mu \in \trop$ such that $\lambda \tplus \mu = \unit$. The set $S$ is said to be a \emph{tropical cone} when the same conclusion holds even if the requirement that $\lambda \tplus \mu = \unit$ is omitted.  A tropical cone is \emph{polyhedral} if it is finitely generated. These notions are analogous to the classical ones, since the condition $\lambda, \mu \geq \zero$ is trivially satisfied.
Given $V \subset \trop^n$, we denote by  $\tconv(V)$ the smallest (inclusion-wise) tropically convex subset of $\trop^n$ containing $V$.  Similarly,  $\tpos(V)$ denote the smallest tropical cone of $\trop^n$ containing $V$.

\subsubsection{Tropical half-spaces and s-hyperplanes}
An \emph{(affine) tropical half-space} is a subset of $\trop^n$ of the form:
\begin{equation}
\max (\alpha_1 + x_1 , \dots , \alpha_n + x_n , \alpha_{n+1} ) \geq \max(  \beta_1 + x_1  \dots , \beta_n + x_n,  \beta_{n+1}  ) \ ,
\label{eq:halfpsace_definition}
\end{equation}
where $\alpha,\beta \in \trop^{n+1}$ . When $\alpha_{n+1} = \beta_{n+1} = \zero$, it is said to be a \emph{linear tropical half-space}. 
Throughout this paper, we assume that half-spaces are defined by non-trivial inequalities:
\begin{assumption} \label{assumption_A} There is at least one non-null coefficient in the
  inequality~\eqref{eq:halfpsace_definition}, \ie\
\[ \max \left (\max_{j \in [n+1] } \alpha_j, \max_{j \in [n+1]} \beta_j \right) > \zero \,. \]
\end{assumption}
Without loss of generality  (see~\cite[Lemma~1]{GaubertKatz2011minimal}), we also always assume that half-spaces are induced by an inequality satisfying the following condition:
\begin{assumption} \label{assumption_B} Each variable appears on at most one side of the
  inequality~\eqref{eq:halfpsace_definition}, \ie\, 
\[ \min(\alpha_j , \beta_j) = \zero  \text{ for all }  j \in [n+1]   \,.\]
\end{assumption}
Then, we can concisely describe a tropical half-space with a
signed row vector $a = (a_{1j}) \in \strop^{1 \times n}$ and a signed scalar $b \in \strop$ as:
\begin{align*}
  \tropH^{\geq}(a,b) :&= \{ x \in \trop^n \mid a^+_{11} \ttimes x_1 \tplus \cdots \tplus a^+_{1n} \ttimes x_n \tplus b^+ \geq
  a^-_{11} \ttimes x_1 \tplus \cdots \tplus a^-_{1n} \ttimes x_n \tplus b^- \} \\
&= \{ x \in \trop^n \mid a^+ \tdot x \tplus b^+ \geq a^- \tdot x \tplus b^-\} \ .
\end{align*}

A \emph{signed tropical hyperplane}, or \emph{s-hyperplane}, is defined as the set of the solutions $x \in \trop^n$ of
an equality of the form:
\begin{equation}
 \tropH(a,b) = \{ x \in \trop^n \mid a^+ \tdot x \tplus b^+ = a^- \tdot x \tplus b^-\} \; ,
\label{eq:signed_hyperplane_definition}
\end{equation}
where $a \in \strop^{1 \times n}$ and $b \in \strop$.  
When $\tropH^{\geq}(a,b)$ is a non-empty proper subset of $\trop^n$, its boundary is  $\tropH(a,b)$.

\begin{remark}
  The set $\tropH(a,b)$ is said to be \emph{signed} because it corresponds to the tropicalization of the intersection of a usual hyperplane with the
  non-negative orthant over Puiseux series; see Section \ref{sec:puiseux}. A tropical (unsigned) hyperplane is defined by an unsigned row vector $a = (a_{1j})
  \in \trop^{1 \times n}$ and an unsigned scalar $b \in \trop$ as the set of all points $x \in \trop^n$ such that the maximum is attained at least
  twice in $a \tdot x \tplus b = \max(a_{11} + x_1, \dots, a_{1n}+ x_n, b)$; see \cite{richter2005first}.  This corresponds to the tropicalization of
  an entire ordinary hyperplane.
\end{remark}

\subsubsection{Tropical polyhedra}
A \emph{tropical polyhedron} is the intersection of finitely many tropical affine half-spaces. It will be denoted by a signed matrix $A \in \strop^{m
  \times n}$ and a signed vector $b \in \strop^m$ as:
\begin{equation*}
  \tropP(A,b) := \{x \in \trop^n \mid A^+ \tdot x \tplus b^+ \geq A^- \tdot x \tplus b^- \} =  \bigcap_{i \in [m] } \tropH^{\geq} (\pr_i, b_i) \ .
\end{equation*}
If all those tropical halfspaces are linear, \ie\ if $b$ is identically $\zero$, that intersection is a tropical polyhedral cone.

\begin{example}
  The tropical polyhedron depicted in Figure~\ref{fig:intro_example} is defined by the following matrix and vector.
  \begin{equation*}
A =     \left (
  \begin{array}{ccc}
        -5 & -3  \\
      \tminus(-7) & -5  \\
      -7 & -2  \\
      -2 & \tminus (-6) \\
    \end{array}
    \right ) 
 \text{ and } 
b = \left (
 \begin{array}{c}
        \tminus 0 \\
     0  \\
      \tminus 0 \\
     \tminus 0\\
    \end{array}
\right )
 \;
 \begin{tabular}{c}
    \begin{tikzpicture}[baseline,scale=1] \draw[fill=orange ] (0,0.0)
        rectangle (1,0.1);
      \end{tikzpicture} \\
 \begin{tikzpicture}[baseline,scale=1] \draw[fill={rgb:red,50;green,0;blue,50} ] (0,0.0)
        rectangle (1,0.1);
      \end{tikzpicture} \\
 \begin{tikzpicture}[baseline,scale=1] \draw[fill={rgb:red,10;green,50;blue,10} ] (0,0.0)
        rectangle (1,0.1);
      \end{tikzpicture} \\
 \begin{tikzpicture}[baseline,scale=1] \draw[fill={rgb:red,50;green,50;blue,0} ] (0,0.0)
        rectangle (1,0.1);
      \end{tikzpicture}
 \end{tabular}
  \end{equation*}
The half-space depicted in orange in Figure~\ref{fig:intro_example} is $\tropH^{\geq}(A_1, b_1) = \{ x \in \trop^2 \mid \max(x_1 -5, x_2 -3 ) \geq 0 \}$. Its boundary is the signed hyperplane  $\tropH(A_1, b_1) = \{ x \in \trop^2 \mid \max(x_1 -5, x_2 -3 ) = 0 \}$. 
The last three rows yield the inequalities:
\begin{align*}
  \max(x_2, 0) &\geq x_1 -7  \,,\\
  \max(x_1 -7, x_2 -2) &\geq 0 \,, \\
  x_1 &\geq \max(x_2 -6, 0) \,,
\end{align*}
which define the half-spaces respectively depicted in purple, green and khaki in Figure~\ref{fig:intro_example}.
\end{example}

A point $x$ in a tropical polyhedron $\tropP(A,b)$ clearly satisfies the inequalities $x_j \geq \zero$ for all $j \in[n]$.  Although redundant, including these inequalities in the representation of a tropical polyhedron is occasionally useful.  
\begin{assumption} \label{assumption_C} For all $j \in [n]$, all points $x\in\tropP(A,b)$ satisfy
  $x_j> \zero$ or the non-negativity constraint $x_j \geq \zero$ appears in the external representation of $\tropP(A,b)$,
  \ie\ there exists a row index $i \in [m]$ such that $(\pr_i \ b_i)$ is the row vector whose $j$th entry is $\unit$
  while all other entries are $\zero$.
\end{assumption}
 
The Minkowski--Weyl theorem holds in the tropical case: a tropical polyhedron can be defined either externally (\ie\ by means of half-spaces), or internally as the convex hull of finitely many points and rays. 
\begin{theorem}[{\cite[Theorem~2]{GaubertKatz2011minimal}}] \label{thm:trop_minkowski_weyl}
A subset $\tropP \subset \trop^n$ is a tropical polyhedron if, and only if, there exist two finite sets $V,R \subset \trop^n$ such that 
\begin{equation*}
\tropP =\{ x \tplus y \mid x \in  \tconv(V) \text{ and } y \in \tcone(R) \} \ .
\end{equation*}
\end{theorem}

It will be convenient to homogenize a tropical polyhedron $\tropP(A,b)$ into the tropical
polyhedral cone $\tropC := \{ x \in \trop^{n+1} \mid \epm^+ \tdot x \geq \epm^- \tdot x \}$, where $\epm := (A\ b)$. As a tropical cone,
$\tropC$ is closed under tropical scalar multiplication. For this reason, we identify $\tropC$ with its image in the \emph{tropical projective space}
 \[
 \tropProj^n \ := \  \left\{ \R \tscale x \mid x \in \trop^{n+1} \setminus \{(\zero, \dots, \zero)\} \right\} \,.
 \]

The points of the tropical polyhedron $\tropP(A,b)$ are associated with elements of the tropical polyhedral cone
$\tropC$ by the following bijection:
\begin{equation}
\begin{aligned}
\tropP(A,b) & \longrightarrow \{ y \in \tropC \mid y_{n+1} = \unit \} \\
x & \longmapsto (x,\unit) 
\end{aligned}
\label{eq:homogenization}
\end{equation}
The points of the form $(x,\zero)$ in $\tropC$ correspond to the rays in the recession cone of $\tropP(A,b)$, see~\cite{GaubertKatz2011minimal}.
 
\begin{remark}\label{rem:tcone}
  Let $R\in\trop^{m\times n}$ be a matrix with finite coefficients only. Then $\tropP=\tcone(R)$ is a tropical
  polyhedral cone in $\trop^n$ such that the image of $\tropP\cap\R^n$ under the canonical projection from $\R^n$ to the
  \emph{tropical torus} $\{ \R \tscale x \mid x \in \R^n \}$ is a ``tropical polytope'' in the sense of Develin and
  Sturmfels \cite{develin2004}.  Via this identification, the tropical linear halfspaces which are non-empty proper subsets of $\trop^n$ correspond to the
  ``tropical halfspaces'' studied in \cite{Joswig05}.  The tropical projective space defined above compactifies the
  tropical torus (with boundary).
\end{remark}

\subsection{Puiseux series} \label{sec:puiseux}

The set $\R\{\!\{t\}\!\}$ of \emph{(generalized) Puiseux series} with real coefficients is the set of formal power series
\begin{equation*}
  \x = \sum_{ \alpha \in \R } x_{\alpha} t^{\alpha}
\end{equation*}
with $x_{\alpha} \in \R$ such that the support $\{ x_{\alpha} \mid x_{\alpha} \neq 0 \}$ is either a
finite set or the set of valuations of a increasing unbounded sequence. By definition of its support, every non-null Puiseux
series $\x$ admits a smallest exponent $\alpha_{\min} \in \R$. The real number $- \alpha_{\min}$ is called the \emph{valuation} of
$\x$ and is denoted $\val(\x)$. By convention, we set $\val(0) = - \infty$. The \emph{leading coefficient}, denoted
$\lc(\x)$, is the coefficient $x_{\alpha_{\min}}$ of the smallest exponent $\alpha_{\min} =- \val(\x)$ when $\x \neq 0$, and $0$
otherwise.  Throughout the paper, we write $\puiseux$ instead of $\R\{\!\{t\}\!\}$.

The set of generalized Puiseux series, equipped with the sum and product of formal power series, constitutes a field.  It
can be identified with a subfield of the field of Hahn series, \ie\ formal power series with arbitrary well-ordered
support.  A variant of this field 
was also considered by Hardy under the name of ``generalized Dirichlet series''.  Our
approach follows Markwig~\cite{markwig2007field}.

The $n$-fold Cartesian product $\puiseux^n$ is a $\puiseux$-vector space when equipped with the scalar multiplication $(\plambda, \x) \mapsto \plambda \x := ( \plambda \x_1, \dots, \plambda \x_n)$ and the vector addition $(\x, \y) \mapsto \x+\y :=(\x_1 + \y_1, \dots, \x_n + \y_n)$.

A Puiseux series $\x$ is said to be \emph{positive} if $\lc(\x) > 0$, and we write $\x > 0$ in this
case. Similarly, we write $\x > \y$ if $\x - \y > 0$.  This definition turns $\puiseux$ into an ordered field. The topology induced by this order makes $\puiseux$ a topological field.

The valuation is a map from $\puiseux$ to $\trop$ which satisfies
\begin{align*}
  \val(\x\y) &= \val(\x) \ttimes \val(\y) \\
  \val( \x + \y ) & \leq \val(\x) \tplus \val(\y).
\end{align*}
Equality occurs in the last inequality if and only if the leading terms of $\x$ and $\y$ do not cancel. In particular,
cancellation never occurs whenever $\x$ and $\y$ share the same sign.
This property is the main reason for using Puiseux series to study the tropical semiring. Indeed, the map $\val$
defines a homomorphism from the semiring $\puiseux_+$ of non-negative Puiseux series to the tropical semiring.  This
homomorphism is order preserving, that is,
\[
\text{if } \x \geq \y \geq 0 \text{ then } \val(\x) \geq \val(\y) \enspace .
\]

It is convenient to equip the valuation with a sign information. We define the \emph{signed valuation} map by:
\begin{align*}
  \sval:\; &\puiseux \longrightarrow \strop \,\\
         & \x \longmapsto 
         \begin{cases}
           \val(\x) & \text{if} \ \x \geq 0 \ ,\\
           \tminus \val(\x) & \text{otherwise.}
         \end{cases}
\end{align*}
A \emph{lift} of a signed tropical number $x \in \strop$ is a Puiseux series $\x$ such that $\sval(\x)=x$. Clearly,
such a lift is by no means unique.  The set of all lifts will be denoted $\sval^{-1} (x)$.  The signed valuation
map is extended to vectors and matrices by component-wise application.  In the following, any Puiseux series will be
written in bold and its signed valuation with a standard font, \eg\ $x = \sval(\x)$.

\subsection{Puiseux linear programming solves tropical linear programming}

Hyperplanes, half-spaces and convex polyhedra can be defined over an arbitrary ordered field.  The most basic results
used in linear programming (Farkas' lemma, Minkowski--Weyl, Strong Duality, \etc) are of algebraic nature. Their proofs
only rely on the axioms of ordered fields, and consequently are also valid in this setting, see for
instance~\cite{Jeroslow1973,Megiddo1987,FilarAltmanAvrachenkov2002}. Actually, in the present paper, we  deal
with the field $\puiseux$ of Puiseux series with real coefficients, which is known to be real closed~\cite{markwig2007field}, \ie\ each
non-negative element is a square, and every polynomial with odd degree has at least one root. For such a field, stronger
results follow from Tarski's principle: any first-order sentence that is valid over the reals is also valid over an
arbitrary real closed field, and thus valid over $\puiseux$. We refer to~\cite{tarski1951decision,seidenberg1954new} for further details; see
also~\cite{basu2006algorithms} for a recent overview.  In order to have a concise name we call ordinary polyhedra
defined over $\puiseux$ \emph{Puiseux polyhedra}.

In this section, we examine how tropical polyhedra are related with Puiseux polyhedra in $\puiseux_+^n$ via the
valuation map.  In~\cite[Proposition~2.1]{DevelinYu07}, Develin and Yu prove that a tropical polyhedral cone $\tpos(R)$
can be lifted to a Puiseux polyhedral cone in $\puiseux^n_+$ by lifting the set $R$ of generators. This result can be
trivially extended to arbitrary tropical polyhedra, thanks to the tropical Minkowski--Weyl Theorem
(Theorem~\ref{thm:trop_minkowski_weyl}), by lifting the whole internal representation. Alternatively, we shall see that a tropical
polyhedron can also be lifted to a Puiseux polyhedron in $\puiseux_+^n$ by lifting its external representation by
half-spaces.  As a consequence, an optimal solution to a tropical linear program can be found by solving a linear
program over Puiseux series.

We denote by $\H(\a,\b)$ the hyperplane over $\puiseux^n$ defined by the equality $\a \x + \b = 0$, where $\a \in
\puiseux^{1 \times n}$ and $\b \in \puiseux$. The hyperplane $\H(\a,\b)$ induces the half-space $\H^{\geq}(\a,\b)$ by
replacing the equality constraint by the inequality $\geq$.
We will denote Puiseux polyhedra as follows:
\[
\puiseuxP(\A,\b) := \{ \x \in \puiseux^n \mid \A \x + \b \geq 0
\} \ ,
\]
where $\A \in \puiseux^{m \times n}$ and $\b \in \puiseux^m$.

We now consider a tropical linear program:
 \begin{equation}
   \begin{array}{ll}
     \text{\rm minimize} & c \tdot x \ \\
     \text{\rm subject to} & x \in \tropP(A,b)\\
   \end{array}
   \label{eq:trop_linear_prog_pb}  
 \end{equation}
 where $A \in \strop^{m \times n}$, $b \in \strop^m$ are signed matrices and $c \in \trop^{1 \times n}$ is an unsigned
 row vector.
 
\begin{proposition} \label{prop:puiseux_solves_tropical_program}
There is a way to associate to every tropical linear program of
the form~\eqref{eq:trop_linear_prog_pb}
satisfying Assumption \ref{assumption_C} a Puiseux linear program 
\begin{equation}
\begin{array}{ll}
\text{\rm minimize} & \cc \x \ \\
\text{\rm subject to} & \x \in \puiseuxP(\A,\b)\\
\end{array}
\label{eq:puiseux_linear_prog_pb}  
\end{equation}
satisfying $\A \in \sval^{-1} (A)$, $\b \in \sval^{-1}(b)$ and $\cc \in \sval^{-1} (c)$, so that:
\begin{enumerate}[(i)]
\item the image by the valuation of the feasible set of the linear program~\eqref{eq:puiseux_linear_prog_pb} is precisely the feasible set of the tropical linear program~\eqref{eq:trop_linear_prog_pb}; in particular, \eqref{eq:puiseux_linear_prog_pb} is feasible if and only if \eqref{eq:trop_linear_prog_pb} is feasible;
\item the valuation of any optimal solution of~\eqref{eq:puiseux_linear_prog_pb} (if any) is an optimal solution of~\eqref{eq:trop_linear_prog_pb}.
\end{enumerate}
\end{proposition}
Notice that the converse of (ii) does not necessarily hold.  That is, there are tropical linear programs with optimal solutions which do
not arise as projections from any lift; see Example~\ref{example:running_example} below.
\begin{proof}
  To begin with, we will exhibit lifts $\A \in \sval^{-1} (A)$ and $\b \in \sval^{-1}(b)$ of the external representation such that $\val(\puiseuxP(\A,\b)) =
  \tropP(A,b)$.  The inclusion $\val(\puiseuxP(\A,\b)) \subset \tropP(A,b)$ is satisfied for any lifts $\A$, $\b$. Indeed, consider a point $\x \in
  \puiseuxP(\A,\b)$.  Then, by Assumption \ref{assumption_C}, the polyhedron $\puiseuxP(\A,\b)$ is included in the non-negative orthant $\puiseux^n_+$.  Let
  $(\A \ \b) = (\A^+ \ \b^+) - (\A^- \ \b^-)$ where the entries of $(\A^+ \ \b^+)$ and $(\A^- \ \b^-)$ are non-negative. Every point $\x \in \puiseuxP(\A,\b)$
  satisfies $\A^+ \x + \b^+ \geq \A^- \x +\b^-$.  Since the Puiseux series on each side of these inequalities are non-negative, the valuation preserves their
  ordering and $A^+ \tdot x \tplus b^+ \geq A^- \tdot x \tplus b^-$.

  We claim that the reverse inclusion holds for any lift of the form $(\A \ \b) = (\A^+ \ \b^+) - (\A^- \ \b^-)$ defined, for $i \in [m]$ and $j \in [n]$ by:
\begin{equation}\label{eq:puiseux_lift}
  \begin{aligned}
    \A^+ & = (\alpha t^{-a^+_{ij}})  \text{ and } \A^- = (t^{-a^-_{ij}}) \\
    \b^+ &= (\alpha t^{-b^+_i}) \text{ and } \b^- = (t^{-b^-_i})
  \end{aligned}
\end{equation}
  where $\alpha$ is a real number strictly greater than $n+1$, and $A = (a_{ij})$. 
 To see this, observe that for any $x \in \tropP(A,b)$ the lift $\x = (t^{-x_1}, \dots,
  t^{-x_n})$ belongs to the Puiseux polyhedron $\puiseuxP(\A,\b)$. Indeed, for any $i \in [m]$ we have
  \[
  \ppr^-_i \x + \b^-_i = \sum_{j=1}^n t^{-a^-_{ij}- x_j} + t^{-b^-_i} \leq (n+1) t^{-(a^-_i \tdot x \tplus b^-_i)} < \alpha t^{-(a^-_i \tdot x \tplus b^-_i)}
  \]
  and
  \[
  \ppr^+_i \x + \b^+_i = \alpha \Bigl(\sum_{j=1}^n t^{ -a^+_{ij}- x_j} + t^{-b^+_i}\Bigr) \geq \alpha t^{ -(a^+_i \tdot x \tplus b^+_i )} \geq \alpha t^{ -(a^-_i \tdot x \tplus b^-_i )} \ ,
  \]
  thus $\ppr_i \x + \b_i > 0$.  This shows that a lift to real Puiseux series does exist.

  We need to prove the claimed properties of such a lift.  Let $\A$ and $\b$ as above. Since $\val(\puiseuxP(\A,\b)) = \tropP(A,b)$, the Puiseux linear
  program~\eqref{eq:puiseux_linear_prog_pb} is feasible if, and only if, the tropical one~\eqref{eq:trop_linear_prog_pb} is feasible.
  Now take any $\cc \in \sval^{-1}(c)$, \eg\ $\cc_j = t^{-c_j}$ for $j \in [n]$. If~\eqref{eq:puiseux_linear_prog_pb} admits an optimal solution $\x^*$, then $\cc \x \geq \cc \x^* \geq 0$ for all $\x \in \puiseuxP(\A,\b)$.  Since $c$ is non-negative, $c \tdot x \geq c \tdot \val(\x^*)$ for all $x \in \val ( \puiseuxP(\A,\b) ) = \tropP(A,b) $.  This concludes the proof.
\end{proof}

\begin{remark}
  Observe that in Proposition~\ref{prop:puiseux_solves_tropical_program}, the Puiseux linear program~\eqref{eq:puiseux_linear_prog_pb} cannot be
  unbounded. Indeed,  for all lifts $(\A \ \b) \in \sval^{-1}(A \ b )$, we have $\puiseuxP(\A,\b) \subset \puiseux_+^n$ thanks to
  Assumption~\ref{assumption_C}. Since  $c$ has tropically positive entries, its lift $\cc$ also have positive entries. Then the inequality $\cc \x \geq 0$ holds for all $\x \in \puiseuxP(\A, \b)$, and thus provides a lower
  bound for the minimization problem~\eqref{eq:puiseux_linear_prog_pb}.
\end{remark}

\begin{figure}[t]
  \begin{minipage}{\textwidth}
    \centering
  \begin{tikzpicture}[scale=1.0]
 \node[anchor=south west,inner sep=0] (tropical_polytope) at (0,0) {\includegraphics[width=0.45\textwidth]{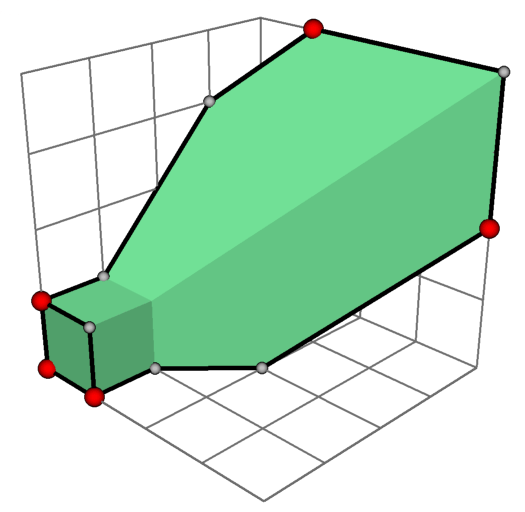}};
 \begin{scope}[x={(tropical_polytope.south east)},y={(tropical_polytope.north west)}]
\node at (-0.02, 0.28) {$(0,0,0)$};
\node at (0.00, 0.91) {$(0,0,4)$};
\node at (0.5, -0.02) {$(4,0,0)$};
\node at (1.01, 0.28) {$(4,4,0)$};
\node at (0.95, 0.92) {$(4,4,4)$};
 \end{scope}
 \node[right  = of tropical_polytope] (tropical_halfspaces)  {\includegraphics[width=0.44\textwidth]{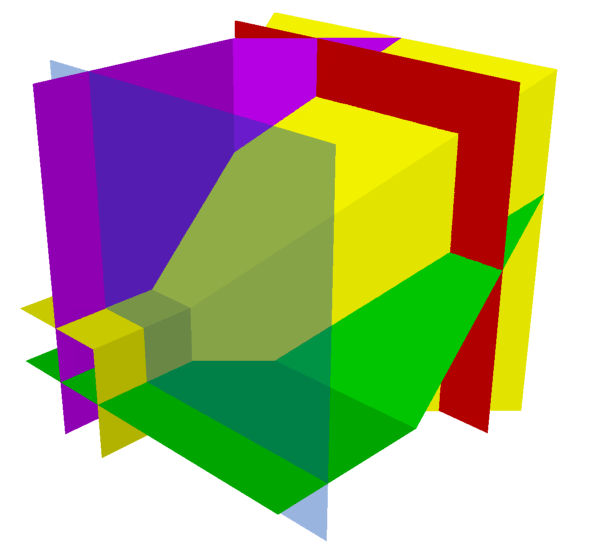}};
\end{tikzpicture}
  \caption{The tropical polyhedron defined by the inequalities \eqref{eq:running_example} and its external representation.}
  \label{fig:tropical_polytope}
\end{minipage}

\begin{minipage}{\textwidth}
  \centering
  \begin{tikzpicture}[scale=1.0]
 \node[anchor=south west,inner sep=0] (lift_polytope) at (0,0) {\includegraphics[width=0.45\textwidth]{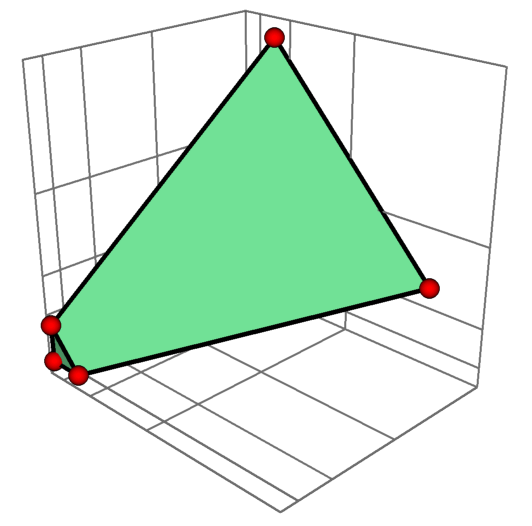}};
 \begin{scope}[x={(lift_polytope.south east)},y={(lift_polytope.north west)}]
\node at (-0.02, 0.28) {$(t^0,t^0,t^0)$};
\node at (0.00, 0.94) {$(t^0,t^0,t^{-4})$};
\node at (0.55, -0.01) {$(t^{-4},t^0,t^0)$};
\node at (1.0, 0.20) {$(t^{-4},t^{-4},t^0)$};
\node at (0.95, 0.94) {$(t^{-4},t^{-4},t^{-4})$};
 \end{scope}
 \node[right  = of lift_polytope] (lift_halfspaces)  {\includegraphics[width=0.42\textwidth]{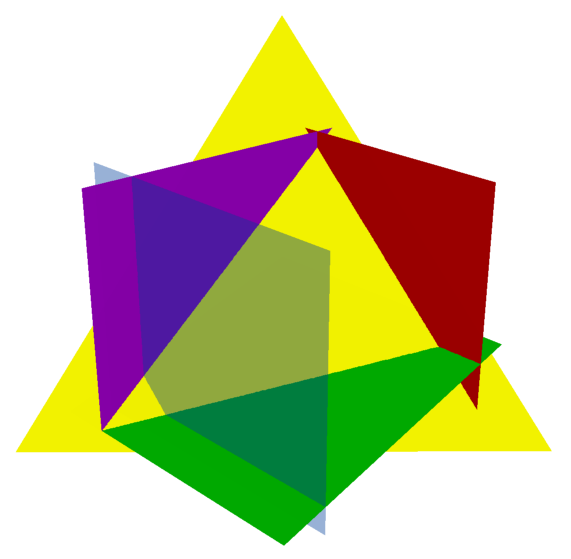}};
\end{tikzpicture}
  \caption{A lift of the tropical polyhedron defined by the inequalities~\eqref{eq:running_example} and its external representation.}
  \label{fig:puiseux_polytope}
\end{minipage}
\end{figure}

We offer a visualization for Proposition~\ref{prop:puiseux_solves_tropical_program} in Figure~\ref{fig:puiseux_polytope}.  As model theory
explains, polyhedra over real Puiseux series, for most purposes, pretty much are the same as classical polyhedra over the reals.  Therefore,
we can also visualize the lift in Proposition~\ref{prop:puiseux_solves_tropical_program} as a classical polyhedron.  More precisely, the
constraints of the lift are provided by~\eqref{prop:puiseux_solves_tropical_program}, and we can choose the lifted objective vector to be
$\cc_j = t^{-c_j}$ for $j \in [n]$. Now, replacing the parameter $t$ by a real number provides a real linear program.  If that real number
is sufficiently small, the classical linear program over the reals is combinatorially equivalent to the Puiseux one, \ie\ that is they share
the same vertex--facet incidences and hence they share the same vertex--edge graph; the optimal vertices are in bijection.

\begin{example} \label{example:running_example}
Throughout the rest of this paper, we will illustrate our results on the following tropical linear program.
  \begin{subequations}
    \renewcommand{\theequation}{\theparentequation.\arabic{equation}}
    \label{eq:running_example}
    \begin{align}
     & \text{minimize } & \max(x_1-2, x_2, x_3-1)&  & \nonumber \\
     & \text{subject to } &
     \max( 0, x_2-1) & \geq \max(x_1-1, x_3-1)&  \begin{tikzpicture}[baseline,scale=1] \draw[fill=yellow ] (0,0.0)
        rectangle (1,0.1);
      \end{tikzpicture} \tag*{$\tropH_1$} \label{eq:running_example_eq2}
      \\
     && x_3 & \geq \max(0, x_2-2) & \begin{tikzpicture}[baseline,scale=1] \draw[fill = green ] (0,0.0) rectangle (1,0.1);
      \end{tikzpicture} \tag*{$\tropH_2$}  \label{eq:running_example_eq3}
      \\
     && x_2 & \geq  0 &  \begin{tikzpicture}[baseline,scale=1] \draw[fill={rgb:red,0;green,102;blue,255}, opacity = 0.7 ]
        (0,0.0) rectangle (1,0.1);
      \end{tikzpicture} \tag*{$\tropH_3$}  \label{eq:running_example_eq4}
      \\
     && x_1 & \geq  \max(0, x_2-3) & \begin{tikzpicture}[baseline,scale=1] \draw[fill={rgb:red,204;green,0;blue,255} ]
        (0,0.0) rectangle (1,0.1);
      \end{tikzpicture} \tag*{$\tropH_4$} \label{eq:running_example_eq5}
      \\
     && 0 & \geq  x_2-4  \ .&
      \begin{tikzpicture}[baseline,scale=1]
        \draw[fill=red ] (0,0.0) rectangle (1,0.1);
      \end{tikzpicture} \tag*{$\tropH_5$} \label{eq:running_example_eq1}
    \end{align}
  \end{subequations}
These constraints define the tropical polyhedron represented in Figure~\ref{fig:tropical_polytope}. A lift of this tropical polyhedron is depicted in Figure~\ref{fig:puiseux_polytope}.
The optimal valuation of this tropical linear program is $0$ and the set of optimal solutions  is the ordinary square: 
\[
\{ (x_1,x_2,x_3) \in \trop^3 \mid 0 \leq x_1 \leq 1 \text{ and } x_2=0  \text{ and } 0 \leq x_3 \leq 1\} .
\]
However, over Puiseux series, there is a unique optimum.  It is the point located in the intersection of three hyperplanes obtained by
lifting the inequalities \eqref{eq:running_example_eq3}, \eqref{eq:running_example_eq4} and \eqref{eq:running_example_eq5}. This point has
valuation $(0,0,0)$, which is an optimum for the tropical linear program.  Corollary~\ref{coro:val_and_inter_commute_for_generic_matrices}
and Proposition~\ref{prop-pivotingimprove} below assert this does not depend on the choice of the lift.  The reason is that our example
satisfies the standard conditions mentioned in Theorem~\ref{th-main}.

We will also present several results in the homogeneous setting. They will be illustrated on the tropical cone defined by the following inequalities:
\begin{equation}
\begin{aligned}
\max(x_4, x_2-1) & \geq \max(x_1-1, x_3-1) \\
x_3 & \geq \max(x_4, x_2-2) \\
x_2 & \geq x_4 \\
x_1 & \geq \max(x_4, x_2-3) \\
x_4 & \geq  x_2-4
\end{aligned}
\label{eq:homogenized_running_example}
\end{equation}
This cone corresponds to the previous polyhedron by the correspondence given in~\eqref{eq:homogenization}, \ie\ the coordinate $x_4$ plays the role of the affine component. For the sake of simplicity, the linear half-spaces in~\eqref{eq:homogenized_running_example} are still referred to as~\eqref{eq:running_example_eq2}--\eqref{eq:running_example_eq1}.
\end{example}

\subsection{The simplex method}
A Puiseux linear program can be solved using the classical simplex method. We briefly recall the basic facts.
Let $\basis \subset [m]$ be a subset of cardinality $n$ such that the submatrix $\A_\basis$, formed from the rows with indices in $\basis$, is non-singular. The intersection $ \bigcap_{i \in \basis} \puiseuxH(\ppr_i, \b_i)$ contains a unique point, which we denote as $\x^\basis$. When $\x^\basis$ belongs to the polyhedron $\puiseuxP(\A,\b)$, it is called a \emph{(feasible)  basic point}.
\begin{remark}
  A basis is usually defined by a partition of the (explicitly bounded) variables $( \s_1, \dots, \s_m )$ in ``basic''  and ``non-basic'' variables, where $\s = \A \x + \b$. Observe that $\basis$ correspond to the ``non-basic'' variables as it  indexes the zero coordinates of $\s$. The set $\basis$ can also be interpreted as the set of ``basic'' variables in the dual program.
\end{remark}
For any set $I\subset[m]$ we let
\[
\puiseuxP_I(\A,\b) := \bigcap_{i \in I} \puiseuxH(\ppr_i, \b_i) \cap \puiseuxP(\A,\b) \,.
\]
A subset $K \subset [m]$ of cardinality $n-1$ defines the \emph{(feasible) edge} $\puiseuxE_K := \puiseuxP_K(\A,\b)$
when $\bigcap_{ i \in K} \puiseuxH(\ppr_i, \b_i)$ is an affine line that intersects $\puiseuxP(\A,\b)$.  Notice that an edge
defined in this way may have ``length zero'', \ie\ as a set it only consists of a single point.  A basic point $\x^\basis$ is
contained in the $n$ edges defined by the sets $\basis \setminus \{k\}$ for $k \in \basis$. The edge $\basis \setminus \{ k \}$ belongs to the line directed by the vector $\d^k$ with coordinates
\begin{equation}\label{eq:edge_direction_vector}
  \d^k_j = (-1)^{k+j}\frac{ \det \M_{kj} }{ \det \A_\basis} \ ,
\end{equation}
where $\M_{kj}$ is the matrix obtained from $\A_\basis$ by deleting its $k$th row and $j$th column.

As we are minimizing, moving along the edge $\basis\setminus \{k \}$ from the basic point $\x^\basis$ improves the objective
function if the \emph{reduced cost} $\y_k = \cc \d^k$ is negative. The vector of reduced costs $ \y = (\y_k)_{k \in \basis}$ forms a
solution of the following linear system of equations:
\begin{equation}
  - \A_\basis^{\top} \y + \cc = 0 \ .
\label{eq:puiseux_reduced_costs}
\end{equation}
Each iteration of the simplex method starts on some basic point $\x^\basis$. An edge $\basis \setminus \{k \}$ with a negative
reduced cost is selected. If no such edge exists, then the basic point is optimal by the Strong Duality Theorem of
Linear Programming \cite[\S5.5]{Schrijver03:CO_A} (which holds in any ordered field such as real Puiseux
series). Otherwise, the algorithm \emph{pivots}, \ie\ moves to the other end of the selected edge.
 Pivoting amounts to finding the \emph{length} $\pmu \in \puiseux$ of the
edge, which is given by:
\begin{equation}
\pmu = \inf \left \{ \frac{\ppr_i \x^\basis + \b_i}{ -\ppr_i \d^k} \mid i \in [m] \setminus \basis \text{ and } \ppr_i \d^k < 0 \right \} .
\label{eq:puiseux_pivot}
\end{equation}
If the edge is bounded, \ie\ if there exists an $i \in [m]\setminus \basis$ such that $\ppr_i \d^k <0$, then the algorithm reaches a new basic
point. Otherwise, the linear program is unbounded, and the valuation $\pmu$ is $\infty$.

\begin{remark}
  Basic points and directions are provided by determinants. If $(A \ b) = \val(\A \ \b)$ is
  tropically generic, this amounts to computing tropical permanents. However, the length $\pmu$ of the edge cannot be
  computed only with valuations.  This difficulty can be observed already in dimension one. Consider the Puiseux
  polyhedron defined by the inequalities:
  \[
  \x \leq 1 \text{ and }
  \x \geq t^{2} \text{ and }
  \x \geq t^{3} \ .
  \]
  Minimizing $\cc=1$ and starting from the basic point $\x=1$, the direction of the single pivot is $\d=-1$.  The pivoting step
  must decide where the edge ends; and in this case the edge length is given by $\pmu = \min(1-t^{2},
  1-t^{3})=1-t^{2}$.  Yet, the valuation of $1-t^{2}$ and $1-t^{3}$ yields zero in both cases.  This shows that, in
  order to find the correct minimum $t^2$, it does not suffice to look at the valuations of the optimal solutions of the Puiseux lift.
\end{remark}

\section{Tropical basic points and tropical edges} \label{sec:tropical_basic_points}

\noindent
Geometrically speaking, the classical simplex method traces the vertex-edge graph of an ordinary polyhedron from one
basic point along a directed path to an optimal solution, which again is basic. Basic points and edges over Puiseux
series are cells of the arrangement of hyperplanes $\{ \puiseuxH(\ppr_i,\b_i) \}_{i \in [m]}$.  It turns out that, under
some genericity assumptions, the valuation of these cells can be described by intersecting tropical half-spaces in
$\{\tropH^{\geq}(\pr_i, b_i) \} _{i \in [m]}$ and s-hyperplanes in $\{\tropH(\pr_i, b_i) \}_{i \in [m]}$.  This result will
be proved in Corollary~\ref{coro:val_and_inter_commute_for_generic_matrices} below.

\subsection{The tangent digraph}

Consider a matrix $\epm  = (\ep_{ij}) \in \strop^{m \times (n+1) }$.
For every point $x\in\trop^{n+1}$ with no $\zero$ entries, we define the \emph{tangent graph} $\graph_x(\epm)$ at the
point $x$ with respect to $\epm$ as a bipartite graph over the following two disjoint sets of nodes:
the  ``coordinate nodes'' $[n+1]$ and
 the ``hyperplane nodes'' $ \{ i \in [m] \mid \epr_i^+ \tdot x = \epr_i^- \tdot x > \zero \}$.
There is an edge between the hyperplane node $i $ and the coordinate node $j$ when $j \in \argmax(|\epr_i|\tdot x)$.

The \emph{tangent digraph} $\digraph_x(\epm)$ is an oriented version of $\graph_x(\epm)$, where the edge between the hyperplane
node $i$ and the coordinate node $j$ is oriented from $j$ to $i$ when $\ep_{ij}$ is tropically positive, and from $i$ to $j$ when $\ep_{ij}$ is tropically negative (if a tangent digraph  contains an edge between $i$ and $j$ then $\ep_{ij} \neq \zero$).

Examples of 
tangent digraphs are given in Figure~\ref{fig:tangent_graphs} below (there, hyperplane nodes are denoted $\tropH_i$). 
The term ``tangent'' comes from the fact that $\digraph_x(W)$ is a combinatorial encoding of the tangent cone at $x$ in the tropical cone $\tropC = \tropP(W, \zero)$, see~\cite{AllamigeonGaubertGoubaultDCG2013}. The tangent digraph is the same for any two points in the same cell of the arrangement of tropical hyperplanes
given by the inequalities.  
The tangent graph $\graph_x(\epm)$ corresponds to the ``types'' introduced in \cite{develin2004} but relative only to the hyperplanes given by the tight inequalities at $x$.

When there is no risk of confusion, we will denote by $\graph_x$ and $\digraph_x$ the tangent graph and digraph, respectively.

\newcommand{\hypOne}{\ref{eq:running_example_eq1}}
\newcommand{\hypTwo}{\ref{eq:running_example_eq2}}
\newcommand{\hypThree}{\ref{eq:running_example_eq3}}
\newcommand{\hypFour}{\ref{eq:running_example_eq4}}

\begin{figure}[t]
  \centering
\tikzset{
 vertex/.style={circle,draw=black, thick }, 
 hyp/.style={rectangle,draw=black, thick }, 
 edge/.style={draw=black,thick, >= triangle 45, ->}
}

\newcommand{\tangentGraphVarNodes} {
\node [vertex]    (z) at (0,0) {$3$};
\node [vertex]    (y) at (2,0) {$2$};
\node [vertex]    (x) at (2,2) {$1$};
\node [vertex]    (b) at (0,2) {$4$};
}

\begin{tikzpicture}[scale=1.3]

\begin{scope}[shift = {(0, 0) }]
\node at (1,-0.6) {At  $(1,0,0)$};
\tangentGraphVarNodes
\node [hyp] (2) at (1,2) {\hypTwo};
\node [hyp] (3) at (0,1) {\hypThree};
\node [hyp] (4) at (1,1) {\hypFour};

\draw[edge,->] (b) -- (2);
\draw[edge,->] (2) -- (x);

\draw[edge,->] (z) -- (3);
\draw[edge,->] (3) -- (b);

\draw[edge,->] (y) -- (4);
\draw[edge,->] (4) -- (b);
\end{scope}

\begin{scope}[shift = {( 3, 0) }]
\node at (1,-0.6) {In the open segment};
\node at (1,-1) {  $](1,0,0), (1,1,0)[$};
\tangentGraphVarNodes
\node [hyp] (2) at (1,2) {\hypTwo};
\node [hyp] (3) at (0,1) {\hypThree};

\draw[edge,->] (b) -- (2);
\draw[edge,->] (2) -- (x);

\draw[edge,->] (z) -- (3);
\draw[edge,->] (3) -- (b);
\end{scope}

\begin{scope}[shift = {( 6, 0) }]
\node at (1,-0.6) {At  $(1,1,0)$};
\tangentGraphVarNodes
\node [hyp] (2) at (1,1) {\hypTwo};
\node [hyp] (3) at (0,1) {\hypThree};

\draw[edge,->] (b) -- (2);
\draw[edge,->] (y) -- (2);
\draw[edge,->] (2) -- (x);

\draw[edge,->] (z) -- (3);
\draw[edge,->] (3) -- (b);
\end{scope}

\begin{scope}[shift = {( 9, 0) }]
\node at (1,-0.6) {In the open segment};
\node at (1,-1) {  $](1,1,0), (2,2,0)[$};
 \tangentGraphVarNodes
 \node [hyp] (2) at (2,1) {\hypTwo};
 \node [hyp] (3) at (0,1) {\hypThree};

 \draw[edge,->] (y) -- (2);
 \draw[edge,->] (2) -- (x);
 \draw[edge,->] (z) -- (3);
 \draw[edge,->] (3) -- (b);
\end{scope}

\begin{scope}[shift = {( 1.5, -4) }]
  \node at (1,-0.6) {At $(2,2,0)$};
  \tangentGraphVarNodes 
  \node [hyp] (2) at (2,1) {\hypTwo}; 
  \node [hyp] (3) at (1,1) {\hypThree};

  \draw[edge,->] (y) -- (2); 
  \draw[edge,->] (2) -- (x);

  \draw[edge,->] (z) -- (3); 
  \draw[edge,->] (3) -- (y); 
  \draw[edge,->] (3) -- (b);
\end{scope}

\begin{scope}[shift = {( 4.5, -4) }]
\node at (1,-0.6) {In the open segment};
\node at (1,-1) {  $](2,2,0), (4,4,2)[$};
  \tangentGraphVarNodes 
  \node [hyp] (2) at (2,1) {\hypTwo}; 
  \node [hyp] (3) at (1,0) {\hypThree};

  \draw[edge,->] (y) -- (2); 
  \draw[edge,->] (2) -- (x);

  \draw[edge,->] (z) -- (3); 
  \draw[edge,->] (3) -- (y);
\end{scope}

\begin{scope}[shift = {( 7.5, -4) }]
\node at (1,-0.6) {At  $(4,4,2)$};
\tangentGraphVarNodes
\node [hyp] (1) at (1,1) {\hypOne};
\node [hyp] (2) at (2,1) {\hypTwo};
\node [hyp] (3) at (1,0) {\hypThree};

\draw[edge,->] (b) -- (1);
\draw[edge,->] (1) -- (y);

\draw[edge,->] (y) -- (2);
\draw[edge,->] (2) -- (x);

\draw[edge,->] (z) -- (3);
\draw[edge,->] (3) -- (y);
\end{scope}

\end{tikzpicture}
\caption{Tangent digraphs at various points of the tropical cone obtained by homogenization of the tropical polyhedron
   defined by the inequalities~\eqref{eq:running_example}. Hyperplane nodes are rectangles and coordinate nodes are
   circles.} 
  \label{fig:tangent_graphs}%
\end{figure}
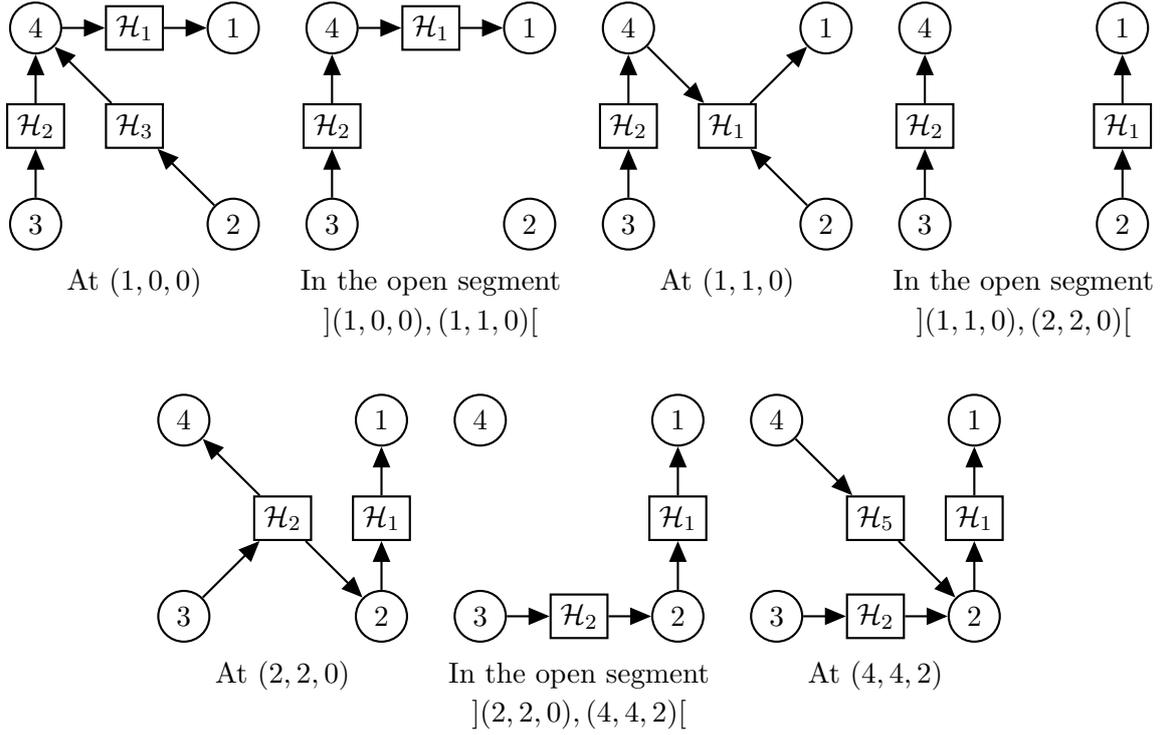

\begin{example}
Let $W$ be the matrix formed by the coefficients of the system~\eqref{eq:homogenized_running_example}, and consider the point 
$x = (1,0,0,0)$ (corresponding to $(1,0,0)$ via the bijection~\eqref{eq:homogenization}).  The inequalities (\hypTwo), (\hypThree) and (\hypFour) are tight at $x$. They read 
\begin{align*}
\max( \underline{x_4}, x_2-1) & \geq \max(\underline{x_1-1}, x_3-1) \\
\underline{x_3} & \geq \max(\underline{x_4}, x_2-2) \\
\underline{x_2} & \geq \underline{x_4}
\end{align*}
where we marked the positions where the maxima are attained. The tangent digraph $\digraph_x(W)$ is depicted in the top left of Figure~\ref{fig:tangent_graphs}. 
For instance, the first inequality provides the arcs from coordinate node $4$ to hyperplane node \hypTwo, and from \hypTwo\ to coordinate node $1$. 
\end{example}

If $I$ and $J$ are respectively subsets of the hyperplane and coordinate nodes of $\graph_x$, a \emph{matching} between $I$ and $J$ is a subgraph of $\graph_x$ with node set $I \cup J$ in which every node is incident to exactly one edge.
\begin{lemma}
  \label{lemma:matching_in_graph_is_max_permutation}
  Let $\epm \in \strop^{m \times (n+1)} $ and $x \in \trop^{n+1}$ be a point with no $\zero$ entries.  Suppose the tangent graph $\graph_x$ contains a matching
  between the hyperplane nodes $I$ and the coordinate nodes $J$. Then the submatrix $\epm'$ of $\epm$ formed from rows $I$ and columns $J$ is such that $|\epm'|$ has a finite
  tropical permanent.  Moreover, the matching yields a maximizing permutation in the latter tropical permanent.
\end{lemma}
\begin{proof}
  Let $\{(i_1, j_1), \dots, (i_q, j_q)\}$ be  a matching between the hyperplanes nodes $I = \{i_1, \dots, i_q \}$ and the coordinate nodes $J= \{ j_1, \dots, j_q \}$. By definition of the tangent graph, 
 for all $p \in [q]$, we have:
  \[
  \begin{aligned}
    |\ep_{i_p j_p}| + x_{j_p} &\geq |\ep_{i_p l}| + x_{l} &\text{ for all } l \in [n+1] \ .
  \end{aligned}
  \]
 Since $x$ has no $\zero$ entries, this implies  $\sum_{p = 1}^q |\ep_{i_p j_p}| \geq  \sum_{p = 1}^q |\ep_{i_p \sigma(i_p)}|$ for any bijection $\sigma: I \rightarrow J$. Thus the tropical permanent of $|\epm'|$ is $\sum_{p = 1}^q |\ep_{i_p j_p}|$, which is obtained with the bijection $i_p \mapsto j_p$.

Consider a $p \in [q]$. By definition of the tangent graph $|\epr_{i_p}| \tdot x > \zero$. Moreover, we suppose that $x$ has finite entries. Thus $|\ep_{i_p j_p}| > \zero$. As a consequence, $\sum_{p = 1}^q |\ep_{i_p j_p}|$ is finite.
\end{proof}

\begin{lemma}
  \label{lemma:graph_of_types_is_cycle_free}\label{lemma:directed_cycle_sign_singularity}
Let $\epm \in \strop^{m \times (n+1)} $ and $x \in \trop^{n+1}$ be a point with no $\zero$ entries.
   If the tangent graph $\graph_x$ contains an  undirected cycle, then the matrix $\epm$ contains a square submatrix $\epm'$ such that $|\epm'|$ is tropically singular and $\tper(|\epm'|) > \zero$ . 
 Moreover, if the cycle is directed in the tangent digraph $\digraph_x$, then $\epm'$ is tropically sign singular.
\end{lemma}
\begin{proof}
   To prove the first statement, let $j_1,i_1,j_2, \dots, i_q, j_{q+1} = j_1$ be an undirected cycle in $\graph_x$. By Assumption \ref{assumption_B} we have $q \geq 2$.  Up to restricting to a subcycle, we may assume that the  cycle is simple, \ie\ the indices $i_1, \dots, i_q$ and $j_1, \dots, j_q$ are pair-wise distinct.  As a consequence,  the maps $\sigma: i_p \mapsto j_p$ and $ \tau: i_p \mapsto j_{p+1}$ for $p\in[q]$ are bijections. The sets of edges $\{(i_p, j_p) \mid p \in [q] \} $ and $\{(i_p, j_{p+1}) \mid p \in [q] \} $ are two distinct matchings between the hyperplane nodes $i_1, \dots, i_p$ and the coordinate nodes $j_1, \dots,  j_p$. Let $\epm'$ be the submatrix of $\epm$ formed from rows $i_1, \dots, i_q$ and columns $j_1, \dots, j_q$.  By Lemma~\ref{lemma:matching_in_graph_is_max_permutation},  
the bijections $\sigma$ and $ \tau$  are both  maximizing in $\tper(|\epm'|)$. Lemma~\ref{lemma:matching_in_graph_is_max_permutation} also shows that  $\tper(|\epm'|) > \zero$.

Now suppose that the cycle is directed. Then, $\ep_{i_p j_p}$ is tropically positive and $\ep_{i_p j_{p+1}}$ is tropically negative for all $p \in [q]$. Consequently, the tropical signs of $\ep_{i_1 j_1} \stimes \cdots \stimes \ep_{i_q j_q}$ and $\ep_{i_1 j_2} \stimes \cdots \stimes \ep_{i_q j_{q+1}}$ differ by $(-1)^q$.  Moreover, $\tau$ is obtained from $\sigma$ by a cyclic permutation of  order $q$, so their signs differs by $(-1)^{q+1}$. As a result, the terms $\tsign(\sigma)\stimes \ep_{i_1 j_1} \stimes  \cdots \stimes \ep_{i_q j_{q}}$ and $\tsign(\tau) \stimes \ep_{i_1 j_2} \stimes \cdots \stimes \ep_{i_q j_{q+1}}$ have opposite tropical signs. This completes the proof.
\end{proof}

\subsection{Cells of  an arrangement of signed tropical hyperplanes}

Our next result shows how the tangent digraph can be used to get sufficient control on the lift of the points in a tropical polyhedron to
points in some Puiseux polyhedron, while dealing mostly with inequality descriptions.  Throughout this section we assume that the extended
matrix $(A \ b)$ is tropically sign generic.

A tropical polyhedron $\tropP(A,b)$ is always the image under the valuation map of a polyhedron over Puiseux series. Indeed,  consider an internal representation $\tconv(V) \tplus \tcone(R)$ of $\tropP(A,b)$ (which exists by  Theorem~\ref{thm:trop_minkowski_weyl}).  
The result of Develin \& Yu, \cite[Proposition~2.1]{DevelinYu07}, implies that any lift $\bm V, \bm R$ of the sets $V, R$ provides a Puiseux polyhedron $\puiseuxP = \conv(\bm V) + \cone(\bm R)$ such that $\val(\puiseuxP )= \tropP(A,b)$.

One can also lift a tropical polyhedron through its inequality representation. For example for any $A \in \strop^{m \times n}$ and $b \in \strop^m$ satisfying Assumption~\ref{assumption_C}, Proposition~\ref{prop:puiseux_solves_tropical_program} provides Puiseux matrices $\A \in \sval^{-1}(A)$ and $\b \in \sval^{-1}(b)$ such that $\val(\puiseuxP(\A,\b)) = \tropP(A,b)$. However, the latter equality may fail for arbitrary lifts  $\A \in \sval^{-1}(A)$ and $\b \in \sval^{-1}(b)$.

\begin{example}
  Consider the tropical polyhedron
  \begin{equation}\label{eq:trop_poly_val_fail}
  \tropP=  \{ x \in \trop^2 \mid  \max (0, x_2) \geq x_1, \;  \max (0, x_1) \geq x_2 , \; x_1 \geq \zero, \; x_2 \geq \zero\} \ .
  \end{equation}
One possible lift of this representation in terms of inequalities is the Puiseux polyhedron
\begin{equation}\label{eq:puiseux_poly_val_fail}
  \puiseuxP=  \{ \x \in \puiseux^2 \mid  2 + \x_2 \geq 2 \x_1, \;  2 + \x_1  \geq 2 \x_2, \; \x_1 \geq 0, \; \x_2 \geq 0\} \ .
\end{equation}
Since $\puiseuxP$ is contained in the non-negative orthant, we have $\val(\puiseuxP) \subseteq \tropP$.  Here this inclusion is strict. The
set $\val(\puiseuxP)$ consists of all the non-positive points in $x \in \trop^2$ with $x_1 \leq 0$ and $x_2 \leq 0$, while $\tropP$
additionally contains the half-line $\{ (\lambda, \lambda) \mid \lambda > 0 \}$.  To show this, suppose that there exists $(\x_1, \x_2) \in
\puiseuxP$ such that $\val(\x_1) = \val(\x_2) = \lambda > 0$. Let $u_1 t^\lambda$ and $u_2 t^\lambda$ be the leading terms of $\x_1$ and
$\x_2$ respectively. Then the inequality $2 + \x_1 \geq 2 \x_2$ implies that $u_1 \geq 2 u_2$, while $2 + \x_2 \geq 2 \x_1$ imposes that
$u_2 \geq 2 u_1$, and we obtain a contradiction as $u_1,u_2>0$.  See Figure~\ref{fig:val_fail}.
\end{example}

\begin{figure}
\begin{tikzpicture}
      \colorlet{myblue}{rgb:red,0;green,102;blue,102}
      \colorlet{mygreen}{rgb:red,10;green,50;blue,10}

\begin{scope}[shift={(0,1.5)}]
  \fill[lightgray] (-0.2,-0.2) -- (-0.2,1) -- (1,1) -- (1,-0.2) -- cycle;

  \draw[myblue] (-0.2,1) -- (1,1) -- (2,2);
  \fill[pattern color =myblue, pattern = north west lines] 
  (-0.2,1) -- (1,1) -- (2,2) -- (2, 1.8) -- (1,0.8) -- (-0.2,0.8) -- cycle;

  \draw[mygreen] (1,-0.2) -- (1,1) -- (2,2);
  \fill[pattern color =mygreen, pattern = north west lines] 
  (1,-0.2) -- (1,1) -- (2,2) -- (2 , 2.2) -- (0.8, 1) -- (0.8, -0.2) -- cycle;

  \draw[gray,->] (-0.2,0) -- (2,0);
  \draw[gray,->] (0,-0.2) -- (0,2);
  \draw (1,0.1) -- (1,-0.1);
  \node[anchor = north] at (1,-0.1) {$0$};
  \node[anchor = north] at (2,0) {$x_1$};
  \draw (0.1,1) -- (-0.1,1);
  \node[anchor = east] at (-0.1,1) {$0$};
  \node[anchor = east] at (0,2) {$x_2$};

\end{scope}

  \begin{scope}[shift={(5,1.5)}]

  \fill[lightgray] (0,0) -- (0,1) -- (2,2) -- (1,0) -- cycle;

  \draw[myblue] (-0.2,0.9) -- (2.6,2.3);
  \fill[pattern color =myblue, pattern = north west lines] 
  (-0.2,0.9) -- (2.6,2.3)  -- (2.6, 2.2) -- (-0.2,0.8)  -- cycle;

  \draw[mygreen] (0.9,-0.2) -- (2.3,2.6);
  \fill[pattern color =mygreen, pattern = north west lines] 
  (0.9, -0.2) -- (2.3, 2.6) -- (2.3, 2.7) -- (0.8,-0.2) -- cycle;

  \draw[gray,->] (-0.2,0) -- (2.6, 0);
  \draw[gray,->] (0,-0.2) -- (0, 2.6);
  \draw (1,0.1) -- (1,-0.1);
  \node[anchor = north] at (1,-0.1) {$t^0$};
  \node[anchor = north] at (2.6,0) {$\x_1$};
  \draw (0.1,1) -- (-0.1,1);
  \node[anchor = east] at (-0.1,1) {$t^0$};
  \node[anchor = east] at (0,2.6) {$\x_2$};  

  \coordinate (puiseux_center) at (1,1);
  
\begin{scope}[shift={(5,0)}]

    \fill[lightgray] (-0.2,-0.2) -- (-0.2,1) --(1,1) -- (1,-0.2) -- cycle;
  \draw (-0.2,1) -- (1,1) -- (1,-0.2);

  \draw[gray,->] (-0.2,0) -- (2,0);
  \draw[gray,->] (0,-0.2) -- (0,2);
  \draw (1,0.1) -- (1,-0.1);
  \node[anchor = north] at (1,-0.1) {$0$};
  \node[anchor = north] at (2,0) {$x_1$};
  \draw (0.1,1) -- (-0.1,1);
  \node[anchor = east] at (-0.1,1) {$0$};
  \node[anchor = east] at (0,2) {$x_2$};

  \coordinate (val_center) at (0.5,0.5);
\end{scope}

\end{scope}

\end{tikzpicture}
\caption{Left: the tropical polyhedron $\tropP$ described in~\eqref{eq:trop_poly_val_fail}; middle: the Puiseux polyhedron $\puiseuxP$ obtained by lifting the inequality representation of $\tropP$ as in~\eqref{eq:puiseux_poly_val_fail}; right: the set $\val(\puiseuxP)$, which is stricly contained in $\tropP$.}
\label{fig:val_fail}
\end{figure}
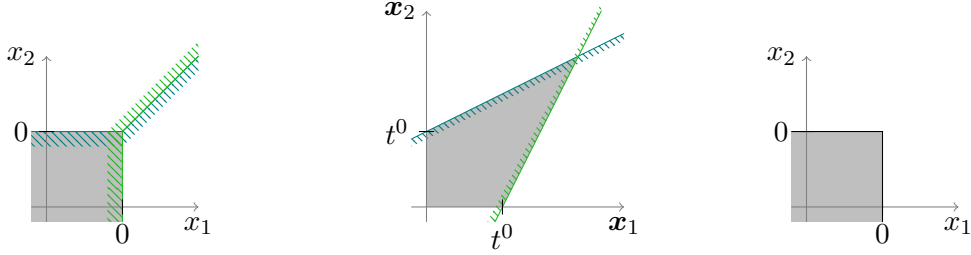

\begin{theorem}\label{thm:val_and_inter_commute_for_generic_matrices}
Suppose that $(A \  b)$ is tropically sign generic. Then the identity 
\[\val \Bigl( \puiseuxP(\A,\b) \cap \puiseux^n_+ \Bigr) \ = \ \tropP(A,b) \]
 holds for any $\A \in \sval^{-1} (A)$ and $\b \in  \sval^{-1} (b)$.
\end{theorem}
\begin{proof} 
  Let $\epm = (A \ b) $. For any $\A \in \sval^{-1} (A)$ and $\b \in \sval^{-1} (b)$, let $\pepm = (\A \ \b) $. We first prove the result for the
  cones $\tropC = \tropP(\epm, \zero)$ and $\puiseuxC = \puiseuxP(\pepm, 0)$. The inclusion $\val(\puiseuxC \cap \puiseux^{n+1}_+) \subset \tropC$ is
  trivial. Conversely, let $x \in \tropC$. Up to removing the columns $j$ of $\epm$ with $x_j = \zero$, we can assume that $x$ has no $\zero$ entries.
  We construct a lift $\x$ of $x$ in the Puiseux cone $\puiseuxC \cap \puiseux^{n+1}_+$ using the tangent digraph $\digraph_x$ with hyperplane node
  set $I$. We claim that it is sufficient to find a vector $v \in \R^{n+1}$ satisfying the following conditions:
  \begin{align}
    \sum_{ j \in \argmax (|\epr_{i}| \tdot x) } \lc(\pep_{ij})  v_j & > 0 \quad \text{for all} \ i \in I \ , \label{eq:val_and_inter_commute_for_generic_matrices1} \\
    v_j & > 0 \quad \text{for all}\  j \in [n+1] \ , \label{eq:val_and_inter_commute_for_generic_matrices2}
  \end{align}
  where $\pepm = (\pep_{ij})$.

Indeed, given such a vector $v$, consider the lift $\x = (v_j t^{-x_j})_j$ of $x$. Clearly $\x \in  \puiseux^{n+1}_+$. If $i \in I$,  then~\eqref{eq:val_and_inter_commute_for_generic_matrices1} ensures that the leading coefficient of $\pepr_i \x$ is  positive.  If $i \not \in I$, two cases can occur. Either $\epr^+_i \tdot x = \epr^-_i \tdot x = \zero$ and thus  $\pepr_i \x = 0$. Otherwise, $\epr^+_i \tdot x > \epr^-_i \tdot x $, so the leading term of $\pepr_i \x $ is positive.  We  conclude that $\pepr_i \x \geq 0$ for all $i \in [m]$. This proves the claim.

  Let $F = (f_{ij}) \in \R^{I \times (n+1)}$ be the real matrix defined by $f_{ij} = \lc(\pep_{ij})$ when $j \in \argmax(|\epr_i| \tdot  x)$ and $f_{ij}=0$ otherwise. We claim that there exists $v \in \R^{n+1}$ such that $F v > 0$ and $v > 0$, or,  equivalently, that the following polyhedron is not empty:
  \begin{equation*}
    \{ v \in \R^{n+1} \mid F v \geq 1, \ v \geq 1 \} \ .
  \end{equation*}
By contradiction, suppose that the latter polyhedron is empty. Then, by Farkas' lemma \cite[\S5.4]{Schrijver03:CO_A},  there exists $\alpha \in  \R^{I}_+$ and $\lambda \in \R^{n+1}_+$ such that:
  \begin{align}
    \transpose{F} \alpha + \lambda & \leq 0 \label{eq:coeff_empty_farkas1}\\
    \sum_{i \in I} \alpha_i +\sum_{j \in [n+1]} \lambda_j & > 0\label{eq:coeff_empty_farkas5}
  \end{align}
  Note that if $\alpha $ is the $0$ vector, then by~\eqref{eq:coeff_empty_farkas5}, there exists a $\lambda_j > 0$ for some $j \in [n+1]$, which contradicts~\eqref{eq:coeff_empty_farkas1}. Thus, the set $K = \{ i \in I \mid \alpha_i > 0 \}$ is not  empty. Let $J \subset [n+1]$ be defined by:
  \begin{equation*}
    J := \bigcup_{i \in K} \argmax(\epr^+_i \tdot x) = \bigcup_{i \in K} \{ j  \mid f_{ij} > 0 \} \ .
  \end{equation*}
  By definition of the tangent digraph, every hyperplane node in $K$ has an incoming arc from a coordinate node in
  $J$. Moreover, for every $j \in J$, the inequality~\eqref{eq:coeff_empty_farkas1} yields:
  \[
  \sum_{i \in I} f_{ij} \alpha_i \leq 0 \ . 
  \]
  This sum contains a positive term $f_{ij} \alpha_i$ (by definition of $J$). Consequently, it must also contain a  negative term $f_{kj} \alpha_k$. Equivalently, $k \in K$ and $f_{kj} < 0$, which means that the coordinate node $j$  has an incoming arc from the hyperplane node $k$. It follows that the tangent digraph $\digraph_x$ contains a  directed cycle (through nodes  $K \cup J$). Then, by Lemma~\ref{lemma:directed_cycle_sign_singularity}, the matrix $\epm$  contains a tropically sign singular submatrix with $\tper(|\epm|) > \zero$. This proves the claim.

  Now we consider the polyhedron $\tropP(A,b)$.  The inclusion $\val(\puiseuxP(\A,\b) \cap \puiseux^n_+) \subset \tropP(A,b)$ is still  valid. Conversely, given $x \in \tropP(A,b)$, the point $x' = (x,\unit) \in \trop^{n+1}$ belongs to the cone  $\tropC$. By the previous proof, there exists a lift $\x'$ of $x'$ in $\puiseuxC \cap \puiseux^{n+1}_+$. Since $\val(\x'_{n+1}) = \unit$,  the point $\x = (\x'_1/ \x'_{n+1},\dots,\x'_n / \x'_{n+1})$ is well-defined.  Furthermore, $\x$ clearly satisfies  $\val(\x) = x$ and it belongs to $\puiseuxP(\A,\b) \cap \puiseux^n_+$.
\end{proof}

Theorem~\ref{thm:val_and_inter_commute_for_generic_matrices} shows that valuation commutes with intersection for half-spaces in general position. This is still true for mixed intersection of half-spaces and (s-)hyperplanes. Similar to our notation for Puiseux polyhedra, we let 
\[ \tropP_I(A,b) :=\bigcap_{i \in I} \tropH(\pr_i, b_i) \cap \tropP(A,b)
\enspace .
\]

\begin{corollary}\label{coro:val_and_inter_commute_for_generic_matrices}  
Suppose that $(A \ b)$ is tropically sign generic. 
Then, for all $\A \in \sval^{-1}(A)$, $\b \in \sval^{-1}(b)$ and $\basis \subset [m]$,
\begin{equation}
  \val\Bigl(  \puiseuxP_\basis(\A,\b) \cap \puiseux^n_+ \Bigr)  \ = \  \tropP_\basis(A,b) \,.
  \label{eq:tropical_faces}
\end{equation}
\end{corollary}

\begin{proof}
  We first prove the result when $\basis = [m]$.  In this case, the claim is about the intersection of all (Puiseux or
  signed tropical) hyperplanes in the arrangement.
  The first inclusion $\val \bigl( \bigcap_{i = 1}^m \puiseuxH(\ppr_i, \b_i) \cap \puiseux^n_+ \bigr) \subset \bigcap_{i=1}^m \tropH(\pr_i, b_i)$ is
  trivial. Conversely, let $x \in \bigcap_{i=1}^m \tropH(\pr_i, b_i)$.  Note that there is nothing to prove if that intersection is empty. The point
  $x$ belongs to the tropical polyhedron $\tropP(A,b)$. By Theorem~\ref{thm:val_and_inter_commute_for_generic_matrices}, $x$ admits a lift in the
  Puiseux polyhedron $\puiseuxP(\A,\b) \cap \puiseux^n_+$. But observe that the choice of tropical signs for the rows of $(A \ b)$ is
  arbitrary. Indeed, if $(A' \ b')$ is obtained by multiplying some rows of $(A\ b)$ by $\tminus \unit$, then $(A' \ b')$ satisfies the conditions of
  Theorem~\ref{thm:val_and_inter_commute_for_generic_matrices} and $x$ belongs to $\tropP(A',b')$. Thus, for any sign pattern $s \in \{-1, +1\}^m$,
  there exists a lift $\x^s$ of $x$ which belongs to the Puiseux polyhedron $\puiseuxP(\A^s, \b^s) \cap \puiseux^n_+$, where $(\A^s \ \b^s) =
  \Bigl(\begin{smallmatrix} s_1 & & \\[-1.5ex] & \ddots & \\[-0.5ex] & & s_m\end{smallmatrix}\Bigr) (\A \ \b)$.

  Since the Puiseux points $\x^s$ are non-negative with valuation $x$, any point in their convex hull is also non-negative with valuation $x$. We claim that the convex hull $\conv \{ \x^s \mid s \in \{-1, +1\}^m \}$ contains a point in the intersection $\bigcap_{i =1}^m \puiseuxH(\ppr_i, \b_i)$. We prove the claim by induction on the number $m$ of  hyperplanes.

  If $m = 1$, we obtain two points $\x^+$ and $\x^-$ on each side of the hyperplane $\puiseuxH(\ppr_1,\b_1)$, and it is  easy to see that their convex hull intersects the hyperplane.
  Now, suppose we have $m \geq 2$ hyperplanes. Let $S^+$ (resp.\ $S^-$) be the set of all signs patterns $s \in \{-1, +1  \}^m$ with $s_m =+1$ (resp.\ $s_m=-1$). By induction, the convex hull $\conv \{\x^s \mid s \in S^+ \}$ contains a point  $\x^+$ in the intersection of the first $m-1$ hyperplanes $\bigcap_{i =1}^{m-1} \puiseuxH(\ppr_i, \b_i)$. Similarly,  $\conv \{\x^s \mid s \in S^-\}$ contains a point $\x^-$ in $\bigcap_{i =1}^{m-1} \puiseuxH(\ppr_i, \b_i)$. The points  $\x^+$ and $\x^-$ are on opposite sides of the last hyperplane $\puiseuxH(\ppr_m, \b_m)$, thus their convex hull  intersects $\puiseuxH(\ppr_m, \b_m)$.

  When $\basis \subsetneq [m]$, the previous proof can be generalized by considering only the sign patterns $s \in \{- 1,
  +1\}^m$ such that $s_i = +1$ for all $i \not \in \basis$.
\end{proof}

By Corollary~\ref{coro:val_and_inter_commute_for_generic_matrices}, 
the intersection of the non-negative orthant $\puiseux^n_+$ with the cells of the arrangement of Puiseux hyperplanes
$\{\puiseuxH(\ppr_i,\b_i) \}_{i \in [m]}$ induces a cellular decomposition of $\trop^n$ into tropical polyhedra.  We call this 
collection of tropical polyhedra the \emph{signed cells} of the arrangement of tropical s-hyperplanes $\{ \tropH(\pr_i,b_i) \}_{i
  \in [m]}$. Notice that the signed cells form an intersection poset thanks to Corollary
\ref{coro:val_and_inter_commute_for_generic_matrices}. 
    
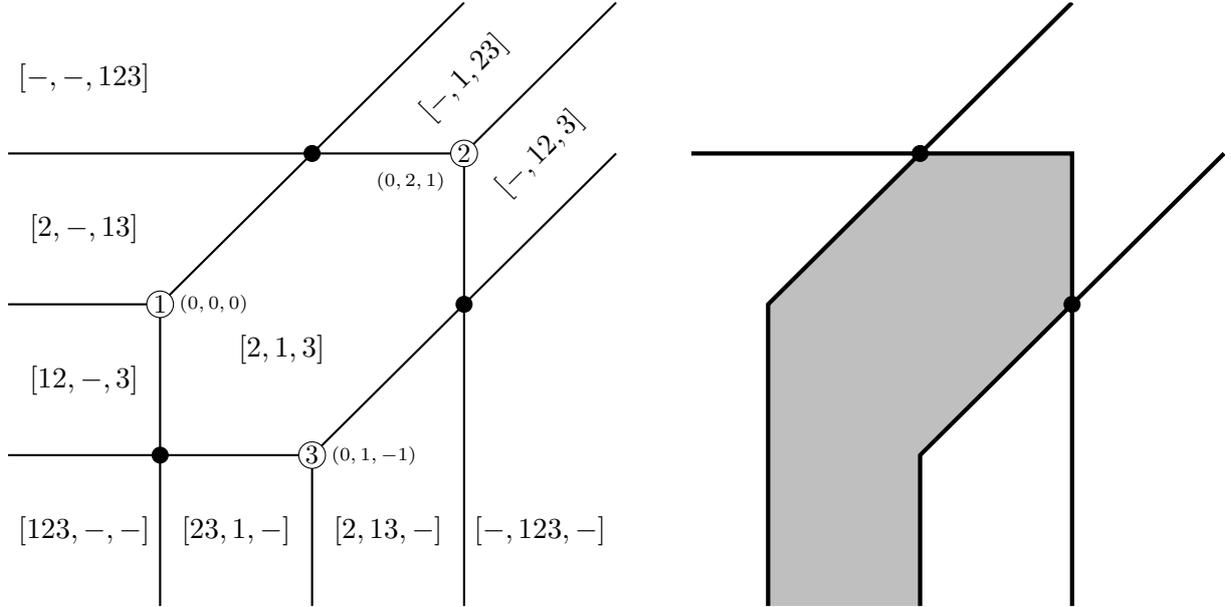
\begin{figure}     
   % H1234.tikz

\begin{tikzpicture}[x  = {(2cm,0cm)},
                    y  = {(0cm,2cm)},
                    scale = 1]

\begin{scope}[shift={(0,0)}]

  \definecolor{pointcolor}{rgb}{ 0,0,0 }
  \tikzstyle{linestyle} = [color=black, thick]

  \coordinate (p1) at (0,0);
  \coordinate (p2) at (2,1);
  \coordinate (p3) at (1,-1);

  \coordinate (e11) at (2,2);
  \coordinate (e12) at (-1,0);
  \coordinate (e13) at (0,-2);

  \coordinate (e21) at (3,2);
  \coordinate (e22) at (-1,1);
  \coordinate (e23) at (2,-2);

  \coordinate (e31) at (3,1);
  \coordinate (e32) at (-1,-1);
  \coordinate (e33) at (1,-2);

  \draw[linestyle] (e13) -- (p1) -- (e11);
  \draw[linestyle] (e12) -- (p1);

  \draw[linestyle] (e22) -- (p2) -- (e23);
  \draw[linestyle] (e21) -- (p2);

  \draw[linestyle] (e32) -- (p3) -- (e31);
  \draw[linestyle] (e33) -- (p3);

  \filldraw[draw=pointcolor,fill=white] (p1) circle (5 pt);
  \filldraw[draw=pointcolor,fill=white] (p2) circle (5 pt);
  \filldraw[draw=pointcolor,fill=white] (p3) circle (5 pt);

  \filldraw[pointcolor] (0,-1) circle (3 pt);
  \filldraw[pointcolor] (1,1) circle (3 pt);
  \filldraw[pointcolor] (2,0) circle (3 pt);

  \node at (p1) {\small{$1$}};
  \node at (p2) {\small{$2$}};
  \node at (p3) {\small{$3$}};

  \node at (p1) [inner sep=7pt, right, black] {\tiny{$(0,0,0)$}};
  \node at (p2) [inner sep=7pt, below left, black] {\tiny{$(0,2,1)$}};
  \node at (p3) [inner sep=7pt, right, black] {\tiny{$(0,1,-1)$}};

  \node at (0.8,-0.3) {$[2,1,3]$};
  \node at (-0.5,1.5) {$[-,-,123]$};
  \node at (-0.5,0.5) {$[2,-,13]$};
  \node at (-0.5,-0.5) {$[12,-,3]$};
  \node at (-0.5,-1.5) {$[123,-,-]$};
  \node at (0.5,-1.5) {$[23,1,-]$};
  \node at (1.5,-1.5) {$[2,13,-]$};
  \node at (2.5,-1.5) {$[-,123,-]$};

  \node [rotate=45] at (2,1.5) {$[-,1,23]$};
  \node [rotate=45] at (2.5,1) {$[-,12,3]$};
\end{scope}

\begin{scope}[shift={(4,0)}]
\clip (-0.5,-2) rectangle (3,2);

  \definecolor{pointcolor}{rgb}{ 0,0,0 }
  \tikzstyle{linestyle} = [color=black, ultra thick]

  \coordinate (p1) at (0,0);
  \coordinate (p2) at (2,1);
  \coordinate (p3) at (1,-1);

  \coordinate (e11) at (2,2);
  \coordinate (e12) at (-1,0);
  \coordinate (e13) at (0,-2);

  \coordinate (e21) at (3,2);
  \coordinate (e22) at (-1,1);
  \coordinate (e23) at (2,-2);

  \coordinate (e31) at (3,1);
  \coordinate (e32) at (-1,-1);
  \coordinate (e33) at (1,-2);

  \fill[lightgray] (e33) -- (p3) -- (2,0) -- (p2) -- (1,1) -- (p1) -- (e13) -- cycle;

  \draw[linestyle] (e13) -- (p1) -- (e11);

  \draw[linestyle] (e22) -- (p2) -- (e23);

  \draw[linestyle] (e33) -- (p3) -- (e31);

  \filldraw[pointcolor] (1,1) circle (3 pt);
  \filldraw[pointcolor] (2,0) circle (3 pt);

\end{scope}
\end{tikzpicture}

  \caption{Unsigned (left) and signed (right) cell decompositions induced by the three tropical s-hyperplanes in Example~\ref{ex:cells}.}
  \label{fig:signedcell}
\end{figure}

The signed cell decomposition coarsens the cell decomposition introduced in \cite{develin2004}, which partitions
$\trop^n$ into ordinary polyhedra.  Here we call the latter cells \emph{unsigned}. In particular, the one dimensional signed cells are
unions of (closed)  one-dimensional unsigned cells. However, some  one-dimensional unsigned cells may not belong to any one dimensional signed cell.
 In the example depicted in Figure~\ref{fig:tropical_polytope}, this is the case for the ordinary line
segment $[(1,0,1),(1,1,1)]$. 

\begin{example}\label{ex:cells}
  Consider the tropical polyhedral cone $\tropC$ in $\trop^3$ given by the three homogenous constraints
  \begin{align}
    x_2 &\geq \max(x_1,x_3) \label{eq:signed:1}\\
    x_1 &\geq \max(x_2-2,x_3-1)\\
    \max(x_1,x_3+1) & \geq x_2-1 \enspace .
  \end{align}
  This gives rise to an arrangement of three tropical s-hyperplanes in which $\tropC$ forms one signed cell; see
  Figure~\ref{fig:signedcell} (right) for a visualization in the $x_1=0$ plane.  Each tropical s-hyperplane yields a unique
  unsigned tropical hyperplane.  An \emph{open sector} is one connected component of the complement of an unsigned tropical
  hyperplane.  The ordinary polyhedral complex arising from intersecting the open sectors of an arrangement of unsigned tropical
  hyperplanes is the \emph{type decomposition} of Develin and Sturmfels \cite{develin2004}.  In our example the type decomposition
  has ten unsigned maximal cells; in Figure~\ref{fig:signedcell} (left), we marked them with labels as in \cite{develin2004}.

  The \emph{apices} of the unsigned tropical hyperplanes arising from the three constraints above are $p_1=(0,0,0)$, $p_2=(0,2,1)$
  and $p_3=(0,1,-1)$.  The tropical convex hull of $p_1$, $p_2$ and $p_3$, with respect to $\min$ as the tropical addition, is the
  topological closure of the unsigend bounded cell $[2,1,3]$.

  The signed cell $C$ is precisely the union of the two maximal unsigned cells $[2,1,3]$ and $[23,1,-]$ together with the
  (relatively open) bounded edge of type $[23,1,3]$ sitting in-between.  The other signed cells come about by replacing ``$\geq$''
  by ``$\leq$'' in some subset of the constraints above.  For instance, exchanging ``$\geq$'' by ``$\leq$'' in \eqref{eq:signed:1}
  and keeping the other two yields the signed cell which is the union of the three unsigned cells $[2,-,13]$, $[12,-,3]$,
  $[123,-,-]$ and two (relatively open) edges in-between.  Altogether there are six maximal signed cells in this case.
\end{example}

The proper notion of a ``face'' of a tropical polyhedron is a subject of active research, see \cite{Joswig05} and
\cite{DevelinYu07}.  Notice that, like the tangent digraphs, the signed and unsigned cells depend on the arrangement of
s-hyperplanes, while several different arrangements may describe the same tropical polyhedron. For example,
\begin{equation}\label{eq:ex_basic_point_not_extreme}
 \{ x \in \trop^2 \mid x_1 \tplus x_2 \leq 1 \} = \{ x \in \trop^2 \mid x_1 \leq 1 \text{ and } x_2 \leq 1 \} \,.
\end{equation}
Even if a canonical external representation exists, see \cite{AllamigeonKatzJCTA2013}, it may not  satisfy the
genericity conditions of Corollary \ref{coro:val_and_inter_commute_for_generic_matrices}.  Thus this approach does not
easily lead to a meaningful notion of faces for tropical polyhedra.

Lacking a good notion of a ``face'', the following two results introduce suitable concepts which are good enough for our algorithms.

\begin{prop_def}[Tropical basic points] \label{prop:def_basic_points} Suppose that $(A \ b)$ is tropically sign generic and that
  Assumption~\ref{assumption_C} holds.  Let $\basis$ be a subset of $[m]$ of cardinality $n$ such that $\tper(|\pr_\basis|) > \zero$.  If the set
  $\tropP_\basis(A,b)$ is not empty, it contains a unique point, called a \emph{(feasible) tropical basic point} of $\tropP(A,b)$.
 
  The \emph{tropical basic points} of $\tropP(A,b)$ are exactly the valuations of the basic points of $\puiseuxP(\A,\b)$ for any
  lift $(\A \ \b) \in \sval^{-1} (A \ b)$.
\end{prop_def}
\begin{proof}
  This is a straightforward consequence of Corollary \ref{coro:val_and_inter_commute_for_generic_matrices} and the
  definition of basic points. Assumption \ref{assumption_C} ensures that the sets of the form
  $\puiseuxP_I(\A,\b)$, for any $I \subset [m]$, are contained in~$\puiseux^n_+$.
\end{proof}
\begin{remark}
Alternatively, the fact that 
 $\tropP_\basis(A,b)$, if it is not empty, contains
a unique element, follows from the Cramer theorem in the symmetrized tropical semiring~\cite{akian1990linear}, see Theorem~\ref{thm:tropical_cramer} below. 
This also implies that the technical condition that $x_j>\zero$ in Assumption~\ref{assumption_C} is not needed to derive the uniqueness of this element. 
\end{remark}

\begin{prop_def}[Tropical edges] \label{prop:def_edges} Suppose that $(A \ b)$ is tropically sign generic and that Assumption~\ref{assumption_C} holds.  Let $K$ be a subset of $[m]$ of cardinality $n-1$ such that $A_K$ has a  tropically sign non-singular maximal minor. If the set $\tropP_K(A,b)$ is not empty then it is called a  \emph{(feasible) tropical edge}. 

 The \emph{tropical edges} of $\tropP(A,b)$ are exactly the valuation of the edges of  $\puiseuxP(\A,\b)$ for any lift $(\A \ \b) \in \sval^{-1} (A \ b)$.
\end{prop_def}
\begin{proof}
  The arguments are the same as in the proof of Proposition-Definition~\ref{prop:def_basic_points}.
\end{proof}
The correspondence between basic points and edges with their tropical counterparts is illustrated in Figures \ref{fig:tropical_polytope} and
\ref{fig:puiseux_polytope} where basic points are depicted by red dots and edges by black lines.
These definitions are meaningful only if $(A \ b)$ is tropically sign generic. Otherwise, $\tropP(A,b)$ may have no basic points in the sense of Proposition-Definition~\ref{prop:def_basic_points}. For instance the set $\{ x \in \trop^2 \mid x_1 \geq x_2 \tplus 0 \text{ and } x_2 \geq x_1 \tplus 0 \}$ does not contain such point.
Notice that our genericity assumptions imply that the tropical edges arise as complete intersections of tropical halfspaces.  In this sense
Corollary~\ref{coro:val_and_inter_commute_for_generic_matrices} is a signed version of~\cite[Prop.~6.3]{SpeyerSturmfels04}.

A point $x$ in a tropically convex set $S$ is called an \emph{extreme point} of $S$ if, for any $y, z \in S$, $x \in \tconv(\{y,z\})$ implies $x=y$ or $x=z$.
\begin{proposition}\label{prop:extreme_basic}
Suppose that $(A \ b)$ is tropically sign generic and that Assumption \ref{assumption_C} holds. 
Then the extreme points of $\tropP(A,b)$ are tropical basic points.
\end{proposition}
\begin{proof}
  Consider any lift $(\A \ \b) \in \sval^{-1} (A \ b)$. Then $\puiseuxP(\A,\b) \subset  \puiseux^n_+$ by Assumption~\ref{assumption_C}, and $\val(\puiseuxP(\A,\b)) = \tropP(A,b) $ by
  Corollary~\ref{coro:val_and_inter_commute_for_generic_matrices}. The basic points of $\puiseuxP(\A, \b)$ are precisely its extreme points. As a result, we have $\puiseuxP(\A, \b) = \conv(\P) + \cone(\RR)$, where $\P$ is the set of basic points, $\conv(\P)$ its convex hull, and $\cone(\RR)$ is a pointed polyhedral cone generated by a finite set $\RR \subset \puiseux^n$. Note that $\RR \subset \puiseux^n_+$ as $\puiseuxP(\A,\b) \subset  \puiseux^n_+$. Thus, by~\cite[Proposition~2.1]{DevelinYu07}, we know that $\tropP(A,b) = \tconv(\val(\P)) \tplus \tcone(\val(\RR))$. By Proposition~\ref{prop:def_basic_points}, $\val(\P)$ is precisely the set of tropical basic points of $\tropP(A,b)$.
The tropical analog of Milman's converse of the Krein-Milman theorem, which is proved for instance in Theorem~2 of~\cite{GaubertKatz2011minimal} in the case of polyhedra, implies that the set of extreme points of $\tropP(A,b)$ is included in $\val(\P)$.
It follows that every extreme point is basic.
\end{proof}

\begin{remark}\label{rk:therearefewerextremepoints}
  While extreme points are basic points, the converse does not hold. For example, $(1,1)$ is a basic point of the tropical polyhedron $\{ x
  \in \trop^2 \mid x_1 \leq 1 \text{ and } x_2 \leq 1 \}$, but it is not extreme.

The polars of sign cyclic polyhedral cones studied
in~\cite{AllamigeonGaubertKatzJCTA2011} are also examples in which not
every basic point is extreme. Actually, Theorems~2 and~6 in that
reference provide combinatorial characterizations of extreme and basic
points in terms of lattice paths. A comparison of both
characterizations shows that the lattice paths are more constrained in
the case of extreme points.
\end{remark}

\section{Pivoting between two tropical basic points} \label{sec:pivot}

\noindent
In this section, we show how to pivot from a tropical basic point to another, \ie\ to move along a tropical edge between
the two points, in a tropical polyhedron $\tropP(A,b)$, where $A \in \strop^{m \times n}$ and $b \in \strop^m$. Under some
genericity conditions, by Corollary~\ref{coro:val_and_inter_commute_for_generic_matrices} this is equivalent to the same
pivoting operation in an arbitrary lift $\puiseuxP(\A,\b)$ over Puiseux series of the considered tropical polyhedron,
where $\A \in \sval^{-1} (A)$ and $\b \in \sval^{-1}(b)$.  However, our method only relies on the tropical matrix $A$
and the tropical vector $b$. The complexity of this tropical pivot operation will be shown to be $O(n(m + n))$, which is
analogous to the classical pivot operation.

Pivoting is more easily described in homogeneous terms.  For $\epm=(A \ b)$ we consider the tropical cone $\tropC =
\tropP(\epm, \zero)$, seen as a subset of the tropical projective space $\tropProj^n$. This cone is defined as the
intersection of the half-spaces $\tropH^{\geq}_i := \{ x \in \tropProj^n \mid \epr_i^+ \tdot x \geq \epr_i^- \tdot x\}$
for $i \in [m]$. Similarly, we denote by $\tropH_i$ the s-hyperplane $\{ x \in \tropProj^n \mid \epr_i^+ \tdot x =\epr_i^-
\tdot x \}$. We also let $\tropC_I:=\tropP_I(\epm,\zero)$ for any subset $I\subset[m]$.

Throughout this section, we make the following assumptions.
\begin{assumption}\label{ass:general_position} 
  The matrix $\epm$ is tropically generic.
\end{assumption}
\begin{assumption}\label{ass:finite_coordinates} 
  Every point in $\tropC \setminus \{ (\zero, \dots, \zero) \}$ has finite coordinates.
\end{assumption}

Assumption~\ref{ass:general_position} is a tropical version of primal non-degeneracy.  It is strictly stronger than the condition that
$\epm=(A \ b)$ is tropically sign generic used in the previous section, and hence, in particular, we can make use of
Corollary~\ref{coro:val_and_inter_commute_for_generic_matrices}.  Under Assumption~\ref{ass:finite_coordinates}, which is strictly stronger
than Assumption~\ref{assumption_C}, the tropical polyhedron $\tropP(A,b)$ is a bounded subset of $\R^n$: To see this, consider $\tropC$ as a
subset of $\trop^{n+1}$. As $\tropC$ is a closed set, Assumption~\ref{ass:finite_coordinates} implies that there exists a vector $\ell \in
\R^{n+1}$ such that $x \geq \ell$ for all $x \in \tropC$.  Let $\tconv(P) \tplus \tcone(R)$ be the internal description of $\tropP(A,b)$
provided by Theorem~\ref{thm:trop_minkowski_weyl}.  If $R$ contains a point $r$, then it is easy to verify that $(r, \zero)$ lies in $\tropC$,
which contradicts Assumption~\ref{ass:finite_coordinates}.  Since every $p \in P$ belongs to $\tropP(A,b)$, the point $(p, \unit) $ belongs
to $\tropC$, and thus $p_j \geq l_j$ for all $j \in [n]$.  It follows that $\tropP(A,b) = \tconv(P)$ is a bounded subset of $\R^n$.

The Assumptions \ref{ass:general_position} and \ref{ass:finite_coordinates} are two out of three conditions required for a tropical linear
program to be standard in the sense of Theorem~\ref{th-main}.

As a consequence, the Puiseux polyhedron $\puiseuxP(\A, \b)$ is also bounded and contained in the interior of $\puiseux^n_+$. 

Through the bijection given in~\eqref{eq:homogenization}, the tropical basic point associated with a suitable subset $\basis \subset [m]$ is identified with the unique projective point $x^\basis \in \tropProj^n$ in the intersection $\tropC_\basis$. Besides, when pivoting from the basic point $x^\basis$, we move along a tropical edge $\tropE_K :=\tropC_K$ defined by a set $K = \basis \setminus \{ \ileaving \} $ for some $\ileaving \in \basis$. 

By Propositions~\ref{prop:def_basic_points} and~\ref{prop:def_edges}, a tropical edge $\tropE_K$ is a tropical line segment $\tconv(x^\basis, x^{\basis'})$. The other endpoint $x^{\basis'} \in \tropProj^n$ is a basic point for $\basis' = K \cup \{\ient \}$, where $\ient \in [m] \setminus \basis$. So, the notation
$\ileaving$ and $\ient$ refers to the indices leaving and entering
the set of active constraints $\basis$ which is maintained
by the algorithm. Notice that the latter set
corresponds to the {\em non-basic} indices in the classical
primal simplex method, so that the indices entering/leaving $\basis$
correspond to the indices leaving/entering the usual basis, respectively.

As a tropical line segment, $\tropE_K$ is known to be the concatenation of at most $n$ ordinary line segments.
\begin{proposition}[{\cite[Proposition~3]{develin2004}}] \label{prop:tropical_edge}
Let $\tropE_K = \tconv(x^\basis, x^{\basis'})$ be a tropical edge. Then there exists an integer $q \in [n]$ and $q+1$ points $\xi^1,  \dots, \xi^{q+1} \in \tropE_K$ such that 
\[
\tropE_K = [\xi^1, \xi^2]  \cup \dots \cup [\xi^{q},\xi^{q+1}] \qquad \text{ where } \xi^1 = x^\basis \text{ and } \xi^{q+1} = x^{\basis'} \ .
\]
Every ordinary segment  is of the form:
\begin{equation}
 [\xi^p, \xi^{p+1}] = \{ x^p + \lambda e^{J_p} \mid 0 \leq \lambda \leq \mu_p \} \ , \label{eq:ordinary_segment}
\end{equation}
where the length of the segment $\mu_p$ is a positive real number, $ J_p \subset [n+1]$, and the $j$th coordinate of the 
vector
 $e^{J_p}$ is equal to $1$ if $j \in J_p$, and to $0$ otherwise.
Moreover, the sequence of subsets $J_1, \dots, J_{q}$ satisfies:
\[ \emptyset \subsetneq J_1 \subsetneq \dots \subsetneq J_{q} \subsetneq [n+1] \ . \]
\end{proposition}

The vector $e^{J_p}$ is called the \emph{direction} of the segment $[\xi^p, \xi^{p+1}]$. The intermediate points $\xi^2, \dots, \xi^q$ are called \emph{breakpoints}.  In the tropical polyhedron depicted in Figure \ref{fig:tropical_polytope}, breakpoints are represented by white dots.

Note that, in the tropical projective space $\tropProj^n$, the directions $e^{J}$ and $-e^{[n+1] \setminus J}$ coincide. Both correspond to the direction of $\trop^n$ obtained by 
removing the $(n+1)$th coordinate of either $-e^{[n+1] \setminus J}$ if $(n+1) \in J$, or $e^{J}$ otherwise.

\subsection{Overview of the pivoting algorithm}

We now provide a sketch of the pivoting operation along a tropical edge $\tropE_K$. Geometrically, the idea is to traverse the ordinary segments $[\xi^1, \xi^2],   \dots , [\xi^{q},\xi^{q+1}]$ of $\tropE_K$. At each point $\xi^p$, for $p \in [q]$, we first determine the direction vector $e^{J_q}$, then move along this direction until the point $\xi^{p+1}$ is reached. As the tangent digraph at a point $x \in \tropC$ encodes the local geometry of the tropical cone $\tropC$ around $x$, the directions vectors can be read from the tangent digraphs. Moreover, the tangent digraphs are acyclic under Assumption~\ref{ass:general_position}. This imposes strong combinatorial conditions on the tangent digraph, which, in turn, allows to easily determine the feasible directions.

For the sake of simplicity, let us suppose that the tropical edge consists of two consecutive segments $[\xi, \xi']$ and $[\xi', \xi'']$, with direction vectors $e^J$ and $e^{J'}$ respectively.

Starting at the basic point $\xi = x^{K \cup\{ \ileaving \}}$, we shall prove below that, at every basic point, the tangent digraph is spanning tree where every hyperplane node has exactly one incoming arc and one outgoing arc. In other words, for every $i \in K \cup \{\ileaving\}$, the sets $\argmax(W^+_i \tdot \xi) $ and $\argmax(W^-_i \tdot \xi)$ are both reduced to a singleton, say $\{j^+_i \}$ and $\{j^-_i\}$.
We want to ``get away'' from the s-hyperplane $\tropH_{\ileaving}$. Since the direction vector $e^J$ is a $0/1$ vector, the only way to do so is to increase the variable indexed by $j^+_{\ileaving}$ while not increasing the component indexed by $j^-_{\ileaving}$. Hence, we must have $j^+_{\ileaving} \in J$ and  $j^-_{\ileaving} \not \in J$.
 While moving along $e^J$, we also want to stay inside the s-hyperplane $\tropH_i$ for $i \in K$. Hence, if  $j^+_i \in J$ for some $i \in K$, we must also have $j^-_i \in J$. Similarly,  if $j^+_i \not \in J$, then we must also have $j^-_i \not \in J$. 
Removing the hyperplane node $\ileaving$ from the tangent digraph $\digraph_{\xi}$ provides two connected components, the first one, $C_+$, contains  $j^+_{\ileaving}$, and the second one, $C_-$ contains  $j^-_{\ileaving}$. From the discussion above, it follows that the set $J$ consists of the coordinate nodes in $C^+$.

When moving along $e^J$ from $\xi$, we leave the s-hyperplane $\tropH_{\ileaving}$. Consequently, the hyperplane node $\ileaving$ ``disappears'' from the tangent digraph. It turns out that this is the only modification that happens to the tangent digraph. More precisely, at every point in the open segment $\oiv{\xi, \xi'}$, the tangent digraph is the graph obtained from $\digraph_{\xi}$ by removing the hyperplane node $\ileaving$ and its two incident arcs. We shall denote this digraph by $\digraph_{\oiv{\xi, \xi'}}$. By construction, $\digraph_{\oiv{\xi, \xi'}}$ is acyclic, consists of two connected components, and every hyperplane node has one incoming and one outgoing arc.

We shall move  from $\xi$ along $e^J$ until ``something'' happens to the tangent digraph. In fact only two things can happens, depending whether $\xi'$ is a breakpoint or a basic point. As we supposed $\xi'$ to be a breakpoint, a new arc $a_\ent$  will ``appear'' in the tangent digraph, \ie\ $\digraph_{\xi'} = \digraph_{\oiv{\xi, \xi'}} \cup\{ a_\ent\}$.  Let us sketch how the arc $a_\ent$ can be found. We denote $a_\ent = (j_\ent, k)$, where $j_\ent$ is a coordinate node and $k\in K$ is a hyperplane node.  We shall see that $j_\ent$ must belongs to $J$ (\ie\ the component $C_+$), while $k$ must belongs to the component $C_-$.

Hence, the arc $a_\ent$ ``reconnects'' the two components $C_+$ and $C_-$ (see Figure~\ref{fig:tangent_graph_update}). Since $k$ had one incoming and one outgoing arc in  $\digraph_{\oiv{\xi, \xi'}}$, it has exactly three incident arcs in $\digraph_{\xi'}$. One of them is $a_\ent =(j_\ent, k)$; a second one, $a_\leaving = (j_\leaving, k)$, has the same orientation as $a_\ent$; and the third one, $a'=(k,l)$, has an orientation opposite to $a_\ent$ and $a_\leaving$.

Let us now find the direction vector $e^{J'}$ of the second segment $[\xi', \xi'']$. Consider the hyperplane node $k$ with the three incidents arcs
$a_\ent, a_\leaving$ and $a'$. By Proposition~\ref{prop:tropical_edge}, we know that $J \subset J'$, hence we must increase the variable
$j_\ent$. Since we want to stay inside the hyperplane $\tropH_k$, it follows that we must also increase the variable indexed by $\ell$. On the other hand, we do not increase the variable $j_\leaving$.
As before, all hyperplane nodes $i \in K \setminus\{k\}$ have exactly one incoming and one outgoing arc.
 Removing the arc $a_\leaving$ from the graph provides two connected component, the first one $C_+'$ contains the coordinate nodes $j_\ent$ and $\ell$ as well as the hyperplane node $k$, while the second one $C_-'$ contains $j_\leaving$. The new direction set $J'$ is given by the coordinate nodes in $C_+'$. 

The tangent digraph in the open segment $\oiv{\xi', \xi''}$ is again constant, and defined by  $\digraph_{\oiv{\xi', \xi''}} = \digraph_{\xi'} \setminus\{a_\leaving\}$. Hence, $\digraph_{\oiv{\xi', \xi''}}$ is an acyclic graph, with two connected components $C'_+$ and $C'_-$, where every hyperplane node has one incoming and one outgoing arc.

The basic point $\xi''$ is reached when a new s-hyperplane $\ient \not \in K$ is hit. This happens when the hyperplane node $\ient$ ``appears'' in the tangent digraph, along with one incoming $(j^+, \ient)$ and one outgoing arc $(\ient,j^-)$.  Observe that we must have $j^- \in J$ and $j^+ \not \in J$. It follows that the two components $C'_+$ and $C'_-$ are reconnected by adding $\ient$ and its two incident arcs.

In Section~\ref{sec:directions}, we prove that the tangent digraph satisfy the above-mentioned characterization, and that they provides the feasible directions. In Section~\ref{sec:lenght_segment}, we characterize the lengths of the ordinary segments, that is we deduce when an arc or an hyperplane node ``appears'' in the tangent digraphs. In Section~\ref{subsec:efficient_pivot}, we prove that the tangent digraphs evolves as described above. It allows us to incrementally update the information needed to find the directions and lengths of the segments. This will finally provide an efficient implementation of the pivoting operation.

\subsection{Directions of ordinary segments }\label{sec:directions}

Given a point $x$ in a tropical cone $\tropD$, we say that the direction $e^J$, with $\emptyset \subsetneq J \subsetneq [n+1]$, is \emph{feasible} from $x$ in $\tropD$ if there exists $\mu > 0$ such that the ordinary segment $\{ x+ \lambda e^J  \mid 0 \leq \lambda \leq \mu \}$ is included in $\tropD$. 
The following lemma will be helpful to prove the feasibility of a direction.

\begin{lemma} \label{lemma:feasible_directions}
Let $x \in \R^{n+1}$. Then, the following properties hold:
\begin{enumerate}[(i)]
\item if $x$ belongs to $\tropH^{\geq}_i\setminus \tropH_i$,  every direction is feasible from $x$ in $\tropH^{\geq}_i$.
\item if $x$ belongs to $\tropH_i$, the direction $e^J$ is feasible from $x$ in the half-space $\tropH^{\geq}_i$ if, and only if, $\argmax(\epr^+_i \tdot x) \cap J \neq \emptyset$ or $\argmax(\epr^-_i \tdot x) \cap J = \emptyset$.
\item if $x$ belongs to $\tropH_i$, the direction $e^J$ is feasible from $x$ in the s-hyperplane $\tropH_i$ if, and only if, the sets $\argmax(\epr^+_i \tdot x) \cap J $ and $\argmax(\epr^-_i \tdot x) \cap J $ are both empty or both non-empty.
\end{enumerate}
\end{lemma}
\begin{proof}
The first point is immediate. 
To prove the last two points, observe that if $x \in \tropH_i$, then $\epr^+_i \tdot x = \epr^-_i \tdot x > \zero$, thanks to $x \in \R^{n+1}$ and  Assumption~\ref{assumption_A}. Then, for $\lambda > 0$ sufficiently small, we have:
\begin{equation*}
\epr^+_i \tdot (x+ \lambda  e^J) = 
\begin{cases}
(\epr^+_i \tdot x) + \lambda & \text{if} \ \argmax(\epr^+_i \tdot x) \cap J \neq \emptyset \ , \\
\epr^+_i \tdot x  & \text{otherwise} \ ,
\end{cases}
\end{equation*}
and the same property holds for $\epr^-_i \tdot x$.
\end{proof}

We propose to determine feasible directions with tangent graphs. It turns out that tangent graphs in a tropical edge have a very special structure.
Indeed, under Assumption~\ref{ass:general_position}, these graphs do not contain any cycle by
Lemma~\ref{lemma:graph_of_types_is_cycle_free}. In other words, they are forests: each connected component
is a tree.  Hence we have the equality
\begin{equation}
\text{number of nodes} = \text{number of edges} + \text{number of connected components} \,.
\label{eq:acyclic_graph}
\end{equation}
Note that \cite[Proposition 17]{develin2004} determines that number of connected components.

We introduce some additional basic notions and notations on directed graphs. Two nodes of a digraph are said to be \emph{weakly connected} if they are connected in the underlying undirected graph.
 Given a directed graph $\digraph$ and a set $\mathcal{A}$ of arcs between some nodes of $\digraph$, we denote by $\digraph \cup \mathcal{A}$ the digraph obtained by
adding the arcs of $\mathcal{A}$. Similarly, if $\mathcal{A}$ is a subset of arcs of $\digraph$, we denote by $\digraph
\setminus \mathcal{A}$ the digraph where the arcs of $\mathcal{A}$ have been removed. By extension, if $\mathcal{N}$ is
a subset of nodes of $\digraph$, then $\digraph \setminus \mathcal{N}$ is defined as the digraph obtained by removing
the nodes in $\mathcal{N}$ and their incident arcs. 
The \emph{degree} of a node of $\digraph$ is defined as the pair $(p_1, p_2)$, where $p_1$ and $p_2$ are the numbers of incoming and outgoing arcs incident to
the node.

\begin{proposition}\label{prop:tangent_graph_interior_edge}
Let $x$ be a point in a tropical edge $\tropE_K$. Then, exactly one of the following  cases arises:
\begin{asparaenum}[(C1)]
\item\label{item:shape1} $x$ is a basic point for the basis $K \cup \{\ileaving \}$, where $\ileaving \in [m] \setminus K$. The tangent graph $\graph_x$ at $x$ is a spanning tree, and the set of hyperplane nodes is $K \cup \{ \ileaving \}$. In the tangent digraph $\digraph_x$, every hyperplane node has degree $(1,1)$. 
Let $J$ be the set of coordinate nodes weakly connected to the unique node in $\argmax(\epr_{\ileaving}^+ \tdot x)$ in the digraph $\digraph_x \setminus \{\ileaving\}$. The only feasible direction from $x$ in $\tropE_K$ is $e^J$.

\item\label{item:shape2} $x$ is in the relative interior 
 of an ordinary segment. The tangent graph $\graph_x$ is a forest with two connected components, and the set of hyperplane nodes is $K$. In the tangent digraph $\digraph_x$, every hyperplane node has degree $(1,1)$.
 Let $J$ be the set of coordinate nodes in one of the components. The two feasible directions from $x$ in $\tropE_K$ are $e^J$ and $-e^J = e^{[n+1] \setminus J}$.

\item\label{item:shape3} $x$ is a breakpoint. The tangent graph $\graph_x$ is a spanning tree, and the set of hyperplane nodes is $K$. In the tangent digraph $\digraph_x$, there is exactly one hyperplane node $\ibreak$ with degree $(2,1)$ or $(1,2)$, while all other hyperplane nodes have degree $(1,1)$. 

Let $a$ and $a'$ be the two arcs incident to $\ibreak$ with same orientation. 
Let $J$ and $J'$ be the set of coordinate nodes weakly connected to $k$ in $\digraph_x \setminus \{a\}$ and $\digraph_x \setminus \{a'\}$, respectively. 
The two feasible directions from $x$ in $\tropE_K$ are $e^{J}$ and $e^{J'}$.
\end{asparaenum}
\end{proposition}

\begin{proof}
  Since $x$  has finite entries, the graph $\graph_x$  contains exactly $n+1$ coordinate nodes. Let $n'$ be the number of
  hyperplane nodes in $\graph_x$. Consider any $i \in K$. Since $x$ is contained in the s-hyperplane $\tropH_i$ and $x \in \R^{n+1}$,
  we have $\epr_i^+ \tdot x = \epr_i^- \tdot x > \zero$. Thus $K$ is contained in the set of hyperplane nodes. Therefore $n' \geq
  n-1$. As there is at least one connected component, there is at most $n+n'$ edges by~\eqref{eq:acyclic_graph}.
  Besides, each hyperplane node is incident to at least two edges, so that there is at least $2n'$ edges in
  $\graph_x$. We deduce that $n' \leq n$. As a result, by using~\eqref{eq:acyclic_graph}, we can distinguish three
  cases:
\begin{enumerate}[(i)]
\item\label{item:tangent_graph_basic_point} $n' = n$, in which case there is only one connected component in $\graph_x$, and exactly $2n$ edges. Besides, all the hyperplane nodes have degree $(1,1)$ in $\digraph_x$.
\item\label{item:tangent_graph_interior_segment} $n' = n-1$, the graph $\graph_x$ contains precisely two connected components and $2n'-2$ edges. As in the previous case, every hyperplane node has degree $(1,1)$ in $\digraph_x$.
\item\label{item:tangent_graph_breakpoint} $n' = n-1$ and $\graph_x$ has one connected component. In this case, there are $2n'-1$ edges. In $\digraph_x$, there is exactly one hyperplane node with degree $(2,1)$ or $(1,2)$, and all the other hyperplane nodes have degree $(1,1)$.
\end{enumerate}
We next show that these cases correspond to the ones in our claim. 

\begin{asparaitem}
\item[\emph{Case~\eqref{item:tangent_graph_basic_point}}:] Since $n' = n$, the set of hyperplanes nodes is of the form $K \cup \{\ileaving\}$ for some $\ileaving \not \in K$. 
Moreover, $\graph_x$ is a spanning tree. As a consequence, it contains a matching between the coordinate nodes $[n]$ and the hyperplanes nodes $K \cup \{\ileaving\}$. Such a matching can be constructed as follows.  Let $\digraph'$ be the digraph obtained by directing the edges of $\graph_x$ towards the coordinate node $n+1$. In this digraph, every coordinate node $ j\in [n]$ has exactly one outgoing arc to a hyperplane node $\sigma(j)$, as there is exactly one path from $j$ to $n+1$ in the spanning tree $\graph_x$. 
Moreover, every hyperplane node $i$ has exactly one incoming arc and one outgoing arc in $\digraph'$. Indeed, $i$ is incident to two arcs in $\digraph'$, and exactly one of them leads to the path to coordinate node $n+1$. We conclude that $\sigma(j) \neq \sigma(j')$ when $j \neq j'$. Thus the set of edges $\{(j, \sigma(j))\mid j \in [n] \}$ forms the desired matching.
Then, by Lemma~\ref{lemma:matching_in_graph_is_max_permutation}, the submatrix $\epm'$ of $\epm$ formed from the columns in $[n]$ and the rows in $K \cup \{\ileaving\}$ satisfy $\tper(|\epm'|) > \zero$. Furthermore, $\epm' = A_{K \cup \{\ileaving \}}$. As a consequence, $x$ is a basic point for the set $K \cup \{\ileaving \}$.

Since the graph $\graph_x$ is a spanning tree where the hyperplane node $\ileaving$ is not a leaf, 
removing $\ileaving$ from $\graph_x$ provides two connected components $C^+$ and $C^-$, containing the coordinate nodes in $\argmax(\epr_{\ileaving}^+ \tdot x)$ and
in $\argmax(\epr_{\ileaving}^- \tdot x)$, respectively. Let $J$ be the set of the coordinate nodes in $C^+$.

We claim that the direction $e^J$ is feasible from $x$ in $\tropE_K$. Indeed, if the hyperplane node $i \in K$ belongs to $C^+$,  then $\argmax(\epr^+_i \tdot x) \subset J$ and $\argmax(\epr^-_i \tdot x) \subset J$. In contrast, if the node $i \in K$ belongs to $C^-$,  we have
$\argmax(\epr^+_i \tdot x) \cap J = \argmax(\epr^-_i \tdot x) \cap J = \emptyset$. By Lemma~\ref{lemma:feasible_directions}, this shows that the direction $e^J$ is feasible in all s-hyperplanes  $\tropH_i$ with $i \in  K$. 
It is also feasible in the half-space $\tropH^\geq_{\ileaving}$, since $x \in \tropH_{\ileaving}$ and $\argmax(\epr^+_{\ileaving} \tdot x) \subset J$. Finally, for all $i \not \in K \cup \{\ileaving\}$, the point $x$ belongs $\tropH^\geq_i \setminus \tropH_i$. Indeed, if $x \in \tropH_i$,  then $i$ would be a hyperplane node. Thus, by Lemma~\ref{lemma:feasible_directions}, the direction $e^J$ is feasible in $\tropH^\geq_i$. As $\tropE_K = (\cap_{i \in K} \tropH_i) \cap (\cap_{i \not \in K} \tropH^\geq_i)$, this proves the claim.

Since $x$ is a basic point it admits exactly one feasible direction in $\tropE_K$. Thus $e^J$ is the only feasible direction from $x$ in $\tropE_K$.

\item[\emph{Case~\eqref{item:tangent_graph_interior_segment}}:] In this case, $\graph_x$ is a forest with two components $C_1$ and $C_2$, and $K$ is precisely the set of hyperplane nodes. Let $J$ be the set of coordinate nodes in $C_1$.
Then Lemma~\ref{lemma:feasible_directions} shows that the direction $e^J$ is feasible from $x$ in $\tropE_K$. Indeed, the point $x$ belongs to $\tropH^{\geq}_i\setminus \tropH_i$ for $i \not \in K$. Besides, for all $i \in K$, the sets $\argmax(\epr^+_i \tdot x) \cap J$ and $\argmax(\epr^-_i \tdot x) \cap J$ are both non-empty if $i$ belongs to $C_1$, and both empty otherwise. 

Symmetrically, the direction $e^{[n+1] \setminus J} = -e^J$ is also feasible in $\tropE_K$, as $[n+1] \setminus J$ is the set of coordinate nodes in the component $C_2$. It follows that $x$ is in the relative interior of an ordinary segment.

\item[\emph{Case~\eqref{item:tangent_graph_breakpoint}}:] The graph $\graph_x$ is a spanning tree. Let $\ibreak$ be the unique half-space node of degree $(2,1)$ or $(1,2)$ in $\digraph_x$ and $a,a'$ the two arcs incident to $\ibreak$ with the same orientation.

Then $\digraph_x \setminus \{a\}$  consists of two weakly connected components $C_1$ and $C_2$. Without loss of generality, we assume that $\ibreak$ belongs to $C_1$. 
Let $J$ be the set of coordinate nodes in $C_1$. We now prove that $e^{J}$ is feasible from $x$ in $\tropE_K$, thanks to Lemma~\ref{lemma:feasible_directions}. Indeed, $x \in \tropH^{\geq}_i\setminus \tropH_i$ for $i \not \in K$. Besides, if $i \in K$, the sets $\argmax(\epr^+_i \tdot x) \cap J$ and $\argmax(\epr^-_i \tdot x) \cap J$ are both non-empty if $i \in C_1$, and both empty if $i \in C_2$. Thus, $e^{J}$ is feasible in the s-hyperplane $\tropH_i$. 

Similarly, let $J'$ be the set of coordinate nodes weakly connected to $\ibreak$ in $\digraph_x \setminus \{a'\}$. Then the direction $e^{J'}$ is also feasible.
 Note that $J$ and $J'$ are neither equal nor complementary. Thus, there are two distinct and non-opposite directions which are feasible from $x$ in $\tropE_K$, which implies than $x$ is a breakpoint. \qedhere
\end{asparaitem}
\end{proof}

\begin{example} \label{exmp:directions}
  Figure~\ref{fig:tangent_graphs} depicts the tangent digraphs at every point of the tropical edge $\tropE_K$ for
  $K=\{$\ref{eq:running_example_eq2},\ref{eq:running_example_eq3}$\}$, and this illustrates Proposition~\ref{prop:tangent_graph_interior_edge}.
  The set $\basis=\{$\ref{eq:running_example_eq2},\ref{eq:running_example_eq3},\ref{eq:running_example_eq4}$\}$ of constraints determines
 the basic point
  $x^\basis = (1,0,0)$. From its tangent digraph, we deduce that the initial ordinary segment of the edge $\tropE_K$ is directed by  $e^{\{2\}}$. 

  The tangent digraph at a point in $](1,1,0),(1,0,0)[$ has exactly two weakly connected components. They yield the feasible directions $e^{\{2\}}$ and
  $e^{\{1,3,4\}}$, which correspond to the vectors $(0,1,0)$ and $(0,-1,0)$ of $\trop^3$. 
  
At the breakpoint $(1,1,0)$, the tangent digraph is weakly connected, and the hyperplane node~\ref{eq:running_example_eq2} has degree $(2,1)$. Removing the arc from coordinate node $4$ to~\ref{eq:running_example_eq2} provides two weakly connected components, respectively $\{1, 2\} \cup \{\text{\ref{eq:running_example_eq2}}\}$ and $\{3, 4\} \cup \{\text{\ref{eq:running_example_eq3}}\}$. The coordinate nodes of the component containing \ref{eq:running_example_eq2} yields the feasible direction $e^{\{1,2\}}$. Similarly, it can be verified that the other feasible direction, obtained by removing the arc from coordinate node $2$, is the vector $e^{\{1,3,4\}}$.
\end{example}

\subsection{Moving along an ordinary segment} \label{sec:lenght_segment} In this section we provide the exact details on
how to obtain the next point $\xi'$ from a given point $\xi$ and a direction given in terms of the set $J$.  We
determine whether $\xi'$ is a basic point or a break point, and we determine the length $\mu$ of the resulting segment
$[\xi, \xi'] = \{ \xi + \lambda e^J \mid 0 \leq \lambda \leq \mu \}$ of the tropical edge $\tropE_K$.  The metric
results from this section will interpreted in terms of tangent digraph in the next section.

For all $i \in [m]$, we define:
\begin{align*}
\lambda^+_{i}(\xi, J) & :=  (\epr^+_i \tdot \xi) - \max_{j \in J} (\ep^+_{ij} + \xi_j) \ , \\
\lambda^-_{i}(\xi, J) & :=  (\epr^+_i \tdot \xi) - \max_{j \in J}( \ep^-_{ij} + \xi_j) \ ,
\end{align*}
where $\epm = (\ep_{ij})$.  When it is clear from the context, $\lambda^+_{i}(\xi, J)$ and $\lambda^-_{i}(\xi, J)$ will be simply denoted by $\lambda^+_i$ and $\lambda^-_i$.  
By Assumptions~\ref{assumption_A} and~\ref{ass:finite_coordinates}, we have $\epr^+_i \tdot \xi > \zero$. In contrast, $\max_{j \in J} (\ep^+_{ij} + \xi_j)$ and
$\max_{j \in J} (\ep^-_{ij} + \xi_j)$ may be equal to $-\infty$, in which case we use the convention $-(-\infty) = +\infty$, and so $\lambda^+_i = +\infty$ and
$\lambda^-_i = +\infty$, respectively. When $\max_{j \in J} (\ep^+_{ij} + \xi_j )$ and $\max_{j \in J} (\ep^-_{ij} + \xi_j )$ are finite, the scalars $\lambda_i^+ $
and $\lambda_i^-$ are non-negative real numbers. 

Let $x ^{\lambda} := \xi + \lambda e^J$.  Observe that $\lambda^+_i$ is the smallest $\lambda \geq 0$ such that
$\epr^+_i \tdot \xi = \ep^+_{ij} + x^{\lambda}_j$ for some $j \in J$.  Similarly, $\lambda^-$ is the smallest $\lambda
\geq 0$ such that $\epr^+_i \tdot \xi = \ep^-_{ij} + x^{\lambda}_j$ for some $j \in J$.  More precisely, we have
\begin{align*}
  \epr^+_i \tdot x^{\lambda} &= 
  \begin{cases}
    \epr^+_i \tdot \xi & \text{ if } 0  \leq \lambda \leq \lambda^+_i \\
    (\epr^+_i \tdot \xi) + \lambda  - \lambda^+_i   & \text{ if }   \lambda \geq \lambda^+_i 
  \end{cases} \\
\epr^-_i \tdot x^{\lambda} &= 
\begin{cases}
    \epr^-_i \tdot \xi & \text{ if } 0  \leq \lambda \leq \beta^-_i   \\
  (\epr^+_i \tdot \xi) + \lambda - \lambda^-_i       & \text{ if }  \lambda  \geq  \beta^-_i  \enspace ,\\
\end{cases} 
\end{align*}
where $\beta^-_i = \lambda^-_i + (\epr^-_i \tdot \xi) - (\epr^+_i \tdot \xi)$. In particular $\beta^-_i \leq
\lambda^-_i$ and equality holds when $i \in K$.  The evolution of $\epr^+_i \tdot (\xi + \lambda e^J)$ versus $\epr^-_i
\tdot (\xi + \lambda e^J)$ is visualized in Figure~\ref{fig:variation_with_lambda2}.

The endpoint $\xi'$  is either a breakpoint or a basic point. We will prove that it is a basic point if a new hyperplane node $\ient \not \in K$ ``appears'' in the tangent digraph.  In that case the index $\ient $  must belong to the following set:
\[
\enteringHyp(\xi, J) := \{i \in [m] \setminus K \mid \argmax(\epr^+_i \tdot \xi) \cap J = \emptyset \} \ .
\]
We shall see that $\xi'$ is a breakpoint if a hyperplane node $\ibreak \in K$ ``acquires'' a new arc, and thus become of degree $(2,1)$ or $(1,2)$. 
 Such a node $\ibreak$ must be an element of the following set:
\[
\breakHyp(\xi,J) := \{i \in K \mid \argmax(\epr^+_{i} \tdot \xi) \cap J = \emptyset \ \text{and} \ \argmax(\epr^-_{i} \tdot \xi) \cap J = \emptyset \} \ .
\]
We already mentioned that the notation $\ient$ (and so, $\enteringHyp(\xi, J)$) and $\ileaving$ is chosen by analogy
with the entering or leaving indices in the classical simplex method.  Note that the set $\breakHyp(\xi,J)$ does not
have any classical analog. It represents intermediate indices which shall be examined before a leaving index is found.

When this does not bear the risk of confusion, we simply use the notations $\breakHyp$ and $\enteringHyp$.

\begin{proposition}\label{prop:maximal_step}
Let $\{ \xi + \lambda e^J \mid 0 \leq  \lambda \leq \mu \}$ be an ordinary segment of a tropical edge $\tropE_K$. 
The following properties hold:
\begin{enumerate}[(i)]
\item\label{item:maximal_step1} the length $\mu$  of the segment is the greatest scalar $\lambda \geq 0$ satisfying the following conditions:
\begin{equation} \label{eq:maximal_step}
\begin{aligned}
\lambda & \leq \min(\lambda^+_{i}, \lambda^-_{i}) &&\text{for all} \  i \in \breakHyp \ , \\
\lambda & \leq \lambda^-_{i} &&\text{for all} \  i \in  \enteringHyp \ \text{such that} \ \lambda^-_i \leq \lambda^+_i .
\end{aligned}
\end{equation}
\item\label{item:maximal_step3} if  $\mu = \lambda^-_{\ient}$ for $\ient \in \enteringHyp$, then $\xi + \mu e^J$ is a basic point for the basis $K \cup \{\ient\}$.
\item\label{item:maximal_step4} if  $\mu = \min(\lambda^+_k, \lambda^-_k)$ for $\ibreak \in \breakHyp$, then $\xi + \mu e^J$ is a breakpoint. 
\end{enumerate}
\end{proposition}
\begin{proof}
Let $x ^{\lambda} := \xi + \lambda e^J$ for all $\lambda \geq 0$. 

\begin{figure}[t] 
  \centering
  \begin{tikzpicture}
    \draw[->] (0,0) -- (5,0);
    \draw[->] (0,0) -- (0,4);
    
    \draw[very thick, red] (0,2) -- (4,2) -- (5,3);
    \draw (0,-0.1) -- (0,0.1);
    \node[below] at (0,-0.1) {$0$};
    \draw (4,-0.1) -- (4,0.1);
    \node[below] at (4,-0.1) {$\lambda^+_i$};
    \node[below] at (5,-0.1) {$\lambda$};

    \draw[very thick] (0,1) -- (2,1) -- (5,4);
    \draw (2,-0.1) -- (2,0.1);
    \node[below] at (2,-0.1) {$\beta^-_i$};
    \draw (0,-0.1) -- (0,0.1);
    \node[below] at (0,-0.1) {$0$};
    \draw (3,-0.1) -- (3,0.1);
    \node[below] at (3,-0.1) {$\lambda^-_i$};
    \node[below] at (5,-0.1) {$\lambda$};

    \node[anchor = east] at (0,2) {$\epr^+_i \tdot \xi$};
    \node[anchor = east] at (0,1) {$\epr^-_i \tdot \xi$};

    \begin{scope}[shift={( 8,0)}]
    \draw[->] (0,0) -- (5,0);
    \draw[->] (0,0) -- (0,4);
    
    \draw[very thick, red] (0,2) -- (1,2) -- (3,4);
    \draw (0,-0.1) -- (0,0.1);
    \node[below] at (0,-0.1) {$0$};
    \draw (1,-0.1) -- (1,0.1);
    \node[below] at (1,-0.1) {$\lambda^+_i$};
    \node[below] at (5,-0.1) {$\lambda$};

    \draw[very thick] (0,1) -- (2,1) -- (5,4);
    \draw (2,-0.1) -- (2,0.1);
    \node[below] at (2,-0.1) {$\beta^-_i$};
    \draw (0,-0.1) -- (0,0.1);
    \node[below] at (0,-0.1) {$0$};
    \draw (3,-0.1) -- (3,0.1);
    \node[below] at (3,-0.1) {$\lambda^-_i$};
    \node[below] at (5,-0.1) {$\lambda$};

    \node[anchor = east] at (0,2) {$\epr^+_i \tdot \xi$};
    \node[anchor = east] at (0,1) {$\epr^-_i \tdot \xi$};

    \draw[dashdotted] (0,2) -- (5,2);
    \end{scope}
  \end{tikzpicture}
\caption{Evolution of $\epr^+_i \tdot (\xi + \lambda e^J)$ (in red) and $\epr^-_i \tdot (\xi + \lambda e^J)$ (in black) with $\lambda \geq 0$ for $i \in \enteringHyp$ and $\lambda^-_i < \lambda^+_i$ (left) or $\lambda^-_i > \lambda^+_i$ (right).}
\label{fig:variation_with_lambda2}
\end{figure}
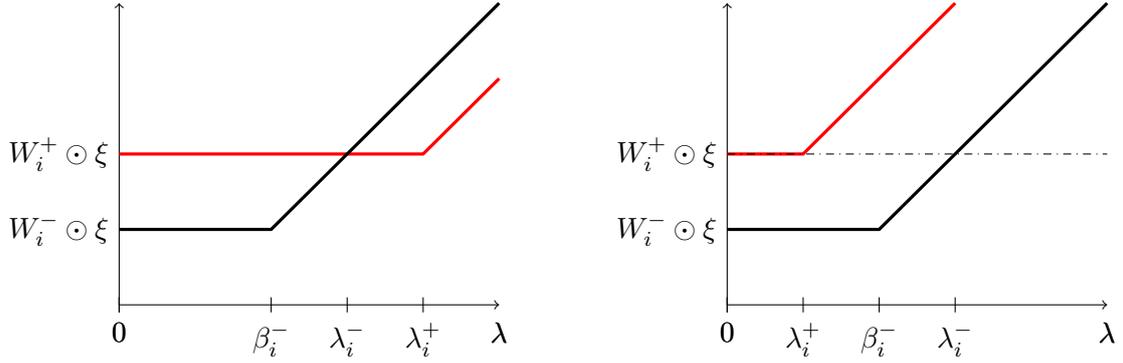

We claim that $x^{\lambda}$ belongs to $\tropE_K$ if $\lambda$ satisfies~\eqref{eq:maximal_step}. 
To that end, we first show that $x^{\lambda} \in \tropH_i$ for $i \in K$.
Consider an  $i \in \breakHyp$. Then $\beta^-_i = \lambda^-_i$. Therefore, for all $0 \leq \lambda  \leq \min(\lambda^+_i, \lambda^-_i)$ we have $x^{\lambda} \in \tropH_i$ since:
\[
 \epr^+_i \tdot x^{\lambda} = \epr^+_i \tdot \xi  = \epr^-_i \tdot \xi  = \epr^-_i \tdot x^{\lambda} \,.
\]
Let $i \in K \setminus \breakHyp$.  Then by Lemma~\ref{lemma:feasible_directions}, $\argmax(\epr^+_i \tdot \xi) \cap J$ and $\argmax(\epr^-_i \tdot \xi) \cap J$ are both non-empty. Thus $\lambda^+_i = \lambda^-_i = \beta^-_i = 0$. Therefore, $x^{\lambda} \in \tropH_i$ for all $\lambda \geq 0$ since in this case:
\[
 \epr^+_i \tdot x^{\lambda} = (\epr^+_i \tdot \xi) + \lambda = \epr^-_i \tdot x^{\lambda} \,.
\]

 We now examine the half-spaces $\tropH^\geq_i$ where $i \in [m] \setminus K$.
If $i \not \in \enteringHyp$ then $\argmax(\epr^+_i \tdot \xi) \cap J \neq \emptyset$. Consequently, $\lambda^+_i = 0$. Thus $x^\lambda \in \tropH^{\geq}_i$ for all $\lambda \geq 0$ as we have:
\[
\epr^+_i \tdot x^{\lambda} = (\epr^+_i \tdot \xi) + \lambda \geq \max(\epr^-_i \tdot \xi, (\epr^+_i \tdot \xi) + \lambda - \lambda^-_i) = \epr^-_i \tdot x^{\lambda} \,.
\]
If $i \in \enteringHyp$ and $ 0 \leq  \lambda \leq \min(\lambda^+_i, \lambda^-_i)$, then $x^\lambda \in \tropH^{\geq}_i$. Indeed :
\[
\epr^+_i \tdot x^{\lambda} = \epr^+_i \tdot \xi  \geq  \max(\epr^-_i \tdot \xi, (\epr^+_i \tdot \xi) + \lambda - \lambda^-_i ) = \epr^-_i \tdot x^{\lambda} \,.
\]
Now if further $\lambda^+_i < \lambda^-_i$, then, for $\lambda \geq \lambda^+_{i}$,
 we have 
\[
\epr^+_i \tdot x^{\lambda}  = (\epr^+_i \tdot \xi) + \lambda - \lambda^+_i  \geq \max(\epr^-_i \tdot \xi,(\epr^+_i \tdot \xi) + \lambda- \lambda^-_i ) = \epr^-_i \tdot x^{\lambda}  \ .
\]
We conclude that if $i \in \enteringHyp$ and $\lambda^+_i < \lambda^-_i$ then $x^\lambda \in \tropH^{\geq}_i$ for all $\lambda \geq 0$.

Second, we claim that the solution set of the inequalities~\eqref{eq:maximal_step} admits a greatest element $\lambda^* \in \R$. By contradiction, suppose that $x^{\lambda} \in \tropE_K$ for all $\lambda \geq 0$. Recall that $e^J$ and $ - e^{[n+1] \setminus J}$ coincide as elements of $\tropProj^n$. Consequently the half-ray $ \{\xi - \lambda e^{[n+1] \setminus J} \mid \lambda \geq 0 \}$ is contained in $\tropE_K$, and thus in $\tropC$. Since $\tropC$ is closed, it contains the point $y\in \trop^{n+1}$ defined by $y_j = \xi_j$ if $j \in J$ and $y_j = \zero$ otherwise. As $J \subsetneq [n+1]$, this contradicts Assumption~\ref{ass:finite_coordinates}.

Third, we claim that $\lambda^* = \mu$. To prove the claim is sufficient to show that $x^{\lambda^*}$ is either a breakpoint or a basic point of $\tropE_K$. We distinguish three cases:
\begin{enumerate}[(a)]
\item\label{item:case1} $\lambda^* = \lambda^-_{i}$ for some $i \in \enteringHyp$. Then $\epr^-_{i} \tdot x^{\lambda^*} = \epr^+_i \tdot \xi$. Moreover $\lambda^* \leq \lambda^+_{i}$ and thus $\epr_{i}^+ \tdot x^{\lambda^*} = \epr_{i}^+ \tdot \xi$.
This implies that $\epr_{i}^+ \tdot x^{\lambda^*} = \epr_{i}^- \tdot x^{\lambda^*} > \zero$.
As a consequence, $ i \not \in K$ is a hyperplane node in the tangent graph $\graph_{x^{\lambda^*}}$. By Proposition~\ref{prop:tangent_graph_interior_edge}, we conclude that $x^{\lambda^*}$ is a basic point for the set $K \cup \{i \}$.

\item $\lambda^* =  \lambda^+_{i} \leq  \lambda^-_{i} $ for some $i \in \breakHyp$. Then, observe that:
\begin{equation}
\argmax(\epr^+_i \tdot x^{\lambda^*} ) = \argmax(\epr^+_i \tdot \xi) \cup \argmax_{j \in J} ( \ep^+_{ij} +\xi_j) \ .\label{eq:maximal_step_proof2}
\end{equation}
The two sets on the right-hand side of~\eqref{eq:maximal_step_proof2} are non-empty and disjoint, since  $i \in \breakHyp$. Thus $\argmax(\epr^+_i \tdot x^{\lambda^*} )$ contains at least two distinct elements . Moreover, $x^{\lambda^*} \in \tropE_K$ by the discussion above, thus $i \in K$ appears as an hyperplane node in $\digraph_{x^{\lambda^*}}$.  Consequently, the hyperplane node $i$ has at least 2 incoming arcs in $\digraph_{x^{\lambda^*}}$. We deduce by Proposition~\ref{prop:tangent_graph_interior_edge} that the degree of the hyperplane node $i$ in $\digraph_{x^{\lambda^*}}$ is $(2,1)$, and that $x^{\lambda^*}$ is a breakpoint. 
\item $\lambda^* =  \lambda^-_{i} \leq  \lambda^+_{i} $ for some $i \in \breakHyp$. By the same argument as above, $\argmax(\epr^-_i \tdot x^{\lambda^*} )$ contains at least two distinct elements. This implies that $x^{\lambda^*}$ is a breakpoint and that the hyperplane node $i$ has degree $(1,2)$.
\end{enumerate}
Note that the arguments above also prove~\eqref{item:maximal_step3} and~\eqref{item:maximal_step4}.
\end{proof}

\begin{remark}\label{remark:degree}
When $\xi + \mu e^J$ is a breakpoint, the proof of Proposition~\ref{prop:maximal_step} ensures that the hyperplane node $\ibreak$ in the tangent digraph $\digraph_{\xi + \mu e^J}$ has degree $(2,1)$ if $\mu=\lambda^+_\ibreak$ or $(1,2)$ if $\mu = \lambda^-_\ibreak$. 
In particular, this proves that $\mu$ is equal to only one scalar among the $\lambda^-_i$, $\lambda^+_k$ and $\lambda^-_k$, where $i \in \enteringHyp$ and $k \in \breakHyp$.
\end{remark}

\begin{example}
  We now have all the ingredients required to perform a tropical pivot. Feasible directions are given by Proposition~\ref{prop:tangent_graph_interior_edge}, while Proposition~\ref{prop:maximal_step} provides the lengths of ordinary segments and the stopping criterion.
  
Let us illustrate this on our running example. We start from the basic point $(4,4,2)$ (\ie\ the point $(4,4,2,0)$ in $\tropProj^3$) given by $\basis = \{ \text{\ref{eq:running_example_eq2}},\text{\ref{eq:running_example_eq3}}, \text{\ref{eq:running_example_eq1}} \}$, and we move along the edge
$\tropE_K$, where $K = \{\text{\ref{eq:running_example_eq2}},\text{\ref{eq:running_example_eq3}}\}$. 
The tangent digraph at $(4,4,2)$ is depicted in the bottom right of Figure~\ref{fig:tangent_graphs}. By Proposition~\ref{prop:tangent_graph_interior_edge}~(C\ref{item:shape1}), the initial direction is $-e^{\{1,2,3\}}$, \ie\ $J = \{ 4\}$. By definition, $\breakHyp$ is formed by the hyperplane nodes which are not adjacent to the coordinate node  $4$ in the tangent digraph. Hence, $\breakHyp =  \{\text{\ref{eq:running_example_eq2}},\text{\ref{eq:running_example_eq3}}\}$. Moreover, in the homogeneous setting, the inequalities~\ref{eq:running_example_eq4} and~\ref{eq:running_example_eq5} read
\begin{align*}
x_2 & \geq x_4 \\
x_1 & \geq \max(x_4, x_2-3) 
\end{align*}
In both of them, the maximum in the left-hand side is reduced to one term, and it does not involve $x_4$. Thus, $\enteringHyp =  \{\text{\ref{eq:running_example_eq4}},\text{\ref{eq:running_example_eq5}}\}$. The reader can verify that:
\begin{align*}
\lambda^+_{\mathref{eq:running_example_eq2}} & = 3 - 0 = 3 & 
\lambda^-_{\mathref{eq:running_example_eq2}} & = 3 - (-\infty) = +\infty \\
\lambda^+_{\mathref{eq:running_example_eq3}} & = 2 - (-\infty) = +\infty & 
\lambda^-_{\mathref{eq:running_example_eq3}} & = 2 - 0 = 2 \\
\lambda^+_{\mathref{eq:running_example_eq4}} & = 4 - (-\infty) = +\infty & 
\lambda^-_{\mathref{eq:running_example_eq4}} & = 4 - 0 = 4 \\
\lambda^+_{\mathref{eq:running_example_eq5}} & = 4 - (-\infty) = +\infty & 
\lambda^-_{\mathref{eq:running_example_eq5}} & = 4 - 0 = 4 
\end{align*}
As a result, the length of the initial ordinary segment is $\mu = 2$, given by $\mu = \lambda^-_{\mathref{eq:running_example_eq3}}  \leq \lambda^+_{\mathref{eq:running_example_eq3}}$. As $\mathref{eq:running_example_eq3} \in \breakHyp$,  the  point $(4,4,2) - 2 e^{\{1,2,3\}} = (2,2,0)$ is a breakpoint. 

The next feasible direction is $-e^{\{1,2\}}$ as $J = \{3,4\}$. We still have  $\enteringHyp =  \{\text{\ref{eq:running_example_eq4}},\text{\ref{eq:running_example_eq5}}\}$ but now  $\breakHyp =  \{\text{\ref{eq:running_example_eq2}}\}$. The length of this ordinary segment is $\mu = 1 = \lambda^+_{\text{\ref{eq:running_example_eq2}}}$.
Consequently, we reach the breakpoint $(1,1,0) = (2,2,0) - 1 e^{\{1,2\}}$, where the next feasible direction,  $-e^{\{2\}}$, is given by $J = \{ 1,3,4 \}$. The set $\breakHyp$ is now empty and $\enteringHyp = \{ \text{\ref{eq:running_example_eq5}}\}$. Clearly, $\mu = 1 = \lambda^-_{  \text{\ref{eq:running_example_eq5}} } $. As $\text{\ref{eq:running_example_eq5}} \in \enteringHyp$, the next endpoint $(1,0,0) = (1,1,0) - 1e^{\{ 2 \}}$ is a basic point. 

\end{example}

\subsection{Efficient implementation of the pivoting operation}\label{subsec:efficient_pivot}
Our implementation of the pivoting operation relies on the incremental update of the tangent digraph along the tropical
edge.  This avoids computing from scratch the tangent digraph at each breakpoint, in which case the time complexity of
the pivoting operation would be naively in $O(n^2 m)$.

In the previous section we described the ``travel'' from a given point $\xi$ into the direction given by $J$ to the next
point, called $\xi'$.  Our key observation is that the tangent digraph is constant in the open segment $\oiv{\xi, \xi'}$
and that it ``acquires'' a new arc or a new hyperplane node when the endpoint $ \xi'$ is reached.  This is made precise
in the lemma below and the subsequent proposition.

\begin{lemma}\label{lemma:tangent_graph_is_constant}
Let $[\xi,\xi'] = \{ \xi + \lambda e^J \mid 0  \leq \lambda \leq \mu \}$ be an ordinary segment of $\tropE_K$. Every point in $\oiv{\xi, \xi'}$ has the same tangent digraph $\digraph_{\oiv{\xi, \xi'}}$, which is equal to the intersection of $\digraph_\xi$ and $\digraph_{\xi'}$.
\end{lemma}

\begin{proof}
Let $x^\lambda := \xi + \lambda e^J$. 
By Proposition~\ref{prop:tangent_graph_interior_edge}, the hyperplane node set of $\digraph_{x^\lambda}$ for $\lambda \in \oiv{0, \mu}$ is equal to $K$. If $\xi$ and $\xi'$ are both basic points, the sets of hyperplane nodes in their tangent digraphs are respectively of the form $K \cup \{ \ileaving \}$ and $K \cup \{ \ient \}$, where $\ileaving, \ient \not \in K$ and $\ileaving \neq \ient$. If one of the two endpoints, say $\xi$, is a breakpoint, the hyperplane node set of its tangent digraph is $K$, while the hyperplane node set of $\digraph_{\xi'}$ contains $K$. In all cases, the intersection of the hyperplane node sets of $\digraph_{\xi}$ and $\digraph_{\xi'}$ is equal to $K$. Moreover, the coordinate node set of $\digraph_x$ is equal to $[n+1]$ for all $x \in [\xi, \xi']$.

Let $i \in \breakHyp$. If $0 < \lambda < \mu$, then in particular $\lambda < \min(\lambda^+_i, \lambda^-_i)$ by Proposition~\ref{prop:maximal_step}. Hence,
\begin{equation}
\argmax(\epr^\pm_i \tdot x^{\lambda}) = \argmax(\epr^\pm_i \tdot \xi) \; . \label{eq:argmax1}
\end{equation}
Besides, $\argmax(\epr^\pm_i \tdot \xi') = \argmax(\epr^\pm_i \tdot x^\mu)$ is a superset of $\argmax(\epr^\pm_i \tdot \xi)$, and the inclusion is strict when $\mu$ is equal to the corresponding scalar $\lambda^+_i$ or $\lambda^-_i$.

Similarly, let $i \in K \setminus \breakHyp$.  By Lemma~\ref{lemma:feasible_directions}, $\argmax(\epr^+_i \tdot \xi) \cap J$ and $\argmax(\epr^-_i \tdot \xi) \cap J$ are both non-empty. Moreover, for all $ \lambda > 0$, we have:
\begin{equation}
\argmax(\epr^\pm_i \tdot x^{\lambda}) = \argmax(\epr^\pm_i \tdot \xi) \cap J \; . \label{eq:argmax2}
\end{equation}
In particular, $\argmax(\epr^\pm_i \tdot \xi') = \argmax(\epr^\pm_i \tdot \xi) \cap J$.

Equations~\eqref{eq:argmax1} and~\eqref{eq:argmax2} ensure that $\argmax(\epr^\pm_i \tdot x^{\lambda}) = \argmax(\epr^\pm_i \tdot \xi) \cap \argmax(\epr^\pm_i \tdot \xi')$ for all $i \in K$ and $\lambda \in \oiv{0, \mu}$. This shows that the arc set of $\digraph_{x^\lambda}$ is precisely the intersection of the arc sets of $\digraph_{\xi}$ and $\digraph_{\xi'}$. 
\end{proof}

\begin{figure}[t]
\begin{center}
\begin{tikzpicture}[myedge/.style={->,thick,shorten <=2pt,shorten >=2pt, >=stealth'},
hyp/.style={draw, thick, minimum size = 3ex, inner sep=2pt},
var/.style={draw, circle,thick, minimum size = 3ex, inner sep=2pt},scale=0.75]

\definecolor{mygreen}{RGB}{93,189,52}
\node[hyp] (h) at (0,0) {$k$};
\node[var] (v1) at (-1,1.8) {};
\node[var] (v2) at (1,1.8) {};
\node[var] (v3) at (0,-1.5) {};
\draw[myedge,->] (v1) -- node[anchor=east] {$a_\leaving$} (h);
\draw[myedge,->] (v2) -- node[anchor=west] {$a_\ent$} (h);
\draw[myedge,->] (h) -- (v3);

\draw[black!70,dashed,thick,rounded corners=0.35cm] ($(v1.south east) + (0.15,-0.4)$) -- ($(v1.west) + (-0.3,-0.3)$) -- (-3,3.5) -- (-0.25,3.5) -- cycle;
\draw[black!70,dashed,thick,rounded corners=0.35cm] ($(v2.south west) + (-0.15,-0.4)$) -- ($(v2.east) + (0.3,-0.3)$) -- (3,3.5) -- (0.25,3.5) -- cycle;
\draw[black!70,dashed,thick,rounded corners=0.35cm] ($(v3.north) + (0,0.5)$) -- (-2,-3.5) -- (2,-3.5) -- cycle;

\node at (0,-4.5) {$\digraph_{\xi'}$};

 \begin{scope}[shift={(-7,0)}]
   \node[hyp] (h) at (0,0) {$k$};
 \node[var] (v1) at (-1,1.8) {};
 \node[var] (v2) at (1,1.8) {};
 \node[var] (v3) at (0,-1.5) {};
 \draw[myedge,->] (v1) -- node[anchor=east] {$a_\leaving$} (h);
 \draw[myedge,->] (h) -- (v3);

 \draw[black!70,dashed,thick,rounded corners=0.35cm] ($(v1.south east) + (0.15,-0.4)$) -- ($(v1.west) + (-0.3,-0.3)$) -- (-3,3.5) -- (-0.25,3.5) -- cycle;
 \draw[black!70,dashed,thick,rounded corners=0.35cm] ($(v2.south west) + (-0.15,-0.4)$) -- ($(v2.east) + (0.3,-0.3)$) -- (3,3.5) -- (0.25,3.5) -- cycle;
 \draw[black!70,dashed,thick,rounded corners=0.35cm] ($(v3.north) + (0,0.5)$) -- (-2,-3.5) -- (2,-3.5) -- cycle;

 \draw[mygreen,very thick,rounded corners=0.55cm] (3.5,3.8) -- (-0.1,3.8) -- ($(h.north east) + (-0.2,0.3)$) -- ($(h.north east) + (1.4,0.3)$) -- cycle;

\node at (0,-4.5) {$\digraph_{\oiv{\xi, \xi'}}$};

\node[mygreen] at (1.5,0.3) {$J$};
 \end{scope}

 \begin{scope}[shift={(+7,0)}]
   \node[hyp] (h) at (0,0) {$k$};
 \node[var] (v1) at (-1,1.8) {};
 \node[var] (v2) at (1,1.8) {};
 \node[var] (v3) at (0,-1.5) {};
 \draw[myedge,->] (v2) -- node[anchor=west] {$a_\ent$} (h);
 \draw[myedge,->] (h) -- (v3);

 \draw[black!70,dashed,thick,rounded corners=0.35cm] ($(v1.south east) + (0.15,-0.4)$) -- ($(v1.west) + (-0.3,-0.3)$) -- (-3,3.5) -- (-0.25,3.5) -- cycle;
 \draw[black!70,dashed,thick,rounded corners=0.35cm] ($(v2.south west) + (-0.15,-0.4)$) -- ($(v2.east) + (0.3,-0.3)$) -- (3,3.5) -- (0.25,3.5) -- cycle;
 \draw[black!70,dashed,thick,rounded corners=0.35cm] ($(v3.north) + (0,0.5)$) -- (-2,-3.5) -- (2,-3.5) -- cycle;

 \draw[orange,very thick,rounded corners=0.6cm] (3.2,3.8) -- (0.1,3.8) -- ($(h.north west) + (-0.1,0)$) --  (-2.5,-3.9) -- (2.5,-3.9) -- cycle;

\node at (0,-4.5) {$\digraph_{\oiv{\xi', \xi''}}$};

\node[orange] at (-1.3,-0.3) {$J'$};
\end{scope}

\end{tikzpicture}
\end{center}
\caption{Illustration of Proposition~\ref{prop:tangent_graph_update}~\eqref{item:interior_to_breakpoint} and Remark~\ref{remark:next_direction}, with a sequence of tangent digraphs around a breakpoint $\xi'$ between two consecutive segments $[\xi, \xi'] \cup [\xi', \xi'']$. The direction of $[\xi, \xi']$, from $\xi$ to $\xi'$, is given by the set of coordinate nodes $J$, indicated in green. The direction of the second segment, from $\xi'$ to $\xi''$, is governed by $J'$ depicted in orange.}\label{fig:tangent_graph_update}
\end{figure}
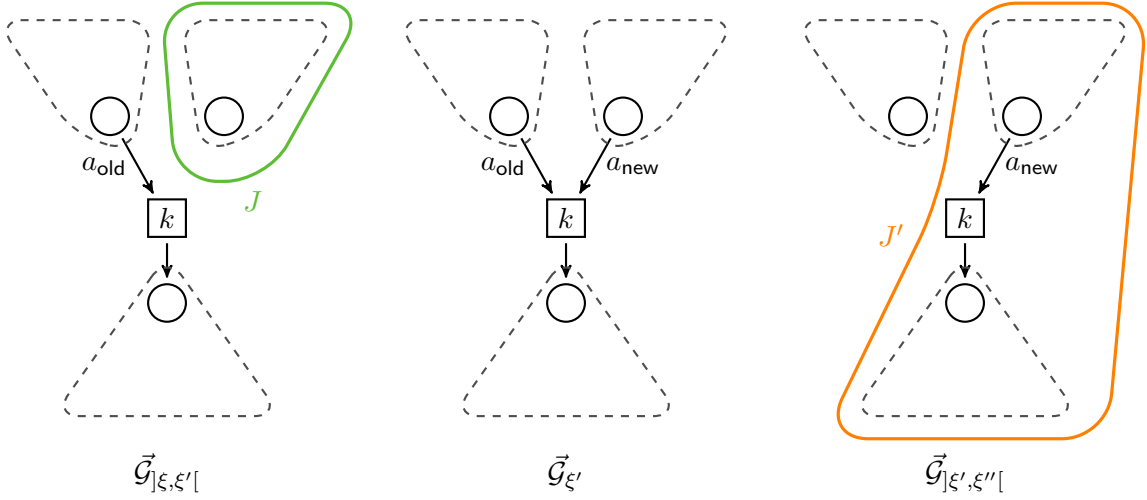

\begin{proposition}\label{prop:tangent_graph_update}
Let $[\xi,\xi'] = \{ \xi + \lambda e^J \mid 0  \leq \lambda \leq \mu \}$ be an ordinary segment of $\tropE_K$.
\begin{enumerate}[(i)]
\item\label{item:basic_point_to_interior}if $\xi$ is a basic point, \ie\ $\xi = x^{K \cup \{\ileaving\}}$ for a given $\ileaving \not \in K$, then 
\[
\digraph_{\oiv{\xi, \xi'}} = \digraph_\xi \setminus \{\ileaving\} \ .
\]
\item\label{item:interior_to_breakpoint}
if $\xi'$ is a breakpoint and $\ibreak$ the hyperplane node of $\digraph_{\xi'}$ with degree $(2,1)$ or $(1,2)$, then 
\[
\digraph_{\xi'} = \digraph_{\oiv{\xi, \xi'}} \cup \{a_\ent\} \ ,
\]
where $a_\ent$ is an arc between $\ibreak$ and the unique element of $\argmax_{j \in J} (|\ep_{\ibreak j}| + \xi_j)$. 

Moreover, if $[\xi', \xi'']$ is the next ordinary segment in $\tropE_K$, then 
\[
\digraph_{\oiv{\xi', \xi''}} = \digraph_{\xi'} \setminus \{a_\leaving\} \ .
\]
where $a_\leaving$ is the unique arc incident to $\ibreak$ with the same orientation as $a_\ent$ in $\digraph_{\xi'}$.
 \end{enumerate}
 \end{proposition}

An illustration of~\eqref{item:interior_to_breakpoint} is given in Figure~\ref{fig:tangent_graph_update}.

\begin{proof}
Let $x^{\lambda} := \xi + \lambda e^J$.
\begin{asparaenum}[(i)]
 
\item By Proposition~\ref{prop:tangent_graph_interior_edge}~(C\ref{item:shape2}), the tangent digraph $\digraph_{\oiv{\xi,\xi'}}$ does not contain the hyperplane node $\ileaving$. As $\digraph_{\oiv{\xi,\xi'}}$ is a subdigraph of $\digraph_\xi$ by Lemma~\ref{lemma:tangent_graph_is_constant}, 
we deduce that it is also a subdigraph of $\digraph_\xi \setminus \{\ileaving\}$. Since $\digraph_{\oiv{\xi,\xi'}}$ and $\digraph_\xi \setminus \{\ileaving\}$ have the same number of nodes and arcs by Proposition~\ref{prop:tangent_graph_interior_edge}, we conclude that they are equal.
 \item We assume that $\ibreak$ has degree $(2,1)$ in $\digraph_{\xi'}$, the proof being similar when $\ibreak$ has degree $(1,2)$. 
To begin with, we know that $\mu = \lambda^+_{\ibreak}(\xi,J)$ thanks to Remark~\ref{remark:degree}. Let $l \in \argmax_{j \in J} (\ep^+_{\ibreak j} + \xi_j)$. Then for all $0 < \lambda < \lambda^+_{\ibreak}(\xi,J)$, we have $\ep^+_{\ibreak l} + x^\lambda_l < \epr^+_{\ibreak} \tdot x^\lambda$, while $\ep_{\ibreak l}^+ + x^\mu_l = \epr^+_{\ibreak} \tdot x^\mu$. It follows that the arc $(l,\ibreak)$ does not belong to $\digraph_{\oiv{\xi,\xi'}}$, whereas it appears in $\digraph_{\xi'}$. We deduce that $\digraph_{\oiv{\xi,\xi'}} \cup \{(l,\ibreak)\}$ is a subgraph of $\digraph_{\xi'}$ thanks to Lemma~\ref{lemma:tangent_graph_is_constant}. Both are equal by Proposition~\ref{prop:tangent_graph_interior_edge}. Note that $\argmax_{j \in J} (\ep^+_{\ibreak j} + \xi_j)$ is reduced to $\{l\}$ as $\ibreak$ has two incoming arcs in $\digraph_{\xi'}$. Due to Remark~\ref{remark:degree} we have $\lambda^+_\ibreak(\xi, J) < \lambda^-_\ibreak(\xi,J)$. It follows that
\[
\argmax_{j \in J} (|\ep_{\ibreak j}| + \xi_j) = \argmax_{j \in J} (\ep^+_{\ibreak j} + \xi_j)  = \{l\} \enspace .
\]

In the second place, by applying Lemma~\ref{lemma:tangent_graph_is_constant} to the segment $[\xi',\xi'']$, 
we know that $\digraph_{\oiv{\xi',\xi''}}$ is a subdigraph of $\digraph_{\xi'}$.
By Proposition~\ref{prop:tangent_graph_interior_edge}, the hyperplane node $\ibreak$ has degree $(1,1)$ in $\digraph_{\oiv{\xi',\xi''}}$. Thus, the digraph $\digraph_{\oiv{\xi',\xi''}}$ 
is either equal to $\digraph_{\xi'} \setminus \{a_\ent\}$ or $\digraph_{\xi'} \setminus \{a_\leaving\}$. 
As the former corresponds to the tangent digraph $\digraph_{\oiv{\xi,\xi'}}$, we deduce that 
$\digraph_{\oiv{\xi',\xi''}} = \digraph_{\xi'} \setminus \{a_\leaving\} $.
\qedhere
\end{asparaenum}
\end{proof}

\begin{remark}\label{remark:next_direction}
We point out that in Proposition~\ref{prop:tangent_graph_update}~\eqref{item:interior_to_breakpoint}, the set $J'$ corresponding to the direction of the next segment $[\xi', \xi'']$ is precisely given by the set of coordinate nodes weakly connected to $\ibreak$ in the digraph $\digraph_{\oiv{\xi',\xi''}} = \digraph_{\xi'} \setminus \{ \aleaving \}$; see Figure~\ref{fig:tangent_graph_update} for an illustration. 

Indeed, according to Proposition~\ref{prop:tangent_graph_interior_edge}~(C\ref{item:shape2}), the digraph $\digraph_{\oiv{\xi',\xi''}}$ consists of two weakly connected components. Let $\mathfrak{J}$ be the set of coordinate nodes of the component containing the hyperplane node $\ibreak$. From any point in $\oiv{\xi', \xi''}$, the two feasible directions are $\pm e^{\mathfrak{J}}$. As a result, $J' = \mathfrak{J}$ or $J' = [n+1] \setminus \mathfrak{J}$. Let $l$ be the coordinate node incident to $a_\ent$. Then $l \in J$ by Proposition~\ref{prop:tangent_graph_update}~\eqref{item:interior_to_breakpoint}, and so $l \in J'$ as $J \subset J'$. Besides, since $a_\ent$ still appears in $\digraph_{\oiv{\xi',\xi''}}$, the coordinate node $l$ is weakly connected to $k$. Therefore, $l \in \mathfrak{J}$. We conclude that $J' = \mathfrak{J}$, as expected.

\end{remark}

The following proposition allows to incrementally maintain the sets $\enteringHyp$, $\breakHyp$ and the associated scalars $\lambda^\pm_i$ along the tropical edge $\tropE_K$.
\begin{proposition} \label{prop:lambda_sets_update}
Let $[\xi, \xi'] \cup [\xi', \xi'']$ be two consecutive ordinary segments of $\tropE_K$, where $[\xi, \xi'] = \{ \xi + \lambda e^J \mid 0 \leq \lambda \leq \mu \}$ and $[\xi', \xi''] = \{ \xi' + \lambda e^{J'} \mid 0 \leq \lambda \leq \mu' \}$. Then,
\begin{enumerate}[(i)]
\item\label{item:breakHyp_subset} $\breakHyp(\xi', J') \subset \breakHyp(\xi, J) $.
\item\label{item:argmax_update}   $\argmax(\epr^+_i \tdot \xi') = \argmax(\epr^+_i \tdot \xi)$ for all $i \in \enteringHyp(\xi', J')$.
\item\label{item:enteringHyp_update} $\enteringHyp(\xi', J') = \{ i \in \enteringHyp( \xi, J) \mid   \mu < \lambda^+_i(\xi, J) \text{ and } \argmax(\epr^+_i \tdot \xi) \cap (J' \setminus J) = \emptyset \}$.
\item\label{item:lambda_update}  for all $i \in  \enteringHyp(\xi', J') \cup \breakHyp(\xi', J')$, we have: 
  \begin{align*}
    \epr^+_i \tdot \xi' &= \epr^+_i \tdot \xi \\
    \lambda^+_i( \xi', J') &= \min \left (  \lambda^+_i (\xi, J) - \mu \ , \  (\epr^+_{i} \tdot \xi) - \max _{ j \in J' \setminus J } ( \ep^+_{ij} + \xi_j)  \right) \,, \\
 \lambda^-_i( \xi', J') &= \min \left (  \lambda^-_i (\xi, J) - \mu \ , \ (\epr^+_{i} \tdot \xi) -   \max _{ j \in J' \setminus J } (\ep^-_{ij} + \xi_j ) \right) \,.
  \end{align*}

\end{enumerate}

\end{proposition}

\begin{proof}

  \begin{asparaenum}[(i)]
  \item  Suppose by contradiction that $i \not \in \breakHyp(\xi, J)$. Then, by Lemma \ref{lemma:feasible_directions} the intersections $\argmax(\epr^+_i \tdot \xi)\cap J$ and $\argmax(\epr^-_i \tdot \xi)\cap J$ are both non-empty. As a consequence,  $\argmax(\epr^+_i \tdot \xi')$ and $\argmax(\epr^-_i \tdot \xi')$ are included in $J$. Since $J \subset J'$ by Proposition \ref{prop:tropical_edge}, we conclude that $i \not \in  \breakHyp(\xi', J')$.
 
\item  First observe that $\enteringHyp(\xi', J') \subset \enteringHyp(\xi, J)$. Indeed, consider an $i \in K \setminus \enteringHyp(\xi, J) $. Then  $\argmax(\epr^+_i \tdot \xi) \cap J \neq \emptyset$, which implies $\argmax(\epr^+_i \tdot \xi') \subset J$.
 Using the inclusion $J \subset J'$, we obtain that $\argmax(\epr^+_i \tdot \xi') \cap J' \neq \emptyset$, and therefore $i \not \in \enteringHyp(\xi', J')$.

Second  if $i \in \enteringHyp(\xi, J)$ satisfies  $\mu \geq \lambda^+_i(\xi, J) $
 then $\argmax(\epr^+_i \tdot \xi')$ intersects $J \subset J'$, thus $i \not \in \enteringHyp(\xi', J')$. As a consequence:
\begin{equation}\label{eq:entering_hyp_update_first_step}
\enteringHyp(\xi', J')  \subset   \{ i \in \enteringHyp(\xi, J) \mid  \mu < \lambda^+_i(\xi, J) \} \,.
\end{equation}
Finally for any $i \in \enteringHyp(\xi', J')$, we have $ \mu < \lambda^+_i(\xi, J)$ and therefore $\argmax(\epr^+_i \tdot \xi') = \argmax(\epr^+_i \tdot \xi)$.

\item 

Using~\eqref{eq:entering_hyp_update_first_step} let us consider an  $i \in \enteringHyp(\xi, J)$ such that $\mu < \lambda^+_i(\xi, J) $.
 Then, as above,  $\argmax(\epr^+_i \tdot \xi') = \argmax(\epr^+_i \tdot \xi)$. Moreover, $i \in \enteringHyp(\xi,J)$ implies $ \argmax(\epr^+_i \tdot \xi) \cap J = \emptyset$. Thus $\argmax(\epr^+_i \tdot \xi') \cap J' = \emptyset$ if and only if $\argmax(\epr^+_i \tdot \xi) \cap (J' \setminus J) = \emptyset$. 

\item  Consider  $i \in  \enteringHyp(\xi', J') \cup \breakHyp(\xi', J')$. 
If $i \in  \enteringHyp(\xi', J')$ then  $\mu < \lambda^+_i(\xi, J) $ by~\eqref{eq:entering_hyp_update_first_step}. Otherwise, if $i \in \breakHyp(\xi', J')$, then $i \in \breakHyp(\xi, J)$ by~\eqref{item:breakHyp_subset} and thus $\mu \leq \lambda^+_i(\xi, J)$ by~\eqref{eq:maximal_step}. In both cases, we obtain $\epr^+_i \tdot \xi' = \epr^+_i \tdot \xi$.

Let us rewrite $\lambda^+_i( \xi', J')$ as follows:
 \begin{align*}
   \lambda^+_i( \xi', J') &=  \min \left ( (\epr^+_i \tdot \xi') - \max_{j \in J} (\ep^+_{ij} + \xi'_j ) \ , \   
 (\epr^+_i \tdot \xi') -  \max_{ j\in J' \setminus J} (  \ep^+_{ij} + \xi'_j) \right ) \,.
 \end{align*}
 We saw that $\epr^+_i \tdot \xi' = \epr^+_i \tdot \xi$. Furthermore, $\xi'_j = \xi_j + \mu$ if $j \in J$ and $\xi'_j = \xi_j$ otherwise.
Thus the first term of the minimum above is equal to:
\[
 (\epr^+_i \tdot \xi) - \max_{j \in J} (\ep^+_{ij} + \xi + \mu ) = \lambda^+_i(\xi, J) - \mu.
\]
The second term satisfies:
\[
 (\epr^+_i \tdot \xi') -  \max_{ j\in J' \setminus J} (  \ep^+_{ij} + \xi'_j) =  (\epr^+_i \tdot \xi) -  \max_{ j\in J' \setminus J} (  \ep^+_{ij} + \xi_j) \,.
\]
The same argument holds for $\lambda^-_i( \xi', J')$.\qedhere
\end{asparaenum}
\end{proof}

We now present an algorithm (Algorithm~\ref{alg:follow_ordinary_segment}) allowing to move along an ordinary segment $[\xi, \xi'] = \{ \xi + \lambda e^J \mid 0 \leq \lambda \leq \mu \}$ of the tropical edge $\tropE_K$. This algorithm takes as input the initial endpoint $\xi$, together with some auxiliary data, including the set $J$ encoding the direction of the segment $[\xi, \xi']$, the tangent digraph in $\oiv{\xi, \xi'}$, the sets $\enteringHyp(\xi,J)$ and $\breakHyp(\xi,J)$, \etc. It also uses an auxiliary function $\Omega$, which is defined for the pairs $(i,j) \in  \enteringHyp(\xi, J) \times [n+1] $, and which returns in time $O(1)$ whether $j \in \argmax(W^+_i \tdot \xi)$. We shall see in the main pivoting algorithm that this function is defined once for all when pivoting over the whole tropical edge.

Algorithm~\ref{alg:follow_ordinary_segment} returns the other endpoint $\xi'$. On top of that, if $\xi'$ is a breakpoint of $\tropE_K$, it provides the set $J'$ corresponding to the direction of the next ordinary segment $[\xi', \xi'']$ of $\tropE_K$, some additional data corresponding to $\xi'$, $J'$ (for instance the sets $\enteringHyp(\xi',J')$ and $\breakHyp(\xi',J')$), and the digraph $\digraph_{\oiv{\xi',\xi''}}$.

Several kinds of data structures are manipulated in Algorithm~\ref{alg:follow_ordinary_segment}, and we need to specify the complexity of the underlying operations. Arithmetic operations over $\trop$ are supposed to be done in time $O(1)$. Tangent digraphs are represented by adjacency lists. They are of size $O(n)$, and so they can be visited in time $O(n)$. Matrices are stored as two dimensional arrays, so an arbitrary entry can be accessed in $O(1)$. Vectors and the values $\epr^+_i \tdot \xi$, $\lambda^+_i(\xi, J)$ and $\lambda^-_i(\xi,J)$  for $i \in [m]$  are stored as arrays of scalars. Apart from $\Jdiff = J' \setminus J$, sets are represented as Boolean arrays, so that testing membership takes $O(1)$. The set $\Jdiff$ is stored as a list, thus iterating over its elements can be done in $O(|\Jdiff|)$.

\begin{algorithm}
\SetKwComment{Comment}{}{}
\SetArgSty{text}
\DontPrintSemicolon
\footnotesize

\KwIn{An endpoint $\xi$ of an ordinary segment $[\xi,\xi']$ of a tropical edge $\tropE_K$ and:
 \begin{asparaitem}
  \item the set $J$ encoding the direction $e^J$ of $[\xi,\xi'] = \{ \xi + \lambda e^J \mid 0 \leq \lambda \leq \mu \}$ 
  \item the tangent digraph $\digraph_{\oiv{\xi,\xi'}}$ in the relative interior of $[\xi,\xi']$
  \item the sets $\enteringHyp(\xi,J)$ and $\breakHyp(\xi,J)$
 \item the scalars $\epr^+_i \tdot \xi$, $\lambda^+_{i}(\xi,J)$ and $\lambda^-_{i}(\xi,J)$ for $i \in \breakHyp(\xi,J) \cup \enteringHyp(\xi,J)$
  \item an auxiliary function $\Omega$ determining in time $O(1)$ if $j \in \argmax(\epr^+_i \tdot \xi)$ for all $i \in \enteringHyp(\xi,J)$ and $j \in [n+1]$
  \end{asparaitem}
}
\vskip1ex
\KwOut{
The other endpoint $\xi'$ and, \\
if $\xi'$ is a basic point, the integer $\ient \not \in K$ such that $\xi' = x^{K \cup \{\ient\}}$;
\\
 if $\xi'$ is a breakpoint:
\begin{asparaitem}
\item the set $J'$ encoding the direction $e^{J'}$ of the next ordinary segment $[\xi', \xi''] = \{ \xi' + \lambda e^{J'} \mid 0 \leq \lambda \leq \mu' \}$ 
\item the tangent digraph  $\digraph_{\oiv{\xi',\xi''}}$
\item the sets $\enteringHyp(\xi',J')$ and $\breakHyp(\xi',J')$
\item the scalars $\epr^+_i \tdot \xi'$, $\lambda^+_{i}(\xi',J')$ and  $\lambda^-_{i}(\xi',J')$ for $i \in \breakHyp(\xi',J') \cup \enteringHyp(\xi',J')$
\end{asparaitem}
}
\vskip3ex
\tikz[remember picture,baseline]{\coordinate (box1 tl);}$\mu \leftarrow \min \{ \min( \lambda^+_i(\xi,J), \lambda^-_i(\xi,J) ) \mid i \in \breakHyp(\xi,J) \text{ or } ( i \in \enteringHyp(\xi,J) \text{ and }  \lambda^-_i(\xi,J) \leq \lambda^+_i(\xi,J) ) \} $ \Comment*[r]{$O(m)$} 
$\xi' \leftarrow \xi + \mu e^J$ \Comment*[r]{$O(n)$} 
\If
{$\mu = \lambda^-_{\ient}(\xi,J)$ for some $\ient \in \enteringHyp(\xi,J)$}{
\Return{$(\xi', \ient)$}\Comment*[r]{($\xi'$ is a basic point)\tikz[remember picture,baseline]{\coordinate (box1 br);}}
} 

\vskip3.5ex
\tikz[remember picture,baseline]{\coordinate (box2 tl);}$\ibreak \leftarrow$ the unique element of $\breakHyp(\xi, J)$ such that $\mu = \min(\lambda^+_{\ibreak}(\xi, J) , \lambda^-_{\ibreak}(\xi,J))$ \Comment*[r]{($\xi'$ is a breakpoint)}
$\ell \leftarrow$ the unique element  in $\argmax_{j \in J}|\ep_{{\ibreak}j}| + \xi_j$ \Comment*[r]{$O(n)$} 

$a_\ent \leftarrow$ the arc from $\ell$ to $\ibreak$ if $ \lambda^+_{{\ibreak}}(\xi,J) < \lambda^-_{{\ibreak}}(\xi,J)$, the arc  from ${\ibreak}$ to $\ell$ otherwise \label{line:new_arc} \Comment*[r]{$O(1)$} 

$\digraph_{\xi'}  \leftarrow \digraph_{\oiv{\xi, \xi'}} \cup \{ a_\ent\}$ \label{line:add_arc}\Comment*[r]{$O(n)$\tikz[remember picture,baseline]{\coordinate (box2 br);}}

compute $\breakHyp(\xi',J')$ by visiting $\digraph_{\xi'}$ \Comment*[r]{$O(n)$}  \label{line:compute_breakHyp}   

\vskip2.5ex
\tikz[remember picture,baseline]{\coordinate (box3 tl);}$a_\leaving \leftarrow$ the only arc incident to ${\ibreak}$ in $\digraph_{\xi'}$ with the same orientation as $a_\ent$   \Comment*[r]{$O(1)$} 
$\digraph_{\oiv{\xi',\xi''}} \leftarrow \digraph_{\xi'} \setminus \{ a_\leaving \}$   \label{line:remove_arc} \Comment*[r]{$O(n)$} 
 
$J' \leftarrow$ coordinate nodes of $\digraph_{\oiv{\xi',\xi''}}$ weakly connected to ${\ibreak}$   
\Comment*[r]{$O(n)$\tikz[remember picture,baseline]{\coordinate (box3 br);}}
 
\vskip4.5ex 
\tikz[remember picture,baseline]{\coordinate (box4 tl);}$\Jdiff \leftarrow$ the list of elements of  $J' \setminus J$ \label{line:Jdiff} \Comment*[r]{$O(n)$} 
$\enteringHyp(\xi',J') \leftarrow \{ i \in \enteringHyp(\xi,J) \mid \mu < \lambda^+_i(\xi,J) \text{ and } \argmax(\epr^+_i \tdot \xi) \cap \Jdiff = \emptyset  \}$ \label{line:compute_enteringHyp} \Comment*[r]{$O(m | J' \setminus J |)$ using the function $\Omega$} %

\For(\Comment*[f]{$O(m)$ iterations}){$ i \in \enteringHyp(\xi',J') \cup \breakHyp(\xi',J')$ 
} 
{
$\epr_i^+ \tdot \xi' := \epr_i^+ \tdot \xi$  \Comment*[r]{$O(1)$}
$\lambda^+_{i}(\xi',J') := \min \Bigl( \lambda^+_i(\xi, J) - \mu \ , \ (\epr^+_i \tdot \xi) - \max\limits_{j \in \Jdiff} ( \ep^+_{ij} + \xi_j)  \Bigr)$\label{line:lambda_plus}\Comment*[r]{$O(| J' \setminus J |)$} 
$\lambda^-_{i}(\xi',J') := \min \Bigl( \lambda^+_i(\xi, J) - \mu \ , \ (\epr^+_i \tdot \xi) - \max\limits_{j \in \Jdiff} ( \ep^-_{ij} + \xi_j)  \Bigr)$
 \label{line:lambda_minus} \Comment*[r]{$O(| J' \setminus J |)$\tikz[remember picture,baseline]{\coordinate (box4 br) at (0, -0.1);}} 

} 

\Return{$(\xi',J',\digraph_{\oiv{\xi',\xi''}}, \enteringHyp(\xi',J'),\breakHyp(\xi',J'),(\epr^+_i \tdot \xi')_i, (\lambda^+_i(\xi',J'))_i, (\lambda^-_i(\xi',J'))_i)$}
\begin{tikzpicture}[remember picture,overlay]
\fitbox{box1}{green!80!black}{Proposition~\ref{prop:maximal_step} \eqref{item:maximal_step1}--\eqref{item:maximal_step3}} ;
\fitbox{box2}{blue!80!black}{Proposition~\ref{prop:tangent_graph_update}~\eqref{item:interior_to_breakpoint}, Remark~\ref{remark:degree}} ; 
\fitbox{box3}{orange!80}{Remark~\ref{remark:next_direction}} ; 
\fitbox{box4}{red!80}{Proposition~\ref{prop:lambda_sets_update}~\eqref{item:enteringHyp_update}--\eqref{item:lambda_update}} ; 
\end{tikzpicture}  
\caption{Traversal of an ordinary segment of an tropical edge}
\label{alg:follow_ordinary_segment}
\end{algorithm}
 
\begin{proposition}\label{prop:follow_ordinary_segment}
Algorithm~\ref{alg:follow_ordinary_segment} is correct, and its time complexity is bounded by $O(n + m |J' \setminus J|)$.
\end{proposition}

\begin{proof}
\begin{asparaitem}
\item[\emph{Correctness}:]
The correctness of the highlighted parts of the algorithm straightforwardly follows from the corresponding results given as annotations.
 
At Line~\ref{line:compute_breakHyp}, 
the set $\breakHyp(\xi',J')$ is built by iterating over the nodes of $\digraph_{\xi'}$, and collecting the hyperplane nodes $i$ with no neighbor in $J'$. 
This is correct since the set of hyperplane nodes  is precisely $K$ (by Proposition~\ref{prop:tangent_graph_interior_edge} and the fact that $\xi'$ is a breakpoint), and because the adjacent nodes of each $i \in K$ are precisely the elements of $\argmax(\epr^+_i \tdot \xi') \cup  \argmax(\epr^-_i \tdot \xi')$  by construction of $\digraph_{\xi'}$. 

\item[\emph{Complexity}:] 
At Lines~\ref{line:add_arc} and~\ref{line:remove_arc}, the operations of removing  or adding an arc can be performed in $O(n)$ by visiting the digraphs. Identifying the arc $a_{\leaving}$ at Line~\ref{line:remove_arc} amounts to iterate over the arcs incident to ${\ibreak}$, and there is exactly 3 such arcs by Proposition~\ref{prop:tangent_graph_interior_edge}.

Testing whether a hyperplane node $i$ of $\digraph_{\xi'}$ satisfies $\argmax(\epr^+_i \tdot \xi') \cap J' = \argmax(\epr^-_i \tdot \xi') \cap J' = \emptyset$ can be done in $O(1)$, by determining whether the adjacent coordinate nodes (at most $3$) in $\digraph_{\xi'}$ belong to $J'$. Thus the set $\breakHyp(\xi',J')$ can be built in time $O(n)$ by iterating over the hyperplane nodes of $\digraph_{\xi'}$. 

Given $i \in \enteringHyp(\xi,J)$, determining whether $\argmax(\epr^+_i \tdot \xi) \cap \Jdiff = \emptyset$ can be performed by calling the auxiliary function $\Omega$ for every element $j \in \Jdiff$. 
It follows that $\enteringHyp(\xi',J')$ can be computed at Line~\ref{line:compute_enteringHyp} in time $O(m |J' \setminus J|)$.

Computations at Lines~\ref{line:lambda_plus} and \ref{line:lambda_minus} are  done by iterating over elements $j \in \Jdiff$ and then retrieving the values of $\epr^+_i \tdot \xi$, $\ep^+_{ij}$, $\ep^-_{ij}$ and $\xi_j$. Since these values are stored in arrays, they can be accessed to in constant time. Therefore, $\lambda^+_i(\xi', J')$ and  $\lambda^-_i(\xi', J')$ are computed in time $O(|\Jdiff|) = O(|J' \setminus J|)$.
The complexity of other operations is easily obtained. In total, the complexity of the algorithm is $O(n+m |J' \setminus J|)$.
\qedhere
\end{asparaitem}
\end{proof}

\begin{algorithm}[h]
\SetKwComment{Comment}{}{}
\SetArgSty{text}
\DontPrintSemicolon
\footnotesize
\KwIn{
A basic point $x^\basis$ of $\tropP(A, b)$, the associated set $\basis$, and an integer $\ileaving \in \basis$}
\KwOut{The other basic point $x^{\basis'}$ of the edge $\tropE_{\basis \setminus \{\ileaving\}}$, and the integer $\ient \in \basis \setminus \{\ileaving\}$ such that $\basis' = (\basis \setminus \{\ileaving\}) \cup \{\ient\}$}
compute $\digraph_{x^\basis} $ \Comment*[r]{$O(m n)$}\label{line:prelim_begin}
$\digraph_{\oiv{\xi^1,\xi^2} } \leftarrow \digraph_{x^{\basis}} \setminus \{ \ileaving \} $ \label{line:compute_digraph}\Comment*[r]{$O(n)$} 
$J \leftarrow$ coordinate nodes weakly connected to the element of $ \argmax(\epr^+_{\ileaving} \tdot x^\basis)$ in $\digraph_{\oiv{\xi,\xi'} }$  \label{line:compute_J} \Comment*[r]{$O(n)$} 
compute $E \leftarrow \enteringHyp(x^\basis,J)$ and $B \leftarrow \breakHyp(x^\basis,J)$ \Comment*[r]{$O(m n)$}
compute $\epr^+_i \tdot x^\basis$, $\lambda^+_i(x^\basis, J)$ and  $\lambda^-_i(x^\basis, J)$ for all $i \in E \cup B$
\Comment*[r]{$O(m n)$}
$\Omega \leftarrow$ function defined on the set 
$E \times [n+1]$ by
$ \Omega(i,j) = \begin{cases} \textbf{true} & \text{if $j \in \argmax(\epr^+_i \tdot x^\basis)$}\\ \textbf{false} & \text{otherwise} \end{cases} $ \label{line:compute_oracle}\Comment*[r]{$O(m n)$}
$\mathit{input} \leftarrow 
 x^\basis, J, \digraph_{\oiv{\xi^1,\xi^2} }, E,B, (\epr^+_i \tdot x)_{i \in E \cup B}, (\lambda^+_i(x^\basis,J))_{i \in E \cup B},  (\lambda^-_i(x^\basis,J))_{i \in E \cup B}$ \label{line:prelim_end}\;

\While(\Comment*[f]{at most $n$ iterations}){\label{line:loop_begin}\textbf{true}}{%
call Algorithm~\ref{alg:follow_ordinary_segment} on $(\mathit{input},\Omega)$ and stores the result in $\mathit{output}$\;
\lIf{$\mathit{output}$ is of the form $(\xi', \ient)$}{
\Return $(\xi', \ient)$
}
\lElse{%
$\mathit{input} \leftarrow \mathit{output}$} \label{line:loop_end}
}
\caption{Linear-time pivoting algorithm}\label{alg:pivot}
\end{algorithm}

\begin{theorem}\label{th:pivot}
Algorithm~\ref{alg:pivot} allows to pivot from a basic point along a tropical edge in time $O(n(m + n))$ and space $O(nm)$.
\end{theorem}
\begin{proof}
First observe that the function $\Omega$ initially defined at Line~\ref{line:compute_oracle} does not need to be updated during the iterations of the loop from Lines~\ref{line:loop_begin} to~\ref{line:loop_end}. Indeed, let $[\xi, \xi']$ and $[\xi', \xi'']$ be two consecutive ordinary segments of direction $e^J$ and $e^{J'}$ respectively. By Proposition~\ref{prop:lambda_sets_update}, we have the inclusion $\enteringHyp(\xi',J') \subset \enteringHyp(\xi,J)$ and the equality $\argmax(\epr^+_i \tdot \xi') = \argmax(\epr^+_i \tdot \xi)$ for all $i \in \enteringHyp(\xi',J')$ . It follows that if $\Omega$ is a function determining whether $j \in \argmax(\epr^+_i \tdot \xi)$ for all $i \in \enteringHyp(\xi,J)$, it can be used as well to determine whether $j \in \argmax(\epr^+_i \tdot \xi')$ for all $i \in \enteringHyp(\xi',J')$.

Then, the correctness of the algorithm straightforwardly follows from Proposition~\ref{prop:tangent_graph_update}~\eqref{item:basic_point_to_interior} (for the computation of $\digraph_{\oiv{\xi^1, \xi^2}}$ at Line~\ref{line:compute_digraph}), Proposition~\ref{prop:tangent_graph_interior_edge} (for the computation of $J$ at Line~\ref{line:compute_J}) and Proposition~\ref{prop:follow_ordinary_segment}. 

The complexity of the operations from Lines~\ref{line:prelim_begin} to~\ref{line:prelim_end} can easily be verified to be in $O(m n)$. 
Let $q \leq n$ be the number of iterations of the loop from Lines~\ref{line:loop_begin} and~\ref{line:loop_end}, and let $e^{J_1} , e^{J_2}, \dots, e^{J_{q}}$ be the directions of the ordinary segments followed during the successive calls to Algorithm~\ref{alg:follow_ordinary_segment}. By Proposition~\ref{prop:follow_ordinary_segment}, the total complexity of the loop is
\[O(n q + m  |J_2 \setminus J_1| + m|J_3 \setminus J_2 | + \dots + m|J_{q} \setminus J_{q-1}|)\,, \]
which can be bounded by $O(n(m + n))$. 
Finally, the space complexity is obviously bounded by $O(n m)$. 
\end{proof}

\section{Reduced costs}\label{sec:reduced_costs}

\noindent
In this section, we introduce the concept of tropical reduced costs, which are merely the signed valuations of the reduced costs over Puiseux series. Then, pivots improving the objective function and optimality over Puiseux series can be determined only by the signs of the tropical reduced costs.
 We show that, under some genericity assumptions, the tropical reduced costs can be computed using only the tropical entries $A$ and $c$ in time $O(n(m + n))$. This complexity is similar to classical simplex algorithm, as this operation corresponds to the update of the inverse of the basic matrix $A_\basis$. 

\subsection{Symmetrized tropical semiring}\label{sec:symmetrized}
Until now our coordinate domain was the set of signed tropical numbers $\strop$.  
As noted in~Section~\ref{subsect-signed}, this has the drawback of not being a
semiring since, in general, $a\tplus(\tminus a)$ is not defined.  This can be remedied by extending $\strop$ to the
\emph{symmetrized tropical semiring} from \cite{akian1990linear}, which we
denote here as $\symtrop$.  
We shall see in particular that the computation of {\em tropical reduced costs} reduces to the resolution of the analogue of a Cramer system
over the symmetrized tropical semiring. 

As a set $\symtrop$ is the union of $\strop$ and a third copy of $\trop$, denoted $\trop_{\bullet}$. The members of the
latter, written as $a^{\bullet}$ for $a \in \trop$, are the \emph{balanced tropical numbers}.
The numbers $a$, $\tminus a$ and $a^{\bullet}$ are pairwise distinct unless $a = \zero$. 
Sign and modulus are extended to $\symtrop$ by setting $\sign(a^{\bullet}) = 0$ and $| a^{\bullet} | = a$.

The addition of two elements $x, y \in \symtrop$, denoted by $x \splus y$, is defined to be $\max(|x|, |y|)$ if the maximum is attained only by elements of
positive sign, $\tminus \max(|x|, |y|)$ if it is attained only by elements of negative sign, and $\max(|x|,|y|)^{\bullet}$ otherwise.  For instance, $( \tminus
1) \tplus 1 \tplus (\tminus 3) = 1^{\bullet} \tplus (\tminus 3) = \tminus 3$.  The multiplication $x \stimes y$ of two elements $x, y \in \symtrop$ yields the
element with modulus $|x| + |y|$ and with sign $\sign(x) \sign(y)$. For example, $(\tminus 1) \stimes 2= \tminus 3$ and $(\tminus 1) \stimes (\tminus 2)= 3$ but
$1^{\bullet} \stimes (\tminus 2) = 3^{\bullet}$.  An element $x\in \strop$ not equal to $\zero$ has a multiplicative inverse $x^{-1}$ which is the element of
modulus $-|x|$ and with the same sign as $x$.  The addition $A \splus B$ and multiplication $A \sdot B$ of two matrices $A = (a_{ij})$ and $B=(b_{ij})$ are
the matrices with entries $a_{ij} \splus b_{ij}$ and $\ssum_k a_{ik} \stimes b_{kj}$, respectively.

The set $\symtrop$ also comes with the \emph{reflection map} $x \mapsto \tminus x$ which sends a balanced number to
itself, a positive number $a$ to $\tminus a$ and a negative number $\tminus a$ to $a$. We will write $x \tminus y$ for
$x \splus (\tminus y)$.
Two numbers $x,y \in \symtrop$ satisfy the \emph{balance relation} $x \bal y$ when $x \tminus y$ is a balanced
number. Note that
\begin{align}
x \bal y \implies  x=y
\qquad\text{for all } x, y \in \strop
\enspace .\label{e-bal-to-eq}
\end{align}
The balance relation is extended entry-wise to vectors in $\symtrop^n$.
In the semiring $\symtrop$, the relation $\bal$ plays the role of the equality relation; in particular the next result
shows that a version of Cramer's Theorem is valid in the tropical setting, up to replacing equalities by balances. 

 The tropical determinant of the square matrix $M = (m_{ij}) \in \symtrop^{n \times n}$ is given by
\begin{equation}
  \tdet( M) = \ssum_{\sigma \in \Sym(n)} \tsign(\sigma) \stimes m_{1 \sigma(1)} \stimes \cdots \stimes m_{n \sigma(n)}
  \label{eq:tdet}
\end{equation}
Also observe  that a square matrix of $\strop^{n \times n}$ is tropically sign singular if and only if its tropical determinant is a balanced number. 

\begin{theorem}[Signed tropical Cramer Theorem~\cite{akian1990linear}]
\label{thm:tropical_cramer}
Let $M \in \symtrop^{n \times n}$ and $d \in \symtrop^n$. Every solution $y \in \strop^n$
  of the system of balances
\begin{equation}
  M \sdot y \bal d
\label{eq:balance_system}
\end{equation}
satisfies
\begin{equation}
\tdet ( M)  \stimes y_j \bal \tdet( M_{j \leftarrow d} ),
\qquad \text{for all } j \in [n], 
  \label{eq:cramer}
\end{equation}
where $M_{j \leftarrow d}$ is the matrix obtained by replacing the $j$th column of $M$ by $d$.

Conversely, if the tropical determinants $\tdet(M)$ and $ \tdet(M_{ j \leftarrow d})$ for ${j\in [n]}$ are not balanced elements, then the vector with
entries $y_j = \tdet( M) ^{-1} \stimes \tdet( M_{j \leftarrow d}) $ is the unique solution of (\ref{eq:balance_system}) in $\strop^n$.
\end{theorem}
This result was proved in~\cite{akian1990linear}; see also~\cite{AGG08b,Akian2013} for more recent discussions.
A different tropical Cramer theorem (without signs) was proved by Richter-Gebert, Sturmfels and Theobald~\cite{richter2005first}; their proof relies
on the notion of a coherent matching field introduced by Sturmfels and Zelevinsky~\cite{SZ}.

\begin{remark}
  The quintuple $(\strop,\max,+,\ominus 0,\trop_\bullet)$ is an example of a ``fuzzy ring'' in the sense of \cite[Definition 1.1]{Dress86}.  In the
  notation of that reference,
$\strop$
is ``the group of units'' and $\trop_\bullet$ is the set denoted ``$K_0$''.
\end{remark}

\subsection{Computing solutions of tropical Cramer systems} \label{sec:computing_reduced_costs} The Jacobi iterative algorithm
of~\cite{akian1990linear} allows one to compute a signed solution $y$ of the system $M \sdot y \bal d$; see also~\cite{Akian2013} for more
information.  We next present a combinatorial version of this algorithm, for the special case where the entries of $M$ and $d$ are in
$\strop$.

Suppose that $\tdet(M) \neq \zero$, and let $\sigma$ be a maximizing permutation in $\tdet(M)$ (or equivalently, in $\tper(|M|)$). The \emph{Cramer
  digraph} of the system associated with $\sigma$ is the weighted bipartite directed graph over the ``column nodes'' $\{1, \dots, n+1\}$ (the index
$n+1$ represents the affine component) and ``row nodes'' $\{1, \dots, n\}$ defined as follows: every row node $i \in [n]$ has an outgoing arc to the
column node $\sigma(i)$ with weight $m_{i\sigma(i)}^{-1}$, and an incoming arc from every column node $j \neq \sigma(i)$ with weight $\tminus m_{ij}$ when
$j \in [n]$, and weight $d_i$ when $j = n+1$.
\begin{example}\label{exmp:cramer}
  The maximizing permutation for the system of balances~\eqref{eq:cramer_system_example} below is $\sigma(1) =1, \sigma(2) = 3 $ and
  $\sigma(3)=2$. The Cramer digraph is represented in Figure~\ref{fig:cramer_graph_example}.
  \begin{equation}
    \label{eq:cramer_system_example}
    \begin{pmatrix}
      \tminus(-1) & -\infty & -\infty \\
      -1 & \tminus(-2) & 0 \\
      \tminus(-1) & 0 & -\infty
    \end{pmatrix}
    \tdot
    \begin{pmatrix}
      y_1 \\ y_2 \\ y_3
    \end{pmatrix}
    \bal
    \begin{pmatrix}
      -2 \\ 0 \\ -1
    \end{pmatrix}
  \end{equation}

\begin{figure}[t]
  \centering
  \begin{tikzpicture}[ scale=1.0,
 vertex/.style={circle,draw=black, thick,inner xsep=1ex, text width=1ex, align = center},
 hyp/.style={rectangle,draw=black, thick,,inner xsep=1ex, text width=2ex, align=center, minimum size=3.5ex },
 edge/.style={draw=black,thick, >= triangle 45, ->} ]

\node [vertex] (3) at (0,4) {$3$};
\node [vertex] (2) at (2,2) {$2$};
\node [vertex] (1) at (4,4) {$1$};

\node [hyp] (x_1) at (2,4) {$y_1$};
\node [hyp] (x_2) at (0,2) {$y_2$};
\node [hyp] (x_3) at (4,2) {$y_3$};
\node [hyp] (c) at (2,0) {$y_4$};

\draw[edge,  color = red] (1)   -- node[above, color = black] {$ \tminus 1$} (x_1)  ; 
\draw[edge, dashdotted] (x_1) -- node[above] {$-1$} (3);
\draw[edge, dashdotted] (x_1) -- node[right] {$\tminus(-1)$} (2);
\draw[edge,  color = red] (3)   -- node[right, color = black] {$0$} (x_2);
\draw[edge, dashdotted] (x_2) -- node[below] {$-2$} (2);
\draw[edge,  color = red] (2)   -- node[below, color = black] {$0$} (x_3);
\draw[edge] (c)   -- node[right] {$0$} (2);
\draw[edge] (c) to[out = 15, in = -90]  node[right, pos = 0.7] {$-2$} (5,3) to [out = 90, in = -15] (1);
\draw[edge] (c)  to[out=165, in = -90]  node[left, pos = 0.7] {$-1$} (-1,3) to [out = 90, in =-165] (3);

  \end{tikzpicture}
\caption{The Cramer digraph for the system of balances in~\eqref{eq:cramer_system_example}. Column nodes are squares and row nodes are circles. Arcs with weight $-\infty$ are omitted. The maximizing permutation $\sigma$ is given by the red arcs. 
 The coordinate $y_j$ of the signed solution $y$ of \eqref{eq:cramer_system_example} is obtained by the multiplication (in $\symtrop$) of the weight on the longest path from $y_4$ to $y_j$. }
  \label{fig:cramer_graph_example}
\end{figure}
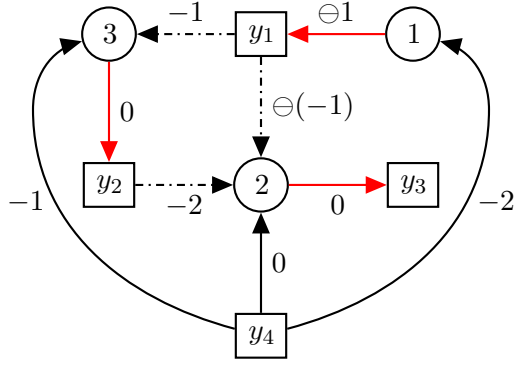
\end{example}

Note that all the coefficients $m_{i\sigma(i)}$ are different from~$\zero$.
In the sequel, it will be convenient to consider
the longest path problem in the weighted
digraph obtained from the Cramer digraph associated with $\sigma$
by forgetting the tropical signs,
i.e., by taking the modulus of each weight. 
Note in particular that there is no directed cycle the weight of which has a positive modulus (otherwise $\sigma$ would not be a maximizing permutation in the tropical determinant of $M$). Consequently, the latter longest path problem is well-defined (longest weights being either finite or $-\infty$, but not $+\infty$).

The \emph{digraph of longest paths} from a node $v$ refers to the subgraph of the Cramer digraph formed by the arcs belonging to a longest path from node $v$. This digraph is acyclic
and each of its nodes is reachable from the node $v$ (possibly with a path of length $\zero$). 
As a result, it always contains a directed tree rooted at $v$. Such a directed tree can be described by a map which sends every node (except the root) to its parent node. Note that by construction of the Cramer digraph, a column node $j$ has only one possible parent node $\sigma^{-1}(j)$. Consequently, we will describe a directed tree of longest paths by a map $\gamma$ that sends every row node to its parent column node. 

\begin{proposition}\label{prop-jacobi-combinatorial}
Let $M \in \strop^{n \times n}$ such that $\tdet(M) \neq \zero$ and $d \in \strop^n$. Let $\sigma$ be a maximizing permutation in the tropical determinant of $M$. 
In the Cramer digraph of the system $M \sdot y \bal d$ associated with $\sigma$, consider the digraph of longest paths from the column node $n+1$. In this digraph of longest paths, choose any directed subtree $\gamma$ rooted at the column node $n+1$. Then, the following recursive relations
\begin{equation}
y_{\sigma(i)} = 
\begin{cases}
 d_i  \stimes m_{i\sigma(i)}^{-1}& \text{when}\ \gamma(i)=n+1 \ ,\\
\tminus  m_{i \gamma(i)} \stimes m_{i\sigma(i)}^{-1}  \stimes y_{\gamma(i)} & \text{otherwise} 
\end{cases}
\label{eq:cramer_recursive_relations}
\end{equation}
provide a solution in $\strop^n$ of the system $M \sdot y \bal d$. 
\label{prop:computing_cramer_solution}
\end{proposition}
\begin{proof}
Since the column node $n+1$ reaches all column nodes in the directed tree defined by $\gamma$, Equation~\eqref{eq:cramer_recursive_relations} defines a point $y$ in $\strop^n$. The modulus $|y_j|$ is the weight of a longest path from the column node $n+1$ to the column node $j$. By the optimality conditions of the longest paths problem, for any $i \in [n]$, we have:
\begin{align*}
|m_{i \sigma(i)}|+ |y_{\sigma(i)}| & \geq |d_i| \ , \\
|m_{i \sigma(i)}| + |y_{\sigma(i)}| & \geq  |m_{ij}| + |y_j| \quad\text{for all} \  j \in [n] \ .
\end{align*}
Furthermore, we have $|m_{i \sigma(i)}| + |y_{\sigma(i)}| = |m_{i\gamma(i)}| + |y_{\gamma(i)}|$ when $\gamma(i) \neq n+1$ and $|m_{i \sigma(i)}| + |y_{\sigma(i)}| = |d_i|$ otherwise. 

Thus, if $\gamma(i) \neq n+1$, the terms $m_{i\sigma(i)} \stimes y_{\sigma(i)}$ and $m_{i \gamma(i)} \stimes y_{\gamma(i)}$ have maximal modulus among the terms of the sum $m_{i1} \stimes y_1 \splus \cdots \tplus m_{in} \stimes y_n \tminus d_i$. Moreover, \eqref{eq:cramer_recursive_relations} ensures that $m_{i\sigma(i)} \stimes y_{\sigma(i)} \tplus m_{i \gamma(i)} \stimes y_{\gamma(i)} $ is balanced. Similarly, if $\gamma(i) = n+1$, then $m_{i\sigma(i)} \stimes y_{\sigma(i)} \tminus d_i$ is balanced and the terms $m_{i\sigma(i)} \stimes y_{\sigma(i)}$ and $d_i$ have maximal modulus in $m_{i1} \stimes y_1 \splus \cdots \tplus m_{in} \stimes y_n \tminus d_i$. In both cases, we conclude that $\Mrow_i \sdot y \bal d_i$.
\end{proof}

A digraph of longest paths for Example~\ref{exmp:cramer} is shown in Figure~\ref{fig:cramer_graph_example}. From the
relations~\eqref{eq:cramer_recursive_relations}, we obtain the signed solution $y = (\tminus(-1), -1, 0)$.

\subsubsection{Complexity analysis.}\label{subsec-complex-reducedcost}
We now discuss the complexity of the method provided by Proposition~\ref{prop:computing_cramer_solution}. First, a maximizing permutation $\sigma$ can
be found in time $O(n^3)$ by the Hungarian method; see \cite[\S17.2]{Schrijver03:CO_A}. Second, the digraph of longest paths, as well as a directed
tree of longest paths, can be determined in time $O(n^3)$ using the Bellman--Ford algorithm; see \cite[\S8.3]{Schrijver03:CO_A}. Last, the solution
$x$ can be computed in time $O(n)$.

However, we claim that the complexity of the second step can be decreased to $O(n^2)$. The idea is to consider a variant of the Cramer digraph with non-positive weights, and then to apply Dijkstra's algorithm to solve the longest paths problem. We exploit the fact that the Hungarian method is a primal-dual algorithm, which returns, along with a maximizing permutation $\sigma$, an optimal solution $(u,v)$ to the dual assignment problem:
\begin{equation}\label{eq:dual_assignment_problem}
\begin{aligned}
\min_{u,v \in \R^n} \quad &\sum_{i=1}^n u_i + \sum_{j=1}^n v_j \\
	& |m_{ij}| - u_i - v_j \leq 0 \quad \text{for all} \ i,j \in [n] \ .
\end{aligned}
\end{equation}
By complementary slackness, we have:
\begin{equation}
|m_{i\sigma(i)}| = u_i + v_{\sigma(i)} \quad \text{for all} \  i \in [n] \ .
\label{eq:assignment_pb_complementary_slackness}
\end{equation}
Since $\tdet(M) \neq \zero$, the assignment problem has a solution with a finite cost. Therefore, the dual problem~\eqref{eq:dual_assignment_problem} is feasible and bounded. Thus it admits a solution $u,v \in \R^n$. 

We make the diagonal change of variables $y_j = v_j \ttimes z_j$, for all $j\in [n]$, where the $z_j$ are the new variables. We consider the matrix
$M'=(m'_{ij})$ obtained from $M$ by the following diagonal scaling, $m'_{ij}= \mu^{-1} \ttimes u_i^{-1}\ttimes m_{ij}\ttimes v_j^{-1}$, where $\mu$ is
a real number to be fixed soon, together with the vector $d'$ with entries $d'_i = \mu^{-1}\ttimes u_i^{-1}\ttimes d_i$ for all $i\in [n]$.  Then,
 dividing (tropically) every row $i$ of the system $M\sdot y \bal d$ by $\mu$ and by $u_i$, and performing the above change of variables, we arrive at the
equivalent system $M'\sdot z \bal d'$.  By choosing $\mu:=\max(\max_{i}(|d_i|-u_i),0)$, we get that $|d'_i|\leq 0$, and $|m'_{ij}|\leq 0$ for all
$i,j\in [n]$.  The longest path problem to be solved in order to apply the construction of Proposition~\ref{prop-jacobi-combinatorial} to
$M'\sdot z \bal d'$ now involves a digraph with non-positive weights.

It follows that the latter problem can be solved by applying Dijkstra's algorithm to the digraph with modified costs. Moreover, the directed tree provided by Dijkstra's algorithm is also valid in the original problem. 

\subsection{Tropical reduced costs as a solution of a tropical Cramer system}
 In the rest of this section, we 
suppose that Assumption~\ref{ass:finite_coordinates} holds, so we only consider basic points $x^\basis$ with finite entries. We also
 make the following assumption, which is a tropical version of dual non-degeneracy.
\begin{assumption}\label{ass:dual_general_position} 
  The matrix $(A^T \ c^T)$ is tropically sign generic.
\end{assumption}

We can now define the vector of \emph{tropical reduced costs} of a set $\basis\subset[m]$ of cardinality
$n$ such that $\tdet(A_I) \neq \zero$ to be the unique solution $y^\basis \in \strop^m$ of the system of $m$ balances
\begin{equation}
\left \{
\begin{aligned}
\transpose{A} \sdot y &\bal \transpose{c}  \\
y_i &\bal \zero \quad \text{ for all } i \in [m] \setminus \basis \ .
\end{aligned}
\right .
\label{eq:trop_reduced_costs}
\end{equation}

\begin{proposition}\label{prop-pivotingimprove}
Let $x^\basis$ be a tropical basic point of $\tropP(A,b)$ for a suitable $\basis \subset [m]$.
 Then there is a unique solution $y^\basis \in \strop^m$ of the system of balances~\eqref{eq:trop_reduced_costs}.

Let $(\A \ \b)$ be any lift of $(A \ b)$. Pivoting from the basic point $\x^\basis$ of the Puiseux polyhedron $\puiseuxP(\A,\b)$  along the edge $\puiseuxE_{\basis \setminus \{k \} }$ (for $k \in \basis $)  improves the objective function if, and only if, the tropical reduced cost $y^\basis_k$ is tropically negative.
The basic point $\x^\basis$ is an optimum of the Puiseux linear program if and only if the tropical reduced costs $y^\basis$ are tropically non-negative.
\end{proposition}
\begin{proof}
  First, the signed valuation of the Puiseux reduced costs $\y^\basis$ yields a signed solution of~\eqref{eq:trop_reduced_costs}. Let us show that
  this solution is unique.  We apply Theorem \ref{thm:tropical_cramer} with $M = \transpose{A_\basis}$ and $d = \transpose{c}$.  Since $\basis$ yields
  a basic point, the matrix $\A_\basis$ is not singular, thus $\tdet(M) \neq \zero$.  By Assumption \ref{ass:dual_general_position}, the tropical
  determinants of the matrices $M$ and $M_{j \leftarrow d}$ for ${j \in [n]} $ belong to $\strop$. Then by \eqref{eq:cramer}, the vector $y^\basis$ with
  entries $y^\basis_j =  \tdet (M )^{-1} \stimes \tdet( M_{j \leftarrow d}) $ is the unique solution of~\eqref{eq:trop_reduced_costs}.

We have shown that the tropical signs of the tropical reduced costs are exactly the signs of the Puiseux reduced costs, which proves the second part of proposition.
\end{proof}

\begin{example}
  In Example~\ref{example:running_example}, the tropical reduced costs associated with $\basis = \{\mathref{eq:running_example_eq2}, \mathref{eq:running_example_eq3},
  \mathref{eq:running_example_eq4}\}$ are given by $ y^I = (\tminus (-1), -1,0)$, which is the signed solution of~\eqref{eq:cramer_system_example}. It follows that the only edge with negative reduced cost is  $\tropE_{\{\text{\ref{eq:running_example_eq3}}, \text{\ref{eq:running_example_eq4}}\}}$.

\end{example}

\begin{algorithm}[t]
\SetKwComment{Comment}{}{}
\DontPrintSemicolon
\SetArgSty{text}
\footnotesize
\KwIn{A basic point $x^\basis$ of $\tropP(A, b)$, the associated set $\basis$, the objective function $c$}
\KwOut{The tropical reduced costs $y^\basis$ }
$\graph_{x^I} \leftarrow$ tangent graph at $x^\basis$ \Comment*[r]{$O(mn)$}
$\sigma \leftarrow$ maximizing permutation in $\tdet(A_\basis)$ obtained by a traversal of $\graph_{x^I}$ \Comment*[r]{$O(n)$} \label{line:max_permutation}
$u \leftarrow-x^\basis$ \Comment*[r]{$O(n)$}  \label{line:reduced_costs3}
$v \leftarrow A^+_\basis \tdot x^\basis$ \Comment*[r]{$O(mn)$} 
$\mu \leftarrow \max(\max_{j \in [n]}(c_j - u_j),0 )$ \Comment*[r]{$O(n)$}
$M' \leftarrow $ tropically signed matrix with entries $m'_{ij} = \mu^{-1} \stimes u^{-1}_i \stimes a_{ji} \stimes v^{-1}_j$ \Comment*[r]{$O(n^2)$} 
$d' \leftarrow$ tropically signed vector with entries $d_i = \mu^{-1} \stimes u^{-1}_i \stimes c_i$ \Comment*[r]{$O(n)$} 
$\vec{C} \leftarrow$ Cramer digraph of the system $M' \sdot y \bal d'$ for the permutation $\sigma$ \Comment*[r]{$O(n^2)$}
apply Dijkstra's algorithm to $\vec{C}$ from column node $n+1$  \Comment*[r]{$O(n^2 + n \log(n))$} 
$\gamma \leftarrow$ the tree of longest paths returned by Dijkstra's algorithm \; $z \leftarrow$ signed vector obtained by applying
\eqref{eq:cramer_recursive_relations} to the tree $\gamma$ \Comment*[r]{$O(n)$} 
\KwRet{$y^\basis$ the signed vector with entries $y^\basis_j=v_j
  \stimes z_j$}\Comment*[r]{$O(n)$} \label{line:reduced_costsend}
\caption{Computing tropical reduced costs}\label{alg:reduced_costs}
\end{algorithm}

\begin{theorem}\label{thm:reduced_costs}
  Algorithm~\ref{alg:reduced_costs} computes the tropical reduced costs. Its time complexity is bounded by $O(n(m + n))$.
\end{theorem}

\begin{proof}
  The maximizing permutation $\sigma$ is computed from $\graph_{x^I}$ in Line \ref{line:max_permutation} as follows. We first determine a matching between the coordinate nodes $1, \dots, n$ and the set $I$ of hyperplane nodes using the technique described in the proof of Proposition~\ref{prop:tangent_graph_interior_edge}, Case~\eqref{item:tangent_graph_basic_point}. By Lemma~\ref{lemma:matching_in_graph_is_max_permutation}, this matching provides a maximizing permutation in $\tdet(|A_\basis|)$.
It can be obviously computed by a
  traversal of $\graph_{x^I}$ starting from coordinate node $n+1$. Since $\graph_{x^I}$ contains $2n+1$ nodes and $2n$ edges (see the proof of Proposition~\ref{prop:tangent_graph_interior_edge}), this traversal requires $O(n)$ operations.  The complexity of the other operations of this algorithm are
  straightforward and are given in annotations. We conclude that the overall time complexity is $O(m(n + n))$.

  Let $v = A^+_\basis \tdot x^\basis$. For any hyperplane node $j \in \basis$ and any $i \in [n]$, we have $v_j \geq |a_{ji}| + x^\basis_i$, where $A = (a_{ij})$. Moreover, equality holds for every edge $(j,i)$ in the tangent graph. In particular with the permutation $\sigma$, we have $v_{\sigma(i)} =
  |a_{\sigma(i)i}| + x^\basis_{i} $. By Assumptions \ref{assumption_A} and \ref{ass:finite_coordinates}, we have $v \in \R^n$ and $x^{\basis} \in \R^n$. Thus $u = -x^\basis$ and $v$ form an optimal solution to the dual assignment problem
  \eqref{eq:dual_assignment_problem} for the matrix $M = \transpose{A_\basis}$.  It follows from the discussion in Section
  \ref{sec:computing_reduced_costs} that the operations between Line \ref{line:reduced_costs3} and \ref{line:reduced_costsend} compute the tropical
  reduced costs.
\end{proof} 

We conclude this section by applying Algorithm~\ref{alg:main} to the running example~\ref{example:running_example}.
\begin{example}\label{example:run_with_reduced_costs}
  We start from the tropical basic point $(4,4,2)$ associated with $\basis= \{\mathref{eq:running_example_eq2}, \mathref{eq:running_example_eq3},\mathref{eq:running_example_eq1}\}$. For this set, tropical reduced costs are
  $y_{\text{\ref{eq:running_example_eq2}}}= \tminus(-1)$, $y_{\text{\ref{eq:running_example_eq3}}}= -1$ and
  $y_{\text{\ref{eq:running_example_eq1}}}= \tminus 4$. We choose $\ileaving = \text{\ref{eq:running_example_eq1}}$ and pivot along the
  tropical edge $\tropE_{\{\text{\ref{eq:running_example_eq2}},\text{\ref{eq:running_example_eq3}}\}}$.

 We arrive at the basic point $(1,0,0)$, associated with $I= \{\mathref{eq:running_example_eq2},\mathref{eq:running_example_eq3},\mathref{eq:running_example_eq4}\}$.
  The reduced costs are $y_{\text{\ref{eq:running_example_eq2}}}= \tminus(-1)$, $y_{\text{\ref{eq:running_example_eq3}}}= -1$ and $y_{\text{\ref{eq:running_example_eq4}}}= 0$.
  The only tropically negative reduced cost is $y_{\text{\ref{eq:running_example_eq2}}}$, thus we pivot along $\tropE_{\{\text{\ref{eq:running_example_eq3}},\text{\ref{eq:running_example_eq4}}\}}$.

  The new basic point is $(0,0,0)$, corresponding to the set
  $\{\text{\ref{eq:running_example_eq3}},\text{\ref{eq:running_example_eq4}},\text{\ref{eq:running_example_eq5}}\}$. The reduced costs  are tropically positive: $y_{\text{\ref{eq:running_example_eq3}}}= -1$, $y_{\text{\ref{eq:running_example_eq4}}}= 0$  and $y_{\text{\ref{eq:running_example_eq5}}}= -2$. Thus $(0,0,0)$ is optimal.
\end{example}

\section{Proof of the main theorem and generalization to Hahn series}\label{sec:generalization}

\noindent
We now have all the tools needed to prove Theorem~\ref{th-main} under the assumptions of primal non-degeneracy (Assumption
\ref{ass:general_position}), finiteness (Assumption \ref{ass:finite_coordinates}) and dual non-degeneracy
(Assumption~\ref{ass:dual_general_position}).  If a tropical linear program satisfies all three conditions we call it \emph{standard}.
 
\begin{proof}[Proof of Theorem~\ref{th-main}.]
  The time complexity of one iteration of the tropical simplex algorithm follows from the complexity of the tropical pivoting operation (Theorem \ref{th:pivot}) and of the computation of tropical reduced costs (Theorem \ref{thm:reduced_costs}). 

Propositions~\ref{prop:def_basic_points} and~\ref{prop:def_edges} ensure that the tropical pivoting operation traces the image by the valuation map of the pivoting operation over Puiseux series. 
By Proposition \ref{prop-pivotingimprove}, choosing the pivot according to the signs of the tropical reduced costs amounts to choosing a pivot according the signs of the Puiseux reduced costs.  

We claim that under  our assumptions, the edges of the Puiseux polyhedron have a positive length (\ie\ as a set, they are not reduced to a point). By contradiction,  suppose that an edge $\puiseuxE_K$ between the basic points $\x^{K \cup\{k \}}$ and $\x^{K \cup\{k' \}}$ have zero length, where $k \neq k'$ and $k, k' \not\in K$. Then $\x^{K \cup\{k \}}=\x^{K \cup\{k' \}}$. Thus the tropical basic point $x = \val (\x^{K \cup\{k \}})=\val(\x^{K \cup\{k' \}})$ is contained in the $n+1$ tropical s-hyperplanes $\tropH(\pr_i, b_i)$ for $i \in K \cup \{k, k'\}$. Since $x$ has finite entries by Assumption  \ref{ass:finite_coordinates}, the $n+1$ elements of  $K \cup \{k, k'\}$ appears as hyperplane nodes in the tangent graph at $x$. This contradicts Proposition~\ref{prop:tangent_graph_interior_edge} and proves the claim.

The  basic points $\x^{K \cup\{k \}}$ and $\x^{K \cup\{k' \}}$ of an edge $\puiseuxE_K$ are related by $\x^{K \cup\{k' \}} = \x^{K \cup\{k \}} + \pmu \d^k$, where $\pmu> 0$ is the length of $\puiseuxE_K$ and $\d^k$ its direction defined in~\eqref{eq:edge_direction_vector}. When pivoting from $\x^{K \cup\{k \}}$ to $\x^{K \cup\{k' \}}$,  the objective value increases by $\pmu (\cc\d^k)$.  Furthermore, $\y_k = \cc \d^k$ is the reduced cost of the pivot along $\puiseuxE_K$ from the basic point $\x^{K \cup\{k\}}$. As a consequence, as long as a pivot with a negative reduced cost is chosen, each iteration improves the objective function over Puiseux series.
By Assumption~\ref{ass:finite_coordinates}, the Puiseux polyhedron is bounded, thus the value of the Puiseux linear program is finite. Therefore, Algorithm~\ref{alg:main} does terminate.

Finally, the output of Algorithm~\ref{alg:main} is a tropical basic point with tropically non-negative reduced costs. By Proposition~\ref{prop-pivotingimprove}, the corresponding Puiseux basic point is an optimum of the Puiseux linear program. Then by Proposition~\ref{prop:puiseux_solves_tropical_program}, the tropical basic point is an optimum of the tropical linear program.
\end{proof}

We described tropical linear programming in the max-plus version of the tropical semiring. However, the proofs of our results also hold in any
semiring $(\trop_G, \max, +)$ which arises from an abelian totally ordered group $(G, +, \geq)$ \ie\ the semiring is defined on the set $\trop_G = G
\cup \{\zero\}$, the order on $G$ is extended to $\trop_G$ by setting $\zero \leq x$ for any $x \in G$, and the maximum is defined with respect to the
order on $G$. In this setting, the notion of ``tropical general position'' still makes sense. Puiseux series are then replaced by the ordered field
$\R[\![t^G]\!]$ of (formal) Hahn series with real coefficients and with value group $(G, +)$; recall that Hahn series are required to have a well ordered support.
The analysis of Section~\ref{sec:tropical_basic_points} relies only on
the fact that the coefficients of the series are real numbers (Theorem~\ref{thm:val_and_inter_commute_for_generic_matrices}). In
Section~\ref{sec:pivot} the description of a tropical edge as the concatenation of ordinary segments still holds. Finally in
Section~\ref{sec:reduced_costs}, the Tropical Cramer Theorem (Theorem~\ref{thm:tropical_cramer}) is still valid in this generalized setting.

\begin{theorem}
The assertions of Theorem~\ref{th-main} remain valid if the tropical
semiring is replaced by $\trop_G$ and if the field of real Puiseux
series is replaced by the field of real Hahn series $\R[\![t^G]\!]$, 
the execution time being now evaluated in a model in which
every arithmetic operation in the group $G$ takes a time $O(1)$.\qed
\end{theorem}

We end this paper by mentioning some simple extensions of the present results.
Our version of the tropical simplex algorithm can readily be adapted to the maximization of a tropical linear form over a tropical polyhedron, instead
of the minimization. Indeed, the former problem can be handled by taking as a cost vector a vector of negative tropical numbers, and lifting it to a
cost vector of real Puiseux series with negative leading coefficients.

More generally, one may consider a tropical cost vector with both negative and positive coordinates.  The present tropical simplex algorithm can still
be defined in this setting, however, its interpretation in terms of tropical optimization problem turns out to be less satisfactory.  Indeed, the cost
function $c\ttimes x$ may be a balanced tropical number for some feasible vectors $x$, whereas there is no total order on the symmetrized tropical
semiring with a natural interpretation in terms of lift to real Puiseux series. Hence, the tropical minimization problem appears to be somehow ill
defined.  However, for an input in general position, the cost function evaluated at any tropical basic point will always be unbalanced.  Then, the
tropical simplex algorithm, with some straightforward modifications, can be used to return a tropical basic point whose cost is minimal among all
tropical basic points, but which may be incomparable with respect to some non-basic feasible points.

We did not study the overall complexity of the tropical simplex algorithm. But we expect, as in the classical case, an exponential behavior on some
particular examples (such as the Klee-Minty cubes \cite{MR0332165}).

The results of this paper should allow the construction of more general tropical pivoting algorithms. In particular, the criss-cross method
\cite{fukuda1997criss}, which pivots between unfeasible basic points of the arrangement of s-hyperplanes, should also tropicalize. Pivots can be
handled with Algorithm \ref{alg:pivot}. The selection of pivots would involve the tropical signs of the basic point and of the reduced costs.

Finally, we briefly comment on the complexity of deciding if the tropical linear program \ref{eq:trop_linear_prog_pb} given by $A \in \strop^{m \times
  n}$, $b \in \strop^m$ and $c \in \trop^{1 \times n}$ satisfies our standard conditions.  It is always safe to assume that the Assumptions
\ref{assumption_A} and \ref{assumption_B} are satisfied, for it takes at most $O(mn)$ time to simplify the input if this is not the case
\cite[Lemma~1]{GaubertKatz2011minimal}.  Verifying the finiteness condition \ref{ass:finite_coordinates} requires to solve $n$ tropical linear
feasibility problems to check for a non-trivial intersection with the boundary of the tropical projective space, which amounts to solving mean-payoff
games \cite{AGG}.  So it is unclear whether or not this can be done in polynomial time.  Checking for non-degeneracity takes exponential time in the
classical case \cite{Erickson96}; and hence this should also hold for the tropical analog, Assumption~\ref{ass:general_position}.

\bibliographystyle{alpha}

\begin{thebibliography}{BPCR06}

\bibitem[AGG09]{AGG08b}
M.~Akian, S.~Gaubert, and A.~Guterman.
\newblock Linear independence over tropical semirings and beyond.
\newblock In G.L. Litvinov and S.N. Sergeev, editors, {\em Proceedings of the
  International Conference on Tropical and Idempotent Mathematics}, volume 495
  of {\em Contemporary Mathematics}, pages 1--38. AMS, 2009.

\bibitem[AGG12]{AGG}
M.~Akian, S.~Gaubert, and A.~Guterman.
\newblock Tropical polyhedra are equivalent to mean payoff games.
\newblock {\em International of Algebra and Computation}, 22(1):125001, 2012.

\bibitem[AGG13]{AllamigeonGaubertGoubaultDCG2013}
X.~Allamigeon, S.~Gaubert, and E.~Goubault.
\newblock Computing the vertices of tropical polyhedra using directed
  hypergraphs.
\newblock {\em Discrete \& Computational Geometry}, 49(2):247--279, 2013.
\newblock E-print \arxiv{0904.3436v4}.

\bibitem[AGG14]{Akian2013}
M.~Akian, S.~Gaubert, and A.~Guterman.
\newblock Tropical {C}ramer determinants revisited.
\newblock In G.L. Litvinov and S.N. Sergeev, editors, {\em Proceedings of the
  International Conference on Tropical and Idempotent Mathematics}, volume 616
  of {\em Contemporary Mathematics}, pages 1--45. AMS, 2014.

\bibitem[AGK11]{AllamigeonGaubertKatzJCTA2011}
X.~Allamigeon, S.~Gaubert, and R.D. Katz.
\newblock The number of extreme points of tropical polyhedra.
\newblock {\em Journal of Combinatorial Theory, Series A}, 118(1):162 -- 189,
  2011.

\bibitem[AGNS11]{AGNS10}
M.~Akian, S.~Gaubert, V.~Nitica, and I.~Singer.
\newblock Best approximation in max-plus semimodules.
\newblock {\em Linear Algebra and its Applications}, 435(12):3261--3296, 2011.

\bibitem[AK13]{AllamigeonKatzJCTA2013}
X.~Allamigeon and R.D. Katz.
\newblock Minimal external representations of tropical polyhedra.
\newblock {\em Journal of Combinatorial Theory, Series A}, 120(4):907--940,
  2013.

\bibitem[Ale13]{Alessandrini13}
D.~Alessandrini.
\newblock Logarithmic limit sets of real semi-algebraic sets.
\newblock {\em Adv. Geometry}, 13(1):155--190, 2013.

\bibitem[BA08]{BA-08}
P.~{Butkovi\v{c}} and A.~Aminu.
\newblock Introduction to max-linear programming.
\newblock {\em IMA Journal of Management Mathematics}, 20(3):233--249, 2008.

\bibitem[BH04]{BriecHorvath04}
W.~Briec and C.~Horvath.
\newblock $\mathbb{B}$-convexity.
\newblock {\em Optimization}, 53:103--127, 2004.

\bibitem[BPCR06]{basu2006algorithms}
S.~Basu, R.~Pollack, and M.F. Coste-Roy.
\newblock {\em Algorithms in real algebraic geometry}, volume~10.
\newblock Springer, 2006.

\bibitem[BV07]{bjorklund}
H.~Bjorklund and S.~Vorobyov.
\newblock A combinatorial strongly subexponential strategy improvement
  algorithm for mean payoff games.
\newblock {\em Discrete Appl. Math.}, 155:210--229, 2007.

\bibitem[CG79]{CG}
R.A. Cuninghame-Green.
\newblock {\em Minimax algebra}, volume 166 of {\em Lecture Notes in Economics
  and Mathematical Systems}.
\newblock Springer-Verlag, Berlin, 1979.

\bibitem[CGB03]{CGB-03}
R.A. Cuninghame-Green and P.~Butkovi\v{c}.
\newblock The equation $a\otimes x=b\otimes y$ over (max,+).
\newblock {\em Theoretical Computer Science}, 293:3--12, 2003.

\bibitem[CGQ04]{cgq02}
G.~Cohen, S.~Gaubert, and J.P. Quadrat.
\newblock Duality and separation theorem in idempotent semimodules.
\newblock {\em Linear Algebra and Appl.}, 379:395--422, 2004.

\bibitem[Cha09]{chaloupka}
J.~Chaloupka.
\newblock Parallel algorithms for mean-payoff games: an experimental
  evaluation.
\newblock In {\em Algorithms---{ESA} 2009}, volume 5757 of {\em Lecture Notes
  in Comput. Sci.}, pages 599--610. Springer, Berlin, 2009.

\bibitem[DG06]{DG-06}
V.~Dhingra and S.~Gaubert.
\newblock How to solve large scale deterministic games with mean payoff by
  policy iteration.
\newblock In {\em Proceedings of the 1st international conference on
  Performance evaluation methodolgies and tools (VALUETOOLS)}, volume 180,
  Pisa, Italy, 2006.
\newblock article No. 12.

\bibitem[Dre86]{Dress86}
A.~Dress.
\newblock Duality theory for finite and infinite matroids with coefficients.
\newblock {\em Adv. in Math.}, 59(2):97--123, 1986.

\bibitem[DS04]{develin2004}
M.~Develin and B.~Sturmfels.
\newblock Tropical convexity.
\newblock {\em Doc. Math.}, 9:1--27 (electronic), 2004.
\newblock correction: ibid., pp.\ 205--206.

\bibitem[DSS05]{DevelinSantosSturmfels05}
M.~Develin, F.~Santos, and B.~Sturmfels.
\newblock On the rank of a tropical matrix.
\newblock In {\em Combinatorial and computational geometry}, volume~52 of {\em
  Math. Sci. Res. Inst. Publ.}, pages 213--242. Cambridge Univ. Press,
  Cambridge, 2005.

\bibitem[DY07]{DevelinYu07}
M.~Develin and J.~Yu.
\newblock Tropical polytopes and cellular resolutions.
\newblock {\em Experiment. Math.}, 16(3):277--291, 2007.

\bibitem[Eri96]{Erickson96}
Jeff Erickson.
\newblock New lower bounds for convex hull problems in odd dimensions.
\newblock In {\em Proc. 12th Ann. ACM Symp. on Computational Geometry}, 1996.

\bibitem[FAA02]{FilarAltmanAvrachenkov2002}
J.A. Filar, E.~Altman, and K.E. Avrachenkov.
\newblock An asymptotic simplex method for singularly perturbed linear
  programs.
\newblock {\em Oper. Res. Lett.}, 30(5):295--307, 2002.

\bibitem[FT97]{fukuda1997criss}
Komei Fukuda and Tam{\'a}s Terlaky.
\newblock Criss-cross methods: A fresh view on pivot algorithms.
\newblock {\em Mathematical Programming}, 79(1-3):369--395, 1997.

\bibitem[GG98]{cras}
S.~Gaubert and J.~Gunawardena.
\newblock The duality theorem for min-max functions.
\newblock {\em C.R. Acad. Sci.}, 326(1):43--48, 1998.

\bibitem[GK11]{GaubertKatz2011minimal}
S.~Gaubert and R.D. Katz.
\newblock Minimal half-spaces and external representation of tropical
  polyhedra.
\newblock {\em Journal of Algebraic Combinatorics}, 33(3):325--348, 2011.

\bibitem[GKK88]{GKK-88}
V.A. Gurvich, A.V. Karzanov, and L.G. Khachiyan.
\newblock Cyclic games and an algorithm to find minimax cycle means in directed
  graphs.
\newblock {\em USSR Computational Mathematics and Mathematical Physics},
  28(5):85--91, 1988.

\bibitem[GKS12]{GKS11}
S.~Gaubert, R.D. Katz, and S.~Sergeev.
\newblock Tropical linear-fractional programming and parametric mean payoff
  games.
\newblock {\em Journal of symbolic computation}, 47(12):1447--1478, December
  2012.

\bibitem[GM84]{gondran84}
M.~Gondran and M.~Minoux.
\newblock Linear algebra in dioids: a survey of recent results.
\newblock {\em Annals of Discrete Mathematics}, 19:147--164, 1984.

\bibitem[GS07]{gauser}
S.~Gaubert and S.~Sergeev.
\newblock Cyclic projectors and separation theorems in idempotent convex
  geometry.
\newblock {\em Fundamentalnaya i prikladnaya matematika}, 13(4):33--52, 2007.
\newblock Engl.\ translation in Journal of Mathematical Sciences (Springer,
  New-York), Vol. 155, No. 6, pp.815--829, 2008.

\bibitem[Jer73]{Jeroslow1973}
R.G. Jeroslow.
\newblock Asymptotic linear programming.
\newblock {\em Operations Research}, 21(5):1128--1141, 1973.

\bibitem[Jos05]{Joswig05}
M.~Joswig.
\newblock Tropical halfspaces.
\newblock In {\em Combinatorial and computational geometry}, volume~52 of {\em
  Math. Sci. Res. Inst. Publ.}, pages 409--431. Cambridge Univ. Press,
  Cambridge, 2005.

\bibitem[KM72]{MR0332165}
V.~Klee and G.J. Minty.
\newblock How good is the simplex algorithm?
\newblock In {\em Inequalities, {III} ({P}roc. {T}hird {S}ympos., {U}niv.
  {C}alifornia, {L}os {A}ngeles, {C}alif., 1969; dedicated to the memory of
  {T}heodore {S}. {M}otzkin)}, pages 159--175. Academic Press, New York, 1972.

\bibitem[LMS01]{litvinov00}
G.L. Litvinov, V.P. Maslov, and G.B. Shpiz.
\newblock Idempotent functional analysis: an algebraic approach.
\newblock {\em Math. Notes}, 69(5):696--729, 2001.

\bibitem[Mar10]{markwig2007field}
T.~Markwig.
\newblock A field of generalised {P}uiseux series for tropical geometry.
\newblock {\em Rend. Semin. Mat. Univ. Politec. Torino}, 68(1):79--92, 2010.

\bibitem[Meg87]{Megiddo1987}
N.~Megiddo.
\newblock On the complexity of linear programming.
\newblock In {\em Advances in Economic Theory: Fifth World Congress (ed. by T.
  Bewley), Cambridge Univ. Press, Cambridge}, pages 225--268, 1987.

\bibitem[Meg89]{megiddo}
N.~Megiddo.
\newblock On the complexity of linear programming.
\newblock In {\em Advances in economic theory ({C}ambridge, {MA}, 1985)},
  volume~12 of {\em Econom. Soc. Monogr.}, pages 225--268. Cambridge Univ.
  Press, Cambridge, 1989.

\bibitem[Plu90]{akian1990linear}
M.~Plus.
\newblock Linear systems in (max,+) algebra.
\newblock In {\em Proceedings of the 29th IEEE Conference on Decision and
  Control}, pages 151--156. IEEE, 1990.
\newblock M. Plus is a collective name for M. Akian, G. Cohen, S. Gaubert, R.
  Nikoukhah, JP. Quadrat.

\bibitem[RGST05]{richter2005first}
J.~Richter-Gebert, B.~Sturmfels, and T.~Theobald.
\newblock First steps in tropical geometry.
\newblock In {\em Idempotent mathematics and mathematical physics}, volume 377
  of {\em Contemp. Math.}, pages 289--317. Amer. Math. Soc., Providence, RI,
  2005.

\bibitem[Sch03]{Schrijver03:CO_A}
A.~Schrijver.
\newblock {\em Combinatorial optimization. {P}olyhedra and efficiency. {V}ol.
  {A}}, volume~24 of {\em Algorithms and Combinatorics}.
\newblock Springer-Verlag, Berlin, 2003.
\newblock Paths, flows, matchings, Chapters 1--38.

\bibitem[Sch09]{schewe2009}
Sven Schewe.
\newblock From parity and payoff games to linear programming.
\newblock In {\em Mathematical foundations of computer science 2009}, volume
  5734 of {\em Lecture Notes in Comput. Sci.}, pages 675--686. Springer,
  Berlin, 2009.

\bibitem[Sei54]{seidenberg1954new}
A.~Seidenberg.
\newblock A new decision method for elementary algebra.
\newblock {\em Annals of Mathematics}, pages 365--374, 1954.

\bibitem[SS04]{SpeyerSturmfels04}
D.~Speyer and B.~Sturmfels.
\newblock The tropical {G}rassmannian.
\newblock {\em Adv. Geom.}, 4(3):389--411, 2004.

\bibitem[SZ93]{SZ}
B.~Sturmfels and A.~Zelevinsky.
\newblock Maximal minors and their leading terms.
\newblock {\em Adv. Math.}, 98(1):65--112, 1993.

\bibitem[Tar51]{tarski1951decision}
A.~Tarski.
\newblock A decision method for elementary algebra and geometry.
\newblock Research report R-109, Rand Corporation, 1951.

\bibitem[Vir01]{Viro2000}
O.~Viro.
\newblock Dequantization of real algebraic geometry on logarithmic paper.
\newblock In {\em European {C}ongress of {M}athematics, {V}ol. {I}
  ({B}arcelona, 2000)}, volume 201 of {\em Progr. Math.}, pages 135--146.
  Birkh\"auser, Basel, 2001.

\bibitem[Zim76]{Zimmermann.K}
K.~Zimmermann.
\newblock {\em Extrem\'aln\'\i\ Algebra}.
\newblock Ekonomick\'y \`ustav \u CSAV, Praha, 1976.
\newblock (in Czech).

\bibitem[Zim77]{zimmermann77}
K.~Zimmermann.
\newblock A general separation theorem in extremal algebras.
\newblock {\em Ekonom.-Mat. Obzor}, 13(2):179--201, 1977.

\bibitem[Zim81]{Zimmermann.U}
U.~Zimmermann.
\newblock {\em Linear and Combinatorial Optimization in Ordered Algebraic
  Structures}.
\newblock North Holland, 1981.

\bibitem[ZP96]{zwick}
U.~Zwick and M.~Paterson.
\newblock The complexity of mean payoff games on graphs.
\newblock {\em Theoret. Comput. Sci.}, 158(1-2):343--359, 1996.

\end{thebibliography}

\end{document}